\documentclass[secthm,seceqn,amsthm,ussrhead,10pt]{amsart}
\usepackage[utf8]{inputenc}
\usepackage[english]{babel}
\usepackage{amssymb,amsmath,amsthm,amsfonts,xcolor,enumerate,hyperref,comment,longtable,cleveref}

\usepackage{times}
\usepackage{cite}
\usepackage{pdflscape}
\usepackage{ulem}
\usepackage[mathcal]{euscript}
\usepackage{tikz}
\usepackage{hyperref}
\usepackage{cancel}
\usepackage{stmaryrd}

\usetikzlibrary{arrows}

\newcommand{\Co}{{\mathbb C}}

\newtheorem{theorem}{Theorem}
\newtheorem{lemma}[theorem]{Lemma}

\newtheorem{remark}[theorem]{Remark}

\newtheorem{definition}[theorem]{Definition}

\crefrangeformat{equation}{#3(#1)#4--#5(#2)#6}

\crefname{theorem}{Theorem}{Theorems}
\crefname{lemma}{Lemma}{Lemmas}
\crefname{corollary}{Corollary}{Corollaries}
\crefname{proposition}{Proposition}{Propositions}
\crefname{definition}{Definition}{Definitions}
\crefname{example}{Example}{Examples}
\crefname{remark}{Remark}{Remarks}
\crefname{section}{Section}{Sections}
\crefname{equation}{\unskip}{\unskip}
\crefname{enumi}{\unskip}{\unskip}

\newenvironment{Proof}[1][Proof.]{\begin{trivlist}
\item[\hskip \labelsep {\bfseries #1}]}{\flushright
$\Box$\end{trivlist}}

\usepackage{stmaryrd}
\usepackage{xcolor}
\newcommand{\red}{\color{black}}

\newcommand{\aut}[1]{\operatorname{\mathrm{Aut}}{#1}}

\newcommand{\A}{\mathbf{A}}
\newcommand{\T}[2]{\mathbf{T}^{#1}_{#2}}
\newcommand{\D}[2]{\mathbf{D}^{#1}_{#2}}
\newcommand{\Nl}{\mathfrak{N}}

\newcommand{\Dt}[2]{\Delta_{#1#2}}
\newcommand{\Dl}[2]{[\Delta_{#1#2}]}
\newcommand{\nb}[1]{\nabla_{#1}}
\newcommand{\orb}{\mathrm{Orb}}

\newcommand{\der}{\mathfrak{Der}}

\newcommand{\as}[1]{\alpha^\star_{#1}}

\newcommand{\la}{\langle}
\newcommand{\ra}{\rangle}
\newcommand{\La}{\Big\langle}
\newcommand{\Ra}{\Big\rangle}

\newcommand{\lb}{\lambda}
\newcommand{\0}{\theta}
\newcommand{\af}{\alpha}
\newcommand{\bt}{\beta}
\newcommand{\gm}{\gamma}
\begin{document}
\noindent{\Large 
The algebraic and geometric classification of nilpotent terminal algebras}
\footnote{The first part of this work is  supported by  FAPESP 16/16445-0, 18/15712-0, 18/09299-2;
CNPq 451499/2018-2, 404649/2018-1;
the President's "Program Support of Young Russian Scientists" (grant MK-2262.2019.1);
    by CMUP (UID/MAT/00144/2019), which is funded by FCT with national (MCTES) and European structural funds through the programs FEDER, under the partnership agreement PT2020;
    by the Funda\c{c}\~ao para a Ci\^encia e a Tecnologia (Portuguese Foundation for Science and Technology) through the project PTDC/MAT-PUR/31174/2017.
    The  second part of this work is supported by the Russian Science Foundation under grant 19-71-10016.
    The authors thank  Abror Khudoyberdiyev    for   constructive comments.
}$^,$\footnote{Corresponding Author: Ivan Kaygorodov - kaygorodov.ivan@gmail.com  }

\

   {\bf   Ivan Kaygorodov$^{a},$  Mykola Khrypchenko$^{b,c}$ \& Yury Popov$^{d,e}$ \\

    \medskip
}

{\tiny

$^{a}$ CMCC, Universidade Federal do ABC. Santo Andr\'e, Brazil

$^{b}$ Departamento de Matem\'atica, Universidade Federal de Santa Catarina, Florian\'opolis, Brazil

$^{c}$ Centro de Matem\'atica e Aplica\c{c}\~oes, Faculdade de Ci\^{e}ncias e Tecnologia, Universidade Nova de Lisboa, Caparica, Portugal

$^{d}$ IMECC, Universidade Estadual de Campinas, Campinas, Brazil

$^{e}$ Saint Petersburg State University, Russia

\

\smallskip

   E-mail addresses:

\smallskip

    Ivan Kaygorodov (kaygorodov.ivan@gmail.com) 
    
    Mykola Khrypchenko (nskhripchenko@gmail.com)    

    Yury Popov (yuri.ppv@gmail.com)

}

\ 

\

\ 

\noindent {\bf Abstract.}
 We give algebraic and geometric classifications of $4$-dimensional complex nilpotent terminal  algebras. 
 Specifically, we find that, up to isomorphism, 
 there are 
 $41$ one-parameter families of $4$-dimensional nilpotent terminal (non-Leibniz) algebras, 
 $18$ two-parameter families of $4$-dimensional nilpotent terminal (non-Leibniz) algebras, 
  $2$ three-parameter families of $4$-dimensional nilpotent terminal (non-Leibniz) algebras,
 complemented by $21$ additional isomorphism classes (see Theorem~\ref{main-alg}). 
 The corresponding geometric variety has dimension 17 and decomposes into 3 irreducible components determined by the Zariski closures of a one-parameter family of algebras, a two-parameter family of algebras and a three-parameter family of algebras (see Theorem~\ref{main-geo}). 
 In particular, there are no rigid $4$-dimensional complex nilpotent terminal  algebras.

\


\ 


\ 

\section*{Introduction}

Algebraic classification (up to isomorphism) of algebras of small dimension from a certain variety defined by a family of polynomial identities is a classic problem in the theory of non-associative algebras. There are many results related to algebraic classification of small dimensional algebras in varieties of Jordan, Lie, Leibniz, Zinbiel and other algebras.
Another interesting approach of studying algebras of a fixed dimension is to study them from a geometric point of view (that is, to study degenerations and deformations of these algebras). The results in which the complete information about degenerations of a certain variety is obtained are generally referred to as the geometric classification of the algebras of these variety. There are many results related to geometric classification of Jordan, Lie, Leibniz, Zinbiel and other algebras \cite{ack, kv17,ikv17, kv16}.

In 1972, Kantor introduced the notion a conservative algebra as a generalization of Jordan algebras \cite{Kantor72}. Unlike other classes of non-associative algebras, this class is not defined by a set of identities. To introduce the notion of a conservative algebra, we need some notation. 
 Let $\mathbb V$ be a vector space, let $A$ be a linear operator on $\mathbb V,$ and let $B$ and $C$ be bilinear operators on $\mathbb V.$ For all $x,y,z \in \mathbb V,$ put
\begin{gather*}
[B,x](y)=B(x,y),\\
[A,B](x,y)= A(B(x,y))- B(A(x),y)-B(x,A(y)),\\
[B,C](x,y,z)=B(C(x,y),z) + B(x,C(y,z))+ B(y,C(x,z))\\
-C(B(x,y),z)- C(x,B(y,z))-C(y,B(x,z)).
\end{gather*}
 
Consider an algebra as a vector space $\mathbb V$ over a field $\mathbb{C}$, together with an
 element $\mu$ of $\operatorname{Hom}(\mathbb V \otimes \mathbb V,  \mathbb V),$ so that $a \cdot b =\mu(a \otimes b).$ 
For an algebra $(\mathbb V, \mathcal P)$ with a multiplication $\mathcal P$ and $x\in \mathbb V$ we denote by $L_x^{\mathcal P}$ the operator of left multiplication by $x.$ %If the multiplication $P$ is fixed, we write $L_x$ instead of $L_x^P.$
Thus, Kantor defines conservative algebras as follows:

\begin{definition}
An algebra $(\mathbb V, \mathcal P),$ where $\mathbb V$ is the vector space and $\mathcal P$ is the multiplication, is called a conservative algebra if there is a new multiplication $\mathcal P^*: \mathbb V\times \mathbb V\rightarrow \mathbb V$ such that 
\begin{equation}\label{uno}
[L_b^{\mathcal P},[L_a^{\mathcal P}, {\mathcal P}]]=-[L_{{\mathcal P^*}(a,b)}^{\mathcal P},{\mathcal P}] 
\textrm{, for all $a, b \in \mathbb V.$}
\end{equation}
Simple calculations take us to the following identity with an additional multiplication $\mathcal P^*,$ which must  hold for all $a, b, x, y\in \mathbb V$:
\begin{equation}\label{dos}
\begin{split}
b(a(xy)-(ax)y-x(ay)) - a((bx)y) + (a(bx))y+(bx)(ay) -a(x(by))+(ax)(by)+x(a(by))= \\
= - \mathcal P^*(a,b)(xy)+(\mathcal P^*(a,b)x)y + x(\mathcal P^*(a,b)y).
\end{split}
\end{equation}
\end{definition}

The class of conservative algebras is very vast. It includes
 all associative algebras, 
 all quasi-associative algebras,
 all Jordan algebras, 
 all Lie algebras, 
 all (left) Leibniz algebras,
 all (left) Zinbiel algebras, 
 and many other classes of algebras.
On the other side, all conservative algebras are "rigid" (by Cantarini and Kac) algebras  \cite{kacan}.

However, this class is very hard to study and for now even the basic general questions about it remain unanswered, so it is a good idea to study its subclasses which are sufficiently wide but easier to deal with. In \cite{Kantor89} Kantor, studying the generalized TKK functor, introduced the class of terminal algebras which is a subclass of the class of conservative algebras. 

\begin{definition}
An algebra $(\mathbb V, \mathcal P),$ where $\mathbb V$ is a vector space and $\mathcal P$ is a multiplication, is called a terminal algebra if for all $a\in \mathbb V$ we have
\begin{equation}\label{tress}
[[[{\mathcal P},a],{\mathcal P}],{\mathcal P}]=0.
\end{equation}

\end{definition}

Note that we can expand the relation (\ref{tress}), obtaining an identity of degree 4. Therefore, the class of terminal algebras is a variety.
Also, about terminal and conservative algebras see, for example \cite{Kantor90, KPP18, cfk19}.
 
The following remark is obtained by straightforward calculations.

\begin{remark}
A commutative algebra satisfying (\ref{tress}) is a Jordan algebra. 
\end{remark}

Aside from Jordan algebras, the class of terminal algebras includes all Lie algebras, all (left) Leibniz algebras and some other types of algebras.

The following characterization of terminal algebras, proved by Kantor \cite[Theorem 2]{Kantor89}, provides a description of this class as a subclass of the class of conservative algebras.

\begin{remark}
An algebra $(\mathbb{V}, {\bf P})$ is terminal if and only if it is conservative
 and the multiplication in the associated superalgebra ${\bf P}^*$ can be defined by
 \begin{equation}
\label{terminal_associated} 
     {\bf P}^*(x,y)=\frac{2}{3}{\bf P}(x,y)+\frac{1}{3}{\bf P}(y,x).
 \end{equation}
\end{remark}

Our method of classification of nilpotent terminal algebras is based on the calculation of central extensions of smaller nilpotent algebras from the same variety.  The algebraic study of central extensions of Lie and non-Lie algebras has a very big story \cite{omirov,hac16,ss78}.
Skjelbred and Sund used central extensions of Lie algebras for a classification of nilpotent Lie algebras  \cite{ss78}. After that, using the method described by Skjelbred and Sund were described  all non-Lie central extensions of  all $4$-dimensional Malcev algebras \cite{hac16}, all non-associative central extensions of $3$-dimensional Jordan algebras, all anticommutative central extensions of $3$-dimensional anticommutative algebras,
all central extensions of $2$-dimensional algebras \cite{cfk18}.

The algebraic classification of nilpotent algebras will be achieved by the calculation of central extensions of algebras from the same variety which have a smaller dimension.
Central extensions of algebras from various varieties were studied, for example, in \cite{ss78,omirov}.
Skjelbred and Sund \cite{ss78} used central extensions of Lie algebras to classify nilpotent Lie algebras.
Using the same method,  
all non-Lie central extensions of  all $4$-dimensional Malcev algebras \cite{hac16},
all non-associative central extensions of all $3$-dimensional Jordan algebras,
all anticommutative central extensions of $3$-dimensional anticommutative algebras,
all central extensions of $2$-dimensional algebras \cite{cfk18}
and some others were described.
One can also look at the classification of
$3$-dimensional nilpotent algebras \cite{fkkv},
$4$-dimensional nilpotent associative algebras \cite{degr1},
$4$-dimensional nilpotent Novikov algebras,
$4$-dimensional nilpotent bicommutative algebras,
$4$-dimensional nilpotent commutative algebras in \cite{fkkv},
$5$-dimensional nilpotent restricted Lie agebras \cite{usefi1},
$5$-dimensional nilpotent Jordan algebras,
$5$-dimensional nilpotent anticommutative algebras \cite{fkkv},
$6$-dimensional nilpotent Lie algebras \cite{degr3, degr2},
$6$-dimensional nilpotent Malcev algebras \cite{hac18},
$6$-dimensional nilpotent Tortkara algebras,
$6$-dimensional nilpotent binary Lie algebras \cite{ack}.

In the present paper, we classify nilpotent terminal algebras of dimension less than or equal to 4, and obtain the complete description of degenerations of these algebras.

\section{The algebraic classification of nilpotent terminal algebras}
\subsection{Method of classification of nilpotent algebras}
Throughout this paper, we  use the notation and methods described in \cite{hac16,cfk18}
and adapted for the terminal case with some modifications (see also \cite{Jac} for a discussion of extensions of algebras in an arbitrary nonassociative variety). Therefore, all statements in this subsection are given without proofs, which can be found in the papers cited above.

Let ${\bf A}$ be a terminal algebra over $\mathbb C$ and $\mathbb V$ a vector space of dimension $s$ over the same base field. Then the space ${\rm Z_T^2}(\bf A,\mathbb V )$ is defined as the set of all  maps $\theta :{\bf A} \times {\bf A} \to {\mathbb V}$ such that 
\[\theta(b,a(xy) - (ax)y - x(ay)) - \theta(a,(bx)y) + \theta(a(bx),y) + \theta(bx,ay)\]
\[- \theta(a,x(by)) + \theta(ax,(by)) + \theta(x,(ab)y) = -\theta({\bf P}^*(a,b),xy) + \theta({\bf P}^*(a,b)x,y) + \theta(x,{\bf P}^*(a,b)y),\]
where ${\bf P}^*$ is given by (\ref{terminal_associated}). Its elements will be called \textit{cocycles}. For a linear map $f: \bf A \to \mathbb V$ define $\delta f: {\bf A} \times
{\bf A} \to {\mathbb V}$ by $\delta f (x,y ) =f(xy).$ One can check that $\delta f\in {\rm Z_T^2}({\bf A},{\mathbb V} ).$ Therefore, ${\rm B^2}({\bf A},{\mathbb V}) =\left\{ \theta =\delta f\ :f\in \operatorname{Hom}({\bf A},{\mathbb V}) \right\}$ is a subspace of ${\rm Z_T^2}({\bf A},{\mathbb V})$ whose elements are called \textit{coboundaries}. We define the \textit{second cohomology space} ${\rm H_T^2}({\bf A},{\mathbb V})$ as the quotient space ${\rm Z_T^2}({\bf A},{\mathbb V}) \big/{\rm B^2}({\bf A},{\mathbb V}).$
The equivalence class of $\theta \in {\rm Z^2}({\bf A},{\mathbb V})$  in ${\rm H^2}({\bf A},{\mathbb V})$ will be denoted by $[\theta].$

\bigskip

For a bilinear map $\theta :{\bf A} \times {\bf A} \to {\mathbb V}$ define on the linear space ${\bf A}_{\theta }:={\bf A}\oplus {\mathbb V}$ a bilinear product $[-,-]_{{\theta}}$ by 
\[[x+x^{\prime },y+y^{\prime }]_{{\theta}}= xy +\theta (x,y) \mbox{ for all } x,y\in {\bf A},x^{\prime },y^{\prime }\in {\mathbb V}.\]
Then ${\bf A}_{\theta }$ is an algebra called an $s${\it{-dimensional central extension}} of ${\bf A}$ by ${\mathbb V}.$ The following statement can be verified directly:

\begin{lemma}
The algebra ${\bf A_{\theta}}$ is terminal \textit{if and only if} $\theta \in {\rm Z_T^2}({\bf A}, {\mathbb V}).$
\end{lemma}

Recall that the {\it{annihilator}} of ${\bf A}$ is defined as the ideal $\operatorname{Ann}({\bf A} ) =\{ x\in {\bf A}:  x{\bf A}+{\bf A}x=0\}.$ Given $\theta \in {\rm Z_T^2}({\bf A}, {\mathbb V}),$ we call the set $\operatorname{Ann}(\theta)=\left\{ x\in {\bf A}:\theta (x, {\bf A} )+\theta ({\bf A},x ) = 0 \right\}$ the {\it{annihilator}} of $\theta.$ 

\begin{lemma}
\label{ann_of_ext}
$\operatorname{Ann}({\bf A}_{\theta }) = (\operatorname{Ann}(\theta) \cap \operatorname{Ann}({\bf A}))
 \oplus {\mathbb V}.$
\end{lemma}

\bigskip

Therefore, $0 \neq \mathbb{V} \subseteq \operatorname{Ann}({\bf A}_{\theta}).$ The following lemma shows that every algebra with a nonzero annihilator can be obtained in the way described above:

\begin{lemma}
Let ${\bf A}$ be an $n$-dimensional terminal algebra such that $\dim(\operatorname{Ann}({\bf A}))=m\neq0.$ Then there exists, up to isomorphism, a unique $(n-m)$-dimensional terminal algebra ${\bf A}^{\prime }$ and a bilinear map $\theta \in {\rm Z_T^2}({\bf A}, {\mathbb V})$ with $\operatorname{Ann}({\bf A}')\cap \operatorname{Ann}(\theta)=0,$ where $\mathbb V$ is a vector space of dimension $m,$ such that ${\bf A}\cong {\bf A}^{\prime }_{\theta}$ and ${\bf A}/\operatorname{Ann}({\bf A})\cong {\bf A}^{\prime }.$
\end{lemma}

In particular, any finite-dimensional nilpotent algebra is a central extension of another nilpotent algebra of strictly smaller dimension. Thus, to classify all nilpotent terminal algebras of a fixed dimension, we need to classify cocycles of nilpotent terminal algebras ${\bf A}'$ of smaller dimension (with an additional condition $\operatorname{Ann}({\bf A}')\cap \operatorname{Ann}(\theta)=0$) and central extensions that arise from them.

\bigskip

We can reduce the class of extensions that we need to consider.

\begin{definition}
Let ${\bf A}$ be an algebra and $I$ be a subspace of ${\rm Ann}({\bf A})$. If ${\bf A}={\bf A}_0 \oplus I$
then $I$ is called an {\it annihilator component} of ${\bf A}$.
 A central extension of an algebra $\bf A$ without an annihilator component is called a non-split central extension.
\end{definition} 

Clearly, we are only interested in non-split extensions (in the contrary case we can cut off annihilator components lying in $\mathbb{V}$ until we obtain a non-split extension). 

Let us fix a basis $e_{1},\ldots,e_{s}$ of ${\mathbb V},$ and $\theta \in {\rm Z_T^2}({\bf A},{\mathbb V}).$ Then $\theta$ can be uniquely written as $\theta (x,y) =\theta (x,y) =\sum_{i=1}^s\theta_{i}(x,y) e_{i},$ where $\theta_{i}\in {\rm Z_T^2}({\bf A},\mathbb C).$ Moreover, $\operatorname{Ann}(\theta)=\operatorname{Ann}(\theta_{1})\cap \operatorname{Ann}(\theta_{2})\cap\cdots \cap \operatorname{Ann}(\theta_{s}).$ Further, $\theta \in {\rm B^2}({\bf A},{\mathbb V}) $\ if and only if all $\theta_{i}\in {\rm B^2}({\bf A},\mathbb C).$ Using this presentation, one can determine whether the extension corresponding to a cocycle $\theta$ is split:

\begin{lemma} \cite[Lemma 13]{hac16}
\label{split_lindep}
Let $\theta (x,y) =\sum_{i=1}^s\theta_{i}(x,y) e_{i}\in {\rm Z_T^2}({\bf A},{\mathbb V})$ be such that $\operatorname{Ann}(\theta)\cap \operatorname{Ann}({\bf A}) = 0.$ Then ${\bf A}_{\theta }$ has an annihilator component if and only if $[\theta_{1}],[\theta_{2}],\ldots,[\theta_{s}]$ are linearly dependent in ${\rm H_T^2}({\bf A},\mathbb C).$
\end{lemma}  

\bigskip

Some cocycles give rise to isomorphic extensions:

\begin{lemma}
Let $\theta, \vartheta \in {\rm Z_T^2}({\bf A}, {\mathbb V})$ be such that $[\theta] = [\vartheta].$ Then ${\bf A}_\theta \cong {\bf A}_\vartheta.$  
\end{lemma}

%Thus, we need to compute the space ${\rm B^2}({\bf A}, \mathbb V).$ This is easily done with the aid of

%\begin{lemma}
%The map $\delta$ restricts to a linear isomorphism between $({\bf A}^2)^*$ and ${\rm B^2}({\bf A}, \mathbb C),$  where $({\bf A}^2)^* %\subseteq {\bf A}^*$ is the dual space of the square of ${\bf A}^2.$
%\end{lemma}

By above, the isomorphism classes of extensions correspond to certain equivalence classes on ${\rm H_T^2}({\bf A},{\mathbb V}).$ These classes can be given in terms of actions of certain groups on this space. In particular, let $\operatorname{Aut}({\bf A})$ be the automorphism group of ${\bf A},$ let $\phi \in \operatorname{Aut}({\bf A}),$ and let $\psi \in \operatorname{GL}(\mathbb V).$ For $\theta \in {\rm Z_T^2}({\bf A},{\mathbb V})$ define 
\[\phi \theta(x,y) =\theta (\phi(x), \phi(y)), ~ \psi\theta (x,y) = \psi(\theta(x,y)).\] 
Then $\phi\theta, \psi\theta \in {\rm Z_T^2}({\bf A},{\mathbb V}).$ Hence, $\operatorname{Aut}({\bf A})$ and $\operatorname{GL}(\mathbb V)$ act on ${\rm Z_T^2}({\bf A},{\mathbb V}).$ It is easy to verify that ${\rm B^2}({\bf A},{\mathbb V})$ is invariant under both actions. Therefore, we have induced actions on ${\rm H_T^2}({\bf A},{\mathbb V}).$

\begin{lemma}
Let $\theta, \vartheta \in {\rm Z_T^2}({\bf A},{\mathbb V})$ be such that $\operatorname{Ann}({\bf A}_\theta) = \operatorname{Ann}({\bf A}_\vartheta) = \mathbb{V}.$ Then ${\bf A}_\theta \cong {\bf A}_\vartheta$ if and only if there exist a $\phi \in \operatorname{Aut}({\bf A}), \psi \in \operatorname{GL}(\mathbb V)$ such that $[\phi\theta] = [\psi\vartheta].$
\end{lemma}

\bigskip

Now we rewrite the above lemma in a form more suitable for computations.

Let ${\mathbb U}$ be a finite-dimensional vector space over $\mathbb C.$ The {\it{Grassmannian}} $G_{k}({\mathbb U})$ is the set of all $k$-dimensional linear subspaces of ${\mathbb V}.$ Let $G_{s}({\rm H_T^2}({\bf A},\mathbb C) )$ be the Grassmannian of subspaces of dimension $s$ in ${\rm H_T^2}({\bf A},\mathbb C).$ There is a natural action of $\operatorname{Aut} ({\bf A})$ on $G_{s}({\rm H_T^2}({\bf A},\mathbb C)):$ for $\phi \in \operatorname{Aut} ({\bf A}), W=\langle [\theta_{1}],[\theta_{2}],\dots,[\theta_{s}] \rangle \in G_{s}({\rm H_T^2}({\bf A},\mathbb C)),$ define 
\[\phi W=\langle [\phi \theta_{1}], [\phi \theta_{2}],\dots,[\phi \theta_{s}]\rangle.\]
Note that this action is compatible with the action of $\operatorname{Aut} ({\bf A})$ on ${\rm H_T^2}({\bf A},{\mathbb V})$ and the above presentation of a cocycle as a collection of $s$ elements of ${\rm H_T^2}({\bf A},\mathbb C).$ Denote the orbit of $W$ under the action of $\operatorname{Aut} ({\bf A})$ by $\mathrm{Orb}(W).$ It is easy to check that given two bases of a subspace 
\begin{equation*}
W=\langle [\theta_{1}],[\theta_{2}],\dots,[\theta_{s}] \rangle =\langle [\vartheta_{1}],[\vartheta_{2}],\dots,[\vartheta_{s}]
\rangle \in G_{s}({\rm H^2}({\bf A},\mathbb C)),
\end{equation*}
we have $\cap_{i=1}^s\operatorname{Ann}(\theta_{i})\cap \operatorname{Ann}({\bf A}) =\cap_{i=1}^s\operatorname{Ann}(\vartheta_{i})\cap \operatorname{Ann}({\bf A}).$ Therefore, we can introduce the set

\begin{equation*}
T_{s}({\bf A}) =\left\{ W=\left\langle [\theta_{1}],%
[\theta_{2}],\dots,[\theta_{s}] \right\rangle \in
G_{s}({\rm H^2}({\bf A},\mathbb C) ) : \cap_{i=1}^s \operatorname{Ann}(\theta_{i})\cap \operatorname{Ann}({\bf A}) =0\right\},
\end{equation*}
which is stable under the action of $\operatorname{Aut}({\bf A}).$

\medskip

Let us denote by $E({\bf A},{\mathbb V})$ the set of all {\it non-split} central extensions of ${\bf A}$ by ${\mathbb V}.$ By Lemmas \ref{ann_of_ext} and \ref{split_lindep}, we can write
\begin{equation*}
E({\bf A},{\mathbb V}) =\left\{ {\bf A}_{\theta }:\theta (x,y) =\sum_{i=1}^s\theta_{i}(x,y) e_{i}\mbox{ and }\langle [\theta_{1}],[\theta_{2}],\dots,
[\theta_{s}] \rangle \in T_{s}({\bf A})\right\}.
\end{equation*}
Also, we have the next result, which can be proved as \cite[Lemma 17]{hac16}.

\begin{lemma}
 Let $\theta(x,y) =\sum_{i=1}^s \theta_{i}(x,y) e_{i}$ and $\vartheta(x,y) = \sum_{i=1}^s\vartheta_{i}(x,y) e_{i}$ be such that ${\bf A}_{\theta},{\bf A}_{\vartheta }\in E({\bf A},{\mathbb V}).$  Then the algebras ${\bf A}_{\theta }$ and ${\bf A}_{\vartheta }$ are isomorphic if and only if 
\[\mathrm{Orb}\langle [\theta_{1}], [\theta_{2}],\dots,[\theta_{s}] \rangle = \mathrm{Orb}\langle [\vartheta_{1}],[\vartheta
_{2}],\dots,[\vartheta_{s}] \rangle. \]
\end{lemma}

Thus, there exists a one-to-one correspondence between the set of $\operatorname{Aut}({\bf A}) $-orbits on $T_{s}({\bf A})$ and the set of isomorphism classes of $E({\bf A},{\mathbb V}).$ Consequently, we have a procedure that allows us, given a terminal algebra ${\bf A}^{\prime }$ of dimension $n,$ to construct all non-split central extensions of ${\bf A}^{\prime }.$ This procedure is as follows:

\medskip

{\centerline{\it Procedure}}

\begin{enumerate}
\item For a given (nilpotent) terminal algebra $\bf{A}^{\prime }$
of dimension $n-s,$ determine ${T}_{s}(\bf{A}^{\prime })$
and $\operatorname{Aut}(\bf{A}^{\prime }).$

\item Determine the set of $\operatorname{Aut}(\bf{A}^{\prime })$-orbits on $%
{T}_{s}(\bf{A}^{\prime }).$

\item For each orbit, construct the terminal algebra corresponding to one of its
representatives.
\end{enumerate}

The above described method gives all (Leibniz and non-Leibniz) terminal algebras. But we are interested in developing this method in such a way that it only gives non-Leibniz terminal algebras, because the classification of all Leibniz algebras is given in \cite{demir}. Clearly, any central extension of a non-Leibniz terminal  algebra is non-Leibniz. But a Leibniz algebra may have extensions which are not Leibniz algebras. More precisely, let ${\bf  L}$ be a Leibniz algebra and $\theta \in {\rm Z_T^2}({\bf  L}, {\mathbb C}).$ Then ${\bf  L}_{\theta }$ is a Leibniz algebra if and only if 
\begin{equation*}
\theta ( x, yz ) = \theta (xy, z )+ \theta (y, xz  ) 
\end{equation*}
for all $x,y,z\in {\bf  L}.$ Define the subspace ${\rm Z_L^2}({\bf  L},{\mathbb C})$ of ${\rm Z_T^2}({\bf  L},{\mathbb C})$ by
\begin{equation*}
{\rm Z_L^2}({\bf  L},{\mathbb C}) =\left\{\begin{array}{c} \theta \in {\rm Z_T^2}({\bf  L},{\mathbb C}) : \theta ( x, yz ) = \theta (xy, z ) + \theta (y, xz  ) \text{ for all } x, y,z\in {\bf  L}\end{array}\right\}.
\end{equation*}
Observe that ${\rm B^2}({\bf  L},{\mathbb C})\subseteq{\rm Z_L^2}({\bf  L},{\mathbb C}).$
Let ${\rm H_L^2}({\bf  L},{\mathbb C}) =%
{\rm Z_L^2}({\bf  L},{\mathbb C}) \big/{\rm B^2}({\bf  L},{\mathbb C}).$ Then ${\rm H_L^2}({\bf  L},{\mathbb C})$ is a subspace of $%
{\rm H_T^2}({\bf  L},{\mathbb C}).$ Define 
\begin{eqnarray*}
{\bf R}_{s}({\bf  L})  &=&\left\{ {\bf W}\in {T}_{s}({\bf  L}) :{\bf W}\in G_{s}({\rm H_L^2}({\bf  L},{\mathbb C}) ) \right\}, \\
{\bf U}_{s}({\bf  L})  &=&\left\{ {\bf W}\in {T}_{s}({\bf  L}) :{\bf W}\notin G_{s}({\rm H_L^2}({\bf  L},{\mathbb C}) ) \right\}.
\end{eqnarray*}
Then ${T}_{s}({\bf  L}) ={\bf R}_{s}(
{\bf  L})$ $\mathbin{\mathaccent\cdot\cup}$ ${\bf U}_{s}(
{\bf  L}).$ The sets ${\bf R}_{s}({\bf  L}) $
and ${\bf U}_{s}({\bf  L})$ are stable under the action
of $\operatorname{Aut}({\bf  L}).$ Thus, the terminal  algebras
corresponding to the representatives of $\operatorname{Aut}({\bf  L}) $%
-orbits on ${\bf R}_{s}({\bf  L})$ are Leibniz  algebras,
while those corresponding to the representatives of $\operatorname{Aut}({\bf  L}%
) $-orbits on ${\bf U}_{s}({\bf  L})$ are not
Leibniz algebras. Hence, we may construct all non-split non-Leibniz terminal algebras $%
\bf{A}$ of dimension $n$ with $s$-dimensional annihilator 
from a given terminal algebra $\bf{A}%
^{\prime }$ of dimension $n-s$ in the following way:

\begin{enumerate}
\item If $\bf{A}^{\prime }$ is non-Leibniz, then apply the Procedure.

\item Otherwise, do the following:

\begin{enumerate}
\item Determine ${\bf U}_{s}(\bf{A}^{\prime })$ and $%
\operatorname{Aut}(\bf{A}^{\prime }).$

\item Determine the set of $\operatorname{Aut}(\bf{A}^{\prime })$-orbits on ${\bf U%
}_{s}(\bf{A}^{\prime }).$

\item For each orbit, construct the terminal  algebra corresponding to one of its
representatives.
\end{enumerate}
\end{enumerate}

\medskip

\subsection{Notations}
Let us introduce the following notations. Let ${\bf A}$ be a terminal algebra with
a basis $e_{1},e_{2}, \ldots, e_{n}.$ Then by $\Delta_{ij}$\ we will denote the
bilinear form
$\Delta_{ij}:{\bf A}\times {\bf A}\longrightarrow \mathbb C$
with $\Delta_{ij}(e_{l},e_{m}) = \delta_{il}\delta_{jm}.$
The set $\left\{ \Delta_{ij}:1\leq i, j\leq n\right\}$ is a basis for the linear space of 
bilinear forms on ${\bf A},$ so every $\theta \in
{\rm Z^2}({\bf A},\bf \mathbb V )$ can be uniquely written as $%
\theta ={\sum }_{i,j=1}^nc_{ij}\Delta_{ij},$ where $%
c_{ij}\in \mathbb C.$
We also denote by

$$\begin{array}{lll}
\T {i*}{j}& \mbox{the }j\mbox{th }i\mbox{-dimensional nilpotent Leibniz  algebra}, \\
\T i{j}& \mbox{the }j\mbox{th }i\mbox{-dimensional nilpotent non-Leibniz terminal algebra}, \\
{\mathfrak{N}}_i& \mbox{the }i\mbox{-dimensional algebra with zero product}, \\
({\bf A})_{i,j} & j\mbox{th }i\mbox{-dimensional central extension of }\bf A. \\
\end{array}$$

Also, it is easy to see that every central extension of $\mathfrak{N}_i$ is a Leibniz algebra.

\subsection{The algebraic classification of 3-dimensional nilpotent terminal algebras}

Observe that a 2-dimensional nilpotent algebra is either $\Nl_2,$ or is isomorphic to 
\begin{align*}%\label{alg-A_3}
    \begin{array}{llllll}
        \T {2*}{01} &:& e_1 e_1 = e_2. 
    \end{array}
\end{align*}
Both of these algebras are nilpotent of index less than 4, and hence are terminal. All central extensions of 2-dimensional nilpotent algebras were described in~\cite{cfk18}. By a direct verification, we have the following list of all 3-dimensional nilpotent terminal algebras with annihilator of codimension 1 or 2
\begin{align}\label{3-dim-term}
    \begin{array}{lllllllll}
        \T {3*}{01}&:&\T {2*}{01}\oplus\Nl_1 &:& e_1 e_1 = e_2;  \\
        \T {3*}{02}&:&(\mathfrak{N}_2)_{3,1} &:& e_1 e_1 = e_3, & e_2 e_2=e_3; \\
        \T {3*}{03}&:& (\mathfrak{N}_2)_{3,2} &:& e_1 e_2=e_3, & e_2 e_1=-e_3; \\
        \T {3*}{04}(\lb)&:&(\mathfrak{N}_2)_{3,3} &:& e_1 e_1 = \lambda e_3, &  e_2 e_1=e_3,  & e_2 e_2=e_3; \\
        \T {3*}{05}&:&(\T {2*}{01})_{3,1} &:& e_1 e_1 = e_2, & e_1 e_2=e_3; \\
        \T 3{01}(\lb)&:&(\T {2*}{01})_{3,2} &:& e_1 e_1 = e_2, & e_1 e_2=\lb e_3, & e_2 e_1=e_3. 
    \end{array}
\end{align}
Notice that only $\T 3{01}(\lb)$ from list \cref{3-dim-term} is non-Leibniz. 

\subsection{The algebraic classification of $4$-dimensional nilpotent terminal algebras}

Analyzing the list of all the 4-dimensional nilpotent  algebras with annihilator of codimension 2, we have only one terminal non-Leibniz algebra:
$$
\begin{array}{lllllllll}
    %\T 4{01}=(\mathfrak{N}_2)_{4,4} &:&  e_1 e_2=e_3, & e_2 e_1=e_4; \\
    %\T 4{02}=(\mathfrak{N}_2)_{4,5} &:& e_1 e_1 = e_3,  & e_2 e_1=e_4; \\
    %\T 4{03}=(\mathfrak{N}_2)_{4,6} &:& e_1 e_1 = e_3,   & e_2 e_2=e_4; \\
    %\T 4{04}=(\mathfrak{N}_2)_{4,7} &:& e_1 e_1 = e_3, & e_1 e_2=e_3, & e_2 e_1=e_4,  & e_2 e_2=e_3; \\
    %\T 4{05}(\lb\geq 0)=(\mathfrak{N}_2)_{4,8} &:& e_1 e_1 = e_3,  & e_2 e_1=e_4,  & e_2 e_2=e_3+\lambda e_4; \\
    %\T 4{06}(\lb)=(\mathfrak{N}_2)_{4,9} &:& e_1 e_1 = e_3, & e_1 e_2=e_4, & e_2 e_1=\lambda e_4; \\
    \T 4{02}&:&(\T {2*}{01} )_{4,1} &:& e_1 e_1 = e_2, & e_1 e_2=e_4, & e_2 e_1=e_3.
\end{array}
$$
To complete the algebraic classification of 4-dimensional nilpotent terminal algebras, we need to describe all the 1-dimensional terminal non-Leibniz extensions of the algebras from the table \cref{3-dim-term}. This is done in Subsections~\ref{aut-and-H^2}--\ref{ext-T_05^3*} and summarized in the main theorem of the first part of the paper.

\begin{theorem}\label{main-alg}
Let $\mathbf A$ be a 4-dimensional nilpotent non-Leibniz terminal algebra over $\mathbb C$. Then $\mathbf A$ is isomorphic to one of the algebras $\T 4{01}-\T 4{44}$ or $\D 4{01}-\D 4{40}$ found below.
\end{theorem}
\begin{Proof}
 The proof is split into several steps presented in Subsections~\ref{aut-and-H^2}--\ref{ext-T_05^3*} below.
\end{Proof}

\subsubsection{Automorphism and cohomology groups of $3$-dimensional nilpotent terminal algebras}\label{aut-and-H^2}

\[
\tiny
\begin{tabular}{|c|c|c|c|}
\hline 
$\A$ & $\aut\A$ & ${\rm Z_T^2}(\A)$ & ${\rm H_T^2}(\A)$\\
\hline

$\T {3*}{01}$ 
&
$\begin{pmatrix}
x &    0  &  0\\
y &  x^2  &  u\\
z &   0  &  v
\end{pmatrix}$ 
& 
$\Big\langle
\begin{array}{l}
     \Dt 11, \Dt 12, \Dt 13, \Dt 21,\\
     \Dt 23, \Dt 31, \Dt 32, \Dt 33
\end{array}
\Big\rangle$ 
&
$\Big\langle
\begin{array}{l}
     \Dl 12, \Dl 13, \Dl 21, \Dl 23\\
     \Dl 31, \Dl 32, \Dl 33
\end{array}
\Big\rangle$
\\
\hline

$\T {3*}{02}$ 
&
$\begin{pmatrix}
x &    y  &  0\\
(-1)^{n+1} y &  (-1)^n x  &  0\\
z &   u  &  x^2+y^2
\end{pmatrix}$ 
&
$\Big\langle
\begin{array}{l}
     \Dt 11, \Dt 12, \Dt 13, \Dt 21\\
     \Dt 22, \Dt 23, \Dt 31, \Dt 32
\end{array}
\Big\rangle$ 
& 
$\Big\langle
\begin{array}{l}
     \Dl 11, \Dl 12, \Dl 13, \Dl 21\\
     \Dl 23, \Dl 31, \Dl 32
\end{array}
\Big\rangle$ 
\\
\hline

$\T {3*}{03}$ 
&
$\begin{pmatrix}
x &    y  &  0\\
z &    u  &  0\\
v &    w  &  xu-yz
\end{pmatrix}$ 
&
$\Big\langle
\begin{array}{l}
     \Dt 11, \Dt 12, \Dt 13, \Dt 21\\
     \Dt 22, \Dt 23, \Dt 31, \Dt 32
\end{array}
\Big\rangle$ 
&
$\Big\langle
\begin{array}{l}
     \Dl 11, \Dl 12, \Dl 13, \Dl 22\\
     \Dl 23, \Dl 31, \Dl 32
\end{array}
\Big\rangle$
\\
\hline

$\T {3*}{04}$ 
&
    $\begin{pmatrix}
    x &      y    &  0\\
    -\lambda y    &    x-y   &  0\\
    z &      u      &  x^2-xy+\lambda y^2
    \end{pmatrix}$
&
$\Big\langle
\begin{array}{l}
     \Dt 11, \Dt 12, \Dt 13, \Dt 21\\
     \Dt 22, \Dt 23, \Dt 31, \Dt 32
\end{array}
\Big\rangle$ 
&
$\Big\langle
\begin{array}{l}
     \Dl 11, \Dl 12, \Dl 13, \Dl 21\\
     \Dl 23, \Dl 31, \Dl 32
\end{array}
\Big\rangle$
\\
\hline

$\T {3*}{05}$ 
&
$\begin{pmatrix}
x &    0  &  0\\
y &  x^2  &  0\\
z &   xy  &  x^3
\end{pmatrix}$ 
& 
$\Big\langle
\begin{array}{l}
     \Dt 11, \Dt 12, \Dt 13,\\
     \Dt 21, \Dt 22 - 3\Dt 31
\end{array}
\Big\rangle$ 
&
$\Big\langle
\begin{array}{l}
     \Dl 13, \Dl 21,\\
     \Dl 22 - 3\Dl 31
\end{array}
\Big\rangle$
\\
\hline

$\T 3{01}$ 
&
$\begin{pmatrix}
x &               0  &  0\\
y &             x^2  &  0\\
z &   (\lambda+1)xy  &  x^3
\end{pmatrix}$ 
& 
$\begin{cases}
\Big\langle
\begin{array}{l}
     \Dt 11, \Dt 12, -\Dt 13 + 3\Dt 31,\\
     \Dt 21,  {\red\Dt 13 + \Dt 22}, \Dt 23  
\end{array}
\Big\rangle & \lb=0\\
\Big\langle
\begin{array}{l}
     \Dt 11, \Dt 12, (\lb-1)\Dt 13 + 3\Dt 31,\\
     \Dt 21, \red{\Dt 13 + \Dt 22}
\end{array}
\Big\rangle & \lb\ne 0
\end{cases}$
&

$\begin{cases}
\Big\langle
\begin{array}{l}
     \Dl 12, -\Dl 13 + 3\Dl 31,\\
     {\red\Dl 13 + \Dl 22}, \Dl 23  
\end{array}
\Big\rangle & \lb=0\\
\Big\langle
\begin{array}{l}
     \Dl 12, (\lb-1)\Dl 13 + 3\Dl 31,\\
     \red{\Dl 13 + \Dl 22}
\end{array}
\Big\rangle & \lb\ne 0
\end{cases}$
\\
\hline

\end{tabular}
\] 

\vskip0.5cm

Since the algebras $\T {3*}{01}, \T {3*}{02},\T {3*}{03},\T {3*}{04}$ and $\T {3*}{05}$ are Leibniz, it is natural to find the relation between the Leibniz and terminal cohomologies of these algebras in order to exclude those cocycles which give Leibniz algebras. 

\vskip0.5cm

$$\begin{tabular}{|c|c|c|}
	\hline 
	$\A$ & ${\rm H_L^2}(\A)$ & ${\rm H_T^2}(\A)$\\
	\hline
	
	$\T {3*}{01}$ 
	&
	$\langle \Dl 12, \Dl 13, \Dl 31, \Dl 33 \rangle$
	&
	${\rm H_L^2}(\A)\oplus \langle \Dl 21, \Dl 23, \Dl 32 \rangle$
	\\
	\hline
	
	$\T {3*}{02}$ 
	&
	$\langle \Dl 11, \Dl 12, \Dl 21 \rangle$ 
	& 
	${\rm H_L^2}(\A)\oplus\langle \Dl 13, \Dl 23, \Dl 31, \Dl 32\rangle$ 
	\\
	\hline
	
	$\T {3*}{03}$ 
	&
	$\Big\langle
	\begin{array}{l}
	\Dl 11, \Dl 12, \Dl 13 - \Dl 31,\\
	\Dl 22, \Dl 23 - \Dl 32
	\end{array}
	\Big\rangle$ 
	&
	${\rm H_L^2}(\A)\oplus\langle \Dl 31, \Dl 32 \rangle$
	\\
	\hline
	
	$\T {3*}{04}$ 
	&
	$\begin{cases}
	\la \Dl 11, \Dl 12, \Dl 21 \ra,   &  \lb\ne 0\\
	\la\Dl 11, \Dl 12, \Dl 21, \Dl 23\ra, &  \lb=0
	\end{cases}$
	&
	$
	\begin{cases}
	{\rm H_L^2}(\A)\oplus \la \Dl 13, \Dl 23, \Dl 31, \Dl 32 \ra,   &  \lb\ne 0\\
	{\rm H_L^2}(\A)\oplus \la \Dl 13, \Dl 31, \Dl 32 \ra,           &  \lb=0
	\end{cases}$
	\\
	\hline
	
	$\T {3*}{05}$ 
	&
	$\la \Dl 13\ra$
	&
	${\rm H_L^2}(\A)\oplus \la \Dl 21,\Dl 22 - 3\Dl 31 \ra$
	\\
	\hline
	
	\end{tabular}$$

%\newpage

\subsubsection{$1$-dimensional  central extensions of $\T {3*}{01}$}

Let us use the following notations: 
\begin{align*}
    \nb 1 = \Dl 12, \nb 2 = \Dl 13, \nb 3 = \Dl 31, \nb 4 = \Dl 33, \nb 5 = \Dl 21, \nb 6 = \Dl 23, \nb 7 = \Dl 32.    
\end{align*}
Take $\0=\sum_{i=1}^7\af_i\nb i\in {\rm H_T^2}(\T {3*}{01}).$
If 
$$
\phi=
\begin{pmatrix}
x &    0  &  0\\
y &  x^2  &  u\\
z &   0  &  v
\end{pmatrix}\in\aut{\T {3*}{01}},
$$
then
$$
\phi^T\begin{pmatrix}
0& \af_1& \af_2\\
\af_5& 0& \af_6\\
\af_3& \af_7& \af_4
\end{pmatrix} \phi=
\begin{pmatrix}
\af^*& \af_1^*& \af_2^*\\
\af_5^*& \af^{**}& \af_6^*\\
\af_3^*& \af_7^*& \af_4^*
\end{pmatrix},
$$
where
\begin{align*}
    \af_1^* &= x^2(\af_1x + \af_7z),\\
    \af_2^* &= u(\af_1x + \af_7z) + v(\af_2x + \af_6y  + \af_4z),\\
    \af_3^* &= u(\af_5x + \af_6z) + v(\af_3x + \af_7y + \af_4z),\\
    \af_4^* &= v((\af_6 + \af_7)u + \af_4v),\\
    \af_5^* &= x^2(\af_5x + \af_6z),\\
    \af_6^* &= \af_6x^2v,\\
    \af_7^* &= \af_7x^2v.
\end{align*}
Hence, $\phi\langle\0\rangle=\langle\0^*\rangle,$ where $\0^*=\sum\limits_{i=1}^7 \af_i^*  \nb i.$
 
 We are interested in $\0$ with $(\af_5,\af_6,\af_7) \neq (0,0,0).$ Moreover, the condition $\theta \in \mathbf{T}_1 (\T{3*}{01})$ gives us $(\alpha_2, \alpha_3, \alpha_4, \alpha_6, \alpha_7) \neq (0,0,0,0,0)$ and 
 $(\alpha_1 + \gamma\alpha_2, \alpha_5 + \gamma\alpha_3, \alpha_6, \alpha_7, \alpha_6 + \gamma\alpha_4, \alpha_7 + \gamma\alpha_4) \neq (0,0,0,0,0,0)$ for all $\gamma \in \mathbb{C}.$
 
\begin{enumerate}

\item Let $\alpha_6 \neq 0, \alpha_7 \neq 0$ and $\af_1 \af_6 - \af_5 \af_7=0.$
    Taking $v =\frac{1}{x^2 \af_7},$ 
    $z=-\frac{x \af_1}{\af_7}$ and $y=\frac{x (\af_1 \af_4 - \af_3 \af_7)}{\af_7^2},$
     we get the family of representatives 
     $$ 
     \langle  \alpha_2^\star \nabla_2 + \alpha_4^\star \nabla_4 +  \alpha_6^\star \nabla_6+  \nabla_7\rangle,
     $$
    where 

\begin{align*}
    \as 2 &= \frac 1{\af_7^3x}(\af_1\af_4\af_6 - \af_1\af_4\af_7 - \af_3\af_6\af_7 + \af_2\af_7^2),\\
    \as 4 &= \frac 1{\af_7^2x^4}(ux^2\af_7(\af_6 + \af_7) + \af_4),\\
    \af_6^\star &= \frac{\af_6}{\af_7}.
\end{align*}

  \begin{enumerate}
      \item $\af_6\neq -\af_7$
      and $\af_1\af_4\af_6 - \af_1\af_4\af_7 - \af_3\af_6\af_7 + \af_2\af_7^2\neq 0$.
Then choosing $u =-\frac{\af_4}{x^2\af_7(\af_6 + \af_7)}$, where $x$ is such that $\as 2=1$, we have the family of representatives 
     $ \langle \nabla_2 + \alpha  \nabla_6+  \nabla_7\rangle_{\af \neq -1,0}.$ 
      
      \item $\af_6\neq -\af_7$
      and $\af_1\af_4\af_6 - \af_1\af_4\af_7 - \af_3\af_6\af_7 + \af_2\af_7^2 = 0$. Then
      choosing $u = -\frac{\af_4}{x^2 \af_7 (\af_6 + \af_7)},$ we have the family of representatives 
     $ \langle \alpha  \nabla_6+  \nabla_7\rangle_{\af \neq -1,0}.$

      \item $\af_6=- \af_7,$ $\af_4=0$
      and $\af_2 \neq - \af_3$. Then choosing $x = \frac{\af_2 + \af_3}{\af_7},$ we have the representative 
     $ \langle \nabla_2 -  \nabla_6+  \nabla_7\rangle,$ which will be joined with the family $ \langle \nabla_2 + \alpha  \nabla_6+  \nabla_7\rangle_{\af \neq -1,0}.$

      \item $\af_6=- \af_7,$ $\af_4=0$
      and $\af_2 = - \af_3$. Then
       we have the representative 
     $ \langle   -  \nabla_6+  \nabla_7\rangle,$
     which will be joined with the family         $ \langle    \alpha  \nabla_6+  \nabla_7\rangle_{\af \neq -1,0}.$ 

      \item $\af_6=- \af_7$ and $\af_4\neq 0$.
      Then choosing $x = \sqrt[4]{\frac{\af_4}{\af_7^2}},$ we have the family of representatives 
     $ \langle \af \nabla_2 +\nabla_4  -  \nabla_6+  \nabla_7\rangle.$
    It gives only two distinct orbits with representatives 
         $ \langle  \nabla_2 +\nabla_4  -  \nabla_6+  \nabla_7\rangle$ and 
              $ \langle \nabla_4  -  \nabla_6+  \nabla_7\rangle.$

\end{enumerate}

\item Let $\alpha_6 \neq 0, \alpha_7 \neq 0$ and $\af_1 \af_6 - \af_5 \af_7\ne 0.$
    Taking 
    $z=-\frac{\af_5x}{\af_6},$ 
    $v = \frac x{\af_6\af_7}(\af_1\af_6 - \af_5\af_7),$
    $y = \frac x{\af_6\af_7}(\af_4\af_5 - \af_3\af_6),$ 
    and 
    $u = \frac x{\af_6\af_7^2}(\af_3\af_6^2 - \af_2\af_6\af_7 + \af_4\af_5(\af_7 -\af_6)),$
     we get the family of representatives 
     $$ \langle   \nabla_1  +  \alpha_4^\star \nabla_4 +  \alpha_6^\star \nabla_6+  \nabla_7\rangle,$$
    where 

\begin{align*}
    \alpha_4^\star &=  \frac 1{\af_7^3x}(\af_3\af_6^2 + \af_1\af_4\af_7 + \af_3\af_6\af_7 - \af_2\af_6\af_7  - \af_2\af_7^2 - \af_4\af_5\af_6), \\
    \af_6^\star &= \frac{\af_6}{\af_7}.
\end{align*}
It gives two families of representatives of distinct orbits 
  $  \langle   \nabla_1  + \alpha  \nabla_6+  \nabla_7\rangle_{\af\neq 0} $
  and 
  $  \langle   \nabla_1  +  \nabla_4 +  \alpha   \nabla_6+  \nabla_7\rangle_{\af\neq 0}$ 
  depending on whether $\af_3\af_6^2 + \af_1\af_4\af_7 + \af_3\af_6\af_7 - \af_2\af_6\af_7  - \af_2\af_7^2 - \af_4\af_5\af_6=0$ or not.

    \item Let $\alpha_6 = 0, \alpha_7 \neq 0.$
    Taking 
    $y = \frac x{\af_7^2} (\af_1 \af_4 + \af_4 \af_5 - \af_3 \af_7),$ $z = -\frac{x \af_1}{\af_7},$ 
    and $u = -\frac{\af_4v}{\af_7} $
    we get  a family of representatives 
     $$ \langle \alpha_2^\star \nabla_2 + \alpha_5^\star \nabla_5 + \nabla_7\rangle,$$
where 
\begin{align*}
  \alpha_2^\star &= \frac 1{x \af_7^2}( \af_2 \af_7-\af_1 \af_4), \\ 
\alpha_5^\star &=  \frac{\af_5x}{\af_7v}.  
\end{align*}
\begin{enumerate}
        \item $\alpha_5 = 0$ and $\alpha_1\alpha_4 - \alpha_2\alpha_7=0$. Then 
        we have the representative $\langle\nabla_7\rangle,$
        which will be joined with the family    $ \langle \alpha  \nabla_6+  \nabla_7\rangle_{\af \neq 0}.$

        \item $\alpha_5 = 0$ and $\alpha_1\alpha_4 - \alpha_2\alpha_7\ne 0$. Then 
        choosing $x= \frac 1{\af_7^2}(\af_2 \af_7-\af_1 \af_4),$ 
        we have the representative $\langle\nabla_2 + \nabla_7\rangle.$
        which will be joined with the family    $ \langle \nabla_2 + \alpha  \nabla_6+  \nabla_7\rangle_{\af \neq 0}$,

        \item  $\alpha_5 \neq 0$ and $\alpha_1\alpha_4 - \alpha_2\alpha_7\ne 0$. Then  
        choosing $x= \frac 1{\af_7^2}(\af_2 \af_7-\af_1 \af_4)$ and
       $v=\frac{\af_5\af_7x^2}{\af_2\af_7-\af_1\af_4}$
        we get the representative $\langle \nabla_2 + \nabla_5 + \nabla_7 \rangle.$

        \item  $\alpha_5 \neq 0$ and $\alpha_1\alpha_4 - \alpha_2\alpha_7=0$. Then  
        choosing 
        $v = \frac {\af_5x}{\af_7}$
        we get the representative $\langle  \nabla_5 + \nabla_7 \rangle.$
    \end{enumerate}

    \item $\alpha_6 \neq 0$ and $\alpha_7 = 0.$
    Taking 
    $y = \frac x{\af_6^2} (\af_1 \af_4 + \af_4 \af_5 - \af_2 \af_6),$ 
    $z = -\frac{x \af_5}{\af_6},$   and $u = -\frac{\af_4v}{\af_6}$
    we get the family of representatives 
     $$ \langle \alpha_1^\star \nabla_1  + \alpha_3^\star \nabla_3 + \nabla_6\rangle,$$
where 
\begin{align*}
   \alpha_1^\star &=  \frac{\af_1x}{\af_6v}, \\ 
\alpha_3^\star &=  \frac 1{x \af_6^2}(\af_3 \af_6-\af_4 \af_5). 
\end{align*}
  
   \begin{enumerate}
       \item $\af_1\neq 0$ and $\af_3 \af_6- \af_4 \af_5\ne 0$. Then
       we have the representative 
            $  \langle \nabla_1  +  \nabla_3 + \nabla_6\rangle.$

       \item $\af_1 = 0$ and $\af_3 \af_6- \af_4 \af_5\ne 0$. Then
       we have the representative $  \langle  \nabla_3 + \nabla_6\rangle.$

       \item $\af_1 \neq 0$ and $\af_3 \af_6 - \af_4 \af_5=0$. Then
       we have the representative $  \langle  \nabla_1 + \nabla_6\rangle.$

       \item $\af_1 = 0$ and $\af_3 \af_6 - \af_4 \af_5=0$. Then
       we have the representative $  \langle  \nabla_6\rangle.$

   \end{enumerate}
   \item $\alpha_5 \neq 0$, $\af_6=0$ and $\alpha_7=0.$ 
    Taking 
    $u =\frac{-v x \af_3 - v z \af_4}{x \af_5}$ and $x = \frac{1}{\sqrt[3]{\af_5}},$
    we get a family of representatives 
    $$ \langle \alpha_1^\star \nabla_1 + \alpha_2^\star \nabla_2 + \alpha_4^\star \nabla_4 + \nabla_5\rangle,$$
where 
\begin{align*}
   \alpha_1^\star &= \frac{\af_1}{\af_5}, \\ 
   \alpha_2^\star &= \frac v{\af_5^2x^3}((\af_2\af_5 -\af_1\af_3)x + \af_4(\af_5 - \af_1)z), \\
   \alpha_4^\star &= \frac{\af_4v^2}{\af_5x^3}. 
\end{align*}

    \begin{enumerate}
        \item $\alpha_4 \neq 0$ and $\af_1 - \af_5\ne 0$. Then 
        choosing $z= -\frac{(\af_2\af_5 -\af_1\af_3)x}{\af_4 (\af_1 -\af_5)}$ and $v = \sqrt{\frac{\af_5x^3}{\af_4}}$
        we have the family of representatives of distinct orbits
        $ \langle \alpha \nabla_1 + \nabla_4 + \nabla_5\rangle_{\alpha \neq 1}.$ 

        \item $\alpha_4 = 0$, $\af_1 - \af_5\ne 0$ and $\af_2\af_5-\af_1\af_3=0$. Then 
        we have the family of representatives of distinct orbits
        $ \langle \alpha \nabla_1 + \nabla_5\rangle_{\alpha \neq 1}$. The corresponding extensions are split. 
        
        \item $\alpha_4 = 0$, $\af_1 - \af_5\ne 0$ and $\af_2\af_5-\af_1\af_3=0$. Then 
        we have the family of representatives of distinct orbits
        $ \langle \alpha \nabla_1 +\nabla_2 +\nabla_5\rangle_{\alpha \neq 1}.$ 

       \item $\alpha_4 \neq 0$, $\af_1 - \af_5=0$ and $\af_2 - \af_3\ne 0$. Then 
        choosing $x = \frac 1{\af_4\af_5}(\af_2-\af_3)^2$,
        $v=\frac 1{\af_4^2\af_5}(\af_2- \af_3)^3$,
        $u = -\frac {\af_2 - \af_3}{\af_4^2\af_5^2}(\af_4^2\af_5z + (\af_2 - \af_3)^2\af_3)$
        we have the representative 
             $ \langle  \nabla_1 + \nabla_2 +\nabla_4+ \nabla_5\rangle.$

       \item $\alpha_4 \neq 0$, $\af_1 - \af_5=0$ and $\af_2 - \af_3= 0$. Then 
        choosing $v=\sqrt {\frac{\af_5x^3}{\af_4}}$ and $u=-\frac v{\af_5x}(\af_3x + \af_4z)$
        we have the representative 
             $ \langle  \nabla_1 + \nabla_4+ \nabla_5\rangle,$
              which will be joined with the family $\langle \af \nabla_1 +  \nabla_4+ \nabla_5 \rangle_{\af \neq 1}$.

       \item $\alpha_4 = 0$, $\af_1 - \af_5=0$ and $\af_2 - \af_3\ne 0$. Then 
        choosing $u=-\frac{\af_3v}{\af_5}$ and $v=\frac{\af_5x^2}{\af_2 - \af_3}$
        we have the representative 
             $ \langle  \nabla_1 + \nabla_2+ \nabla_5\rangle,$
              which will be joined with the family $\langle \af \nabla_1 +  \nabla_2+ \nabla_5 \rangle_{\af \neq 1}$.

       \item $\alpha_4 = 0$, $\af_1 - \af_5=0$ and $\af_2 - \af_3= 0$. Then choosing $u=-\frac{\af_3v}{\af_5}$
        we have the representative 
             $ \langle  \nabla_1 +  \nabla_5\rangle,$
             which gives a split extension.

    \end{enumerate}
     \end{enumerate}

 \      
      
         %Note that 
  % $\langle   \nabla_1  +  \nabla_4 +     \nabla_7\rangle =     \langle   \nabla_2  + \nabla_7\rangle$
   %and 
   % $\langle   \nabla_1   +     \nabla_7\rangle =     \langle   \nabla_7\rangle.$
 
Summarizing, we have the following representatives of distinct orbits: 

$   
\begin{array}{l}
 \langle   \nabla_1  +  \nabla_4 +  \alpha   \nabla_6+  \nabla_7\rangle_{\af\neq 0},\\
 \langle   \nabla_1  + \alpha  \nabla_6+  \nabla_7\rangle_{\af\neq 0},\\
 \langle  \nabla_2 +\nabla_4  -  \nabla_6+  \nabla_7\rangle,\\
 \langle \nabla_2 + \nabla_5 + \nabla_7 \rangle,
\end{array}
\begin{array}{l}
    \langle \nabla_2 + \alpha  \nabla_6+  \nabla_7\rangle,\\
    \langle \nabla_4  -  \nabla_6+  \nabla_7\rangle,\\
    \langle  \nabla_5 + \nabla_7 \rangle,\\
    \langle \alpha  \nabla_6+  \nabla_7\rangle, 
\end{array}
\begin{array}{l}
  \langle  \nabla_1 + \nabla_6\rangle,\\
  \langle \nabla_1  +  \nabla_3 + \nabla_6\rangle,\\  
  \langle  \nabla_3 + \nabla_6\rangle,\\
  \langle  \nabla_6\rangle,
\end{array}
\begin{array}{l}
     \langle \alpha \nabla_1 +\nabla_2 +\nabla_5\rangle,\\ 
    \langle  \nabla_1 + \nabla_2 +\nabla_4+ \nabla_5\rangle,\\
    \langle \alpha \nabla_1 + \nabla_4 + \nabla_5\rangle.
\end{array}
$

The corresponding algebras are:
\[\begin{array}{lllllllllll}
\T {4}{03}(\alpha)_{\alpha\neq 0}&:& e_1e_1 = e_2,& e_1e_2=e_4,& e_2e_3=\alpha e_4,& e_3e_2=e_4,& e_3e_3=e_4; \\
\T {4}{04}(\alpha)_{\alpha\neq 0}&:& e_1e_1 = e_2,& e_1e_2=e_4,& e_2e_3=\alpha e_4,& e_3e_2=e_4; \\
\T {4}{05} &:& e_1e_1 = e_2,& e_1e_3=e_4,& e_2e_3=- e_4,& e_3e_2=e_4,& e_3e_3=e_4; \\
\T {4}{06}  &:& e_1e_1 = e_2,& e_1e_3=e_4,& e_2e_1= e_4,& e_3e_2=e_4; \\
\T {4}{07}(\alpha) &:& e_1e_1 = e_2,& e_1e_3=e_4,& e_2e_3=\alpha e_4,& e_3e_2=e_4; \\
\T {4}{08}  &:& e_1e_1 = e_2,&  e_2e_3=- e_4,& e_3e_2=e_4,& e_3e_3=e_4; \\
\T {4}{09}  &:& e_1e_1 = e_2,&  e_2e_1= e_4,& e_3e_2=e_4;  \\
\T {4}{10}(\alpha) &:& e_1e_1 = e_2,&   e_2e_3=\alpha e_4,& e_3e_2=e_4;  \\
\T {4}{11} &:& e_1e_1 = e_2,& e_1e_2=e_4,& e_2e_3=  e_4; \\
\T {4}{12} &:& e_1e_1 = e_2,& e_1e_2=e_4,& e_2e_3=  e_4,& e_3e_1=e_4; \\
\T {4}{13} &:& e_1e_1 = e_2,&  e_2e_3=  e_4,& e_3e_1=e_4; \\
\T {4}{14} &:& e_1e_1 = e_2,&   e_2e_3=  e_4;  \\
\T {4}{15}(\alpha) &:& e_1e_1 = e_2,& e_1e_2= \alpha e_4,& e_1e_3=  e_4,& e_2e_1=e_4; \\
\T {4}{16} &:& e_1e_1 = e_2,& e_1e_2=  e_4,& e_1e_3=  e_4,& e_2e_1=e_4,& e_3e_3=e_4; \\
\T {4}{17}(\alpha) &:& e_1e_1 = e_2,& e_1e_2= \alpha e_4,& e_2e_1=e_4,& e_3e_3=  e_4.

\end{array}\]

%\newpage

\subsubsection{$1$-dimensional  central extensions of $\T {3*}{02}$}
Let us use the following notations: 
\begin{align*}
    \nb 1 = \Dl 11, \nb 2 = \Dl 12, \nb 3 = \Dl 21, \nb 4 = \Dl 13, \nb 5 = \Dl 23, \nb 6 = \Dl 31, \nb 7 = \Dl 32.    
\end{align*}
Take $\0=\sum_{i=1}^7\af_i\nb i\in {\rm H_T^2}(\T {3*}{02}).$
If 
$$
\phi=
\begin{pmatrix}
x          &       y& 0 \\
(-1)^{n+1}y& (-1)^nx& 0\\
z          & u      & x^2+y^2
\end{pmatrix}\in\aut{\T {3*}{02}},
$$
then
$$
\phi^T\begin{pmatrix}
\af_1& \af_2& \af_4\\
\af_3& 0& \af_5\\
\af_6& \af_7& 0
\end{pmatrix} \phi=
\begin{pmatrix}
\af_1^*+\af^*& \af_2^*& \af_4^*\\
\af_3^*& \af^*& \af_5^*\\
\af_6^*& \af_7^*& \af^{**}
\end{pmatrix},
$$
where
\begin{align*}
    \af_1^* &= \af_1(x^2 - y^2) + 2(-1)^{n+1}(\af_2 + \af_3)xy+(\af_4x+(-1)^{n+1}\af_5y+\af_6x+ (-1)^{n+1}\af_7y)z\\
	&\quad - (\af_4y +(-1)^n\af_5x+\af_6y+(-1)^n\af_7x)u,\\
	\af_2^* &= (-1)^n\af_2x^2 + \af_1xy + (-1)^{n+1}\af_3y^2 + (\af_4x + (-1)^{n+1}\af_5y)u + (\af_6y  + (-1)^n\af_7x)z,\\
	\af_3^* &= (-1)^n\af_3x^2 + \af_1xy + (-1)^{n+1}\af_2y^2 + (\af_4y + (-1)^n\af_5x)z + (\af_6x + (-1)^{n+1}\af_7y)u,\\
	\af_4^* &= (\af_4x+(-1)^{n+1}\af_5y)(x^2+y^2),\\
	\af_5^* &= (\af_4y+(-1)^n\af_5x)(x^2 + y^2),\\
	\af_6^* &= (\af_6x+(-1)^{n+1}\af_7y)(x^2+y^2),\\
	\af_7^* &= (\af_6y + (-1)^n\af_7x)(x^2+y^2).
\end{align*}
Hence, $\phi\langle\0\rangle=\langle\0^*\rangle,$ where $\0^*=\sum\limits_{i=1}^7 \af_i^*  \nb i.$
 
 We are interested in $\0$ with $(\af_4,\af_5,\af_6,\af_7) \neq (0,0,0,0).$  The condition $\theta \in \mathbf{T}_1 (\T{3*}{02})$ does not give us any new restrictions on parameters.

\begin{enumerate}
    \item $\alpha_6=0$ and $\alpha_7=0.$ Then $\af_4\ne 0$ or $\af_5\ne 0$. If $\af_4\ne 0$ and $\af_5=0$, then choosing $x=y=1$, we have $\af^*_4\ne 0$ and $\af^*_5\ne 0$. The same holds for $\af_4=0$ and $\af_5\ne 0$. Thus, we shall assume from the very beginning that $\alpha_4\ne 0$ and $\alpha_5 \neq 0$. 
    \begin{enumerate}
        \item $\alpha^2_4+\alpha^2_5\ne 0$. Then choosing $x=(-1)^{n+1}\frac{\alpha_4y}{\alpha_5},
                    z = \frac{(-1)^ny}{\af_5(\af_4^2 + \af_5^2)}(\af_1\af_4^2 + 2(\af_2 + \af_3)\af_4\af_5 - \af_1\af_5^2)$ and $u = \frac y{\af_5(\af_4^2 + \af_5^2)}(\af_2\af_4^2 -\af_3\af_5^2 - \af_1\af_4\af_5)$ we obtain the representative $\la\as 3\nb 3+\as 4\nb 4\ra$, where
                    \begin{align*}
                        \as 3 &= \frac{(-1)^ny^2}{\af_5^2}(\af_3\af_4^2 - \af_1\af_4\af_5 - \af_2\af_5^2),\\
                        \as 4 &= \frac{(-1)^{n+1}y^3}{\af_5^3}(\af_4^2 + \af_5^2)^2.
                    \end{align*}
                    Thus, we have two representatives $\langle \nabla_4 \rangle $ and $\langle \nabla_3 + \nabla_4 \rangle $ depending on whether $\af_3\af_4^2 - \af_1\af_4\af_5 - \af_2\af_5^2=0$ or not.
        \item  $\alpha^2_4+\alpha^2_5=0$. If $\af_4=-i\af_5$, then choosing $x=0$ and $y=n=1$, we obtain $\af^*_4=i\af^*_5$, so we shall assume from the very beginning that $\alpha_4= i \alpha_5.$ 
        Taking $n=0$,
        $u=\frac{i(\af_2x^2 + \af_1xy - \af_3y^2)}{\af_5(x + iy)}$ and 
        $z=-\frac{\af_3x^2 + \af_1xy - \af_2y^2}{\af_5(x + iy)}$
        we have the representative $\langle \alpha_1^{\star} \nabla_1 + i  \nabla_4 + \nabla_5 \rangle$, where
        \begin{align*}
            \as 1 &= \frac{x - iy}{\af_5(x + iy)^2}(\af_1 - i\af_2 - i\af_3).
        \end{align*}
Thus, we have two representatives $\langle i  \nabla_4 +  \nabla_5 \rangle$ and  $\langle \nabla_1 + i  \nabla_4 +  \nabla_5 \rangle$ depending on whether $\af_1 - i\af_2 - i\af_3=0$ or not.

        \end{enumerate}
        
\item $\alpha_6\neq 0$ or $\alpha_7\neq0$, and $\alpha_6^2+\alpha_7^2\ne 0$. 
 Then we may make $\alpha_6\ne 0$ and $\alpha_7\neq 0$. After that, if we choose $x=(-1)^{n+1}\frac{\alpha_6y}{\alpha_7}\ne 0,$ then $x^2+y^2=\frac {y^2}{\af^2_7}(\af_6^2+\af_7^2)\ne 0$, $\af^*_6=\frac{(-1)^{n+1}y^3}{\af_7^3}(\af_6^2 + \af_7^2)^2\ne 0$ and $\alpha_7^* = 0$. Thus, we shall assume that $\af_6\ne 0$ and $\af_7=0$ from the very beginning. Taking $y = 0,$ we get $\af^*_7=0$ and
 \begin{align*}
    \alpha_1^* &= x (x \alpha_1 + (-1)^{n + 1} u \alpha_5 + z (\alpha_4 + \alpha_6)),\\
	\af_2^* &= x ((-1)^n x \alpha_2 + u \alpha_4),\\
	\af_3^* &= x ((-1)^n x \alpha_3 + (-1)^n z \alpha_5 + u \alpha_6),\\
	\af_4^* &= x^3 \alpha_4,\\
	\af_5^* &= (-1)^n x^3 \alpha_5,\\
	\af_6^* &= x^3 \alpha_6.
\end{align*}
Consider $\alpha_2^* = 0, \alpha_3^* = 0$ as a system of linear equations on $z$ and $u.$ Its determinant is $(-1)^{n+1}x^2\alpha_4\alpha_5,$ so we have the following cases (each of which defines an $\operatorname{Aut}{\T {3*}{02}}$-invariant subset):
\begin{enumerate}
    \item $\alpha_4 \neq 0$ and $\alpha_5 \neq 0$. Taking $u = \frac{(-1)^{n+1} x \alpha_2}{\alpha_4}, z = \frac{x (\alpha_2 \alpha_6-\alpha_3 \alpha_4)}{\alpha_4 \alpha_5},$ we get $\alpha_2^* = 0, \alpha_3^* = 0$ and 
    \begin{align*}
    \alpha_1^* &= \frac{x^2}{\af_4\af_5}(\af_2\af_5^2 - \af_3\af_4^2 + \af_1\af_4\af_5 + \af_2\af_4\af_6 - \af_3\af_4\af_6 +\af_2\af_6^2),\\
	\af_4^* &= x^3 \alpha_4,\\
	\af_5^* &= (-1)^n x^3 \alpha_5,\\
	\af_6^* &= x^3 \alpha_6.
\end{align*}
Thus, we obtain two families of representatives of distinct orbits $\langle \nabla_1 +\alpha  \nabla_4+\beta  \nabla_5+\nabla_6 \rangle_{\alpha\beta\neq 0, \beta \in \mathbb C_{\geq 0}}$ and $\langle \alpha  \nabla_4+\beta \nabla_5+\nabla_6 \rangle_{\alpha\beta\neq 0, \beta \in \mathbb C_{\geq 0}}. $
    \item $\alpha_4 \neq 0$ and $\alpha_5 = 0$. Taking $u = \frac{(-1)^{n+1} x \alpha_2}{\alpha_4}$, we get $\alpha_2^* = 0$ and    
    \begin{align*}
    \alpha_1^* &= x (x \alpha_1 + z (\alpha_4 + \alpha_6)),\\
    \af_3^* &= \frac{(-1)^nx^2}{\af_4}(\af_3\af_4 - \af_2\af_6),\\
	\af_4^* &= x^3 \alpha_4,\\
	\af_6^* &= x^3 \alpha_6.
\end{align*}
    \begin{enumerate}
        \item $\alpha_4 + \alpha_6 \neq 0$. Taking $z = -\frac{x\alpha_1}{\alpha_4 + \alpha_6},$ we get two series of representatives of distinct orbits $\langle \nabla_3 +\alpha \nabla_4 +\nabla_6\rangle_{\alpha\neq -1,0}$ and $\langle \alpha \nabla_4 +\nabla_6\rangle_{\alpha\neq -1,0}.$
        \item $\alpha_4 + \alpha_6 = 0.$ In this case we get the series of representatives of distinct orbits $\langle \nabla_1 +\alpha  \nabla_3 - \nabla_4+\nabla_6 \rangle_{\alpha \in \mathbb C_{\geq 0}}$ and two distinct representatives $\langle \nabla_3 - \nabla_4+\nabla_6 \rangle$ and $\langle - \nabla_4+\nabla_6 \rangle$, which will be joined with the families $\langle \nabla_3 +\alpha \nabla_4 +\nabla_6\rangle_{\alpha\neq -1,0}$ and $\langle \alpha \nabla_4 +\nabla_6\rangle_{\alpha\neq -1,0}$ found above.
    \end{enumerate}
    \item $\alpha_4 = 0$ and $\alpha_5 \neq 0.$ Consider $\alpha_1^* = 0, \alpha_3^* = 0$ as a system of linear equations on $z,u.$ Its determinant is $x^2(\alpha_5^2 + \alpha_6^2).$ 
    \begin{enumerate}
        \item $\alpha_5^2 + \alpha_6^2 \neq 0.$ Taking $u = \frac{(-1)^n x (\alpha_1 \alpha_5 - \alpha_3 \alpha_6)}{ \alpha_5^2 +  \alpha_6^2}, z = -\frac{x ( \alpha_3 \alpha_5 + \alpha_1 \alpha_6)}{ \alpha_5^2 + \alpha_6^2},$ we get $\alpha^*_1 = \alpha^*_3 = 0$ and
        \begin{align*}
    \alpha_2^* &= (-1)^n x^2 \alpha_2,\\
	\alpha_5^* &= (-1)^n x^3 \alpha_5,\\
	\alpha_6^* &= x^3 \alpha_6.
\end{align*}
Therefore, we have two series of representatives of distinct orbits $\langle \nabla_2 + \alpha \nabla_5+\nabla_6 \rangle_{\alpha \in \mathbb C_{\geq 0}, \alpha \neq 0, i}$ and $\langle \alpha \nabla_5+\nabla_6 \rangle_{\alpha \in \mathbb C_{\geq 0}, \alpha \neq 0, i}.$
    \item $\alpha_5^2 + \alpha_6^2 = 0.$ We may assume that $\alpha_5 = i \alpha_6$ (if $\af_5=-i\af_6$, then choose $n=-1$). Taking $n=0$, $u = -\frac{x\alpha_3}{\alpha_6}$and $z = 0,$ we get $\alpha_3 = 0$ and 
    \begin{align*}
    \alpha_1^* &= x^2 (\alpha_1+i\af_3),\\
    \alpha_2^* &= x^2 \alpha_2,\\
	\alpha_5^* &= ix^3 \alpha_6,\\
	\alpha_6^* &= x^3 \alpha_6.
\end{align*}
Therefore, we get the family of representatives of distinct orbits $\langle \nabla_1 + \alpha \nabla_2 + i \nabla_5 + \nabla_6 \rangle$ and two representatives $\langle \nabla_2 + i \nabla_5 + \nabla_6 \rangle, \langle i \nabla_5 + \nabla_6 \rangle$, which will be joined with the families $\langle \nabla_2 + \alpha \nabla_5+\nabla_6 \rangle_{\alpha \in \mathbb C_{\geq 0}, \alpha \neq 0, i}$ and $\langle \alpha \nabla_5+\nabla_6 \rangle_{\alpha \in \mathbb C_{\geq 0}, \alpha \neq 0, i}$ found above. 
    \end{enumerate}
    \item $\alpha_4 = 0$ and $\alpha_5 = 0.$ Taking $z = -\frac{x\alpha_1}{\alpha_6}$ and $u = \frac{(-1)^{n+1}x\alpha_3}{\alpha_6},$ we get $\alpha_1^* = \alpha_3^* = 0$ and 
     \begin{align*}
    \alpha_2^* &= (-1)^nx^2 \alpha_2,\\
	\alpha_6^* &= x^3 \alpha_6.
\end{align*}
Therefore, we get two representatives $\langle \nabla_2 + \nabla_6 \rangle$ and $\langle \nabla_6 \rangle$, which will be joined with the families $\langle \nabla_2 + \alpha \nabla_5+\nabla_6 \rangle_{\alpha \in \mathbb C_{\geq 0}, \alpha \neq 0, i}$ and $\langle \alpha \nabla_5+\nabla_6 \rangle_{\alpha \in \mathbb C_{\geq 0}, \alpha \neq 0, i}$ found above. Note that representatives $\langle \nabla_2 + \nabla_6 \rangle$ and $\langle \nabla_6 \rangle$ define the same orbit.
\end{enumerate}

\item $\alpha_6\neq 0$ or $\alpha_7\neq0$, and $\alpha_6^2+\alpha_7^2=0$. Then we may assume that $\alpha_6= i \alpha_7\ne 0.$

\begin{enumerate}
    \item $\alpha_4 \neq \pm i \alpha_5.$ We may assume that $\alpha_4 \neq 0$ (otherwise $\af_5\ne 0$, so we may take $x=y=1$ to make $\af^*_4\ne 0$). Then choosing $n=0$ and $y=-\frac{x\alpha_5}{\alpha_4}x,$ we get $\af^*_4=\frac{x^3}{\af_4^3}(\af_4^2 + \af_5^2)^2\ne 0$ and $\alpha_5^* = 0$, so we shall assume that $\af_4\ne 0$ and $\af_5=0$ from the very beginning. Now, choosing $n=y=0$, $u = \frac{i x \alpha_3}{\alpha_7}, z = -\frac{x (i \alpha_3 \alpha_4 + \alpha_2 \alpha_7)}{\alpha_7^2},$ we get $\alpha_2^* = \alpha_3^* = \af^*_5=0$ and 
    \begin{align*}
    \alpha_1^* &= \frac{x^2}{\af_7^2}(\af_3\af_4\af_7 + \af_1\af_7^2 - \af_2\af_4\af_7 -i\af_3\af_4^2 - i\af_2\af_7^2 -i\af_3\af_7^2),\\
    \alpha_4^* &= x^3 \alpha_4,\\
	\alpha_6^* &= ix^3 \alpha_7,\\
	\alpha_7^* &= x^3 \alpha_7.
\end{align*}
Thus, we get two series of representatives of distinct orbits $\langle \nabla_1 + \alpha \nabla_4 + i \nabla_6 + \nabla_7 \rangle_{\alpha \neq 0}$ and $\langle \alpha \nabla_4 + i \nabla_6 + \nabla_7 \rangle_{\alpha \neq 0}.$
    
    %Choosing $x=\frac{\alpha_4}{\alpha_5}y$ and $u=-\frac{-\alpha_3x^2+\alpha_1xy+\alpha_2y^2}{\alpha_6x+\alpha_7y},z=-\frac{-\alpha_2x^2+\alpha_1xy+\alpha_3y^3+(\alpha_4x+\alpha_5y)u}{\alpha_6y-\alpha_7x}$ we have the family of representatives $\langle \alpha_1^{\star} \nabla_1+\alpha_4^{\star}\nabla_4 +i\nabla_6+\nabla_7 \rangle.$ Now, by one more action of a suitable automorphism we have 
    
   % \begin{enumerate}
  %      \item if $\alpha_1^{\star}\neq 0,$ then we have the family of representatives $\langle \nabla_1+\alpha \nabla_4 +i\nabla_6+\nabla_7 \rangle_{\alpha\neq 0}.$
  %      \item if $\alpha_1^{\star}= 0,$ then we have the family of representatives $\langle \alpha \nabla_4 +i\nabla_6+\nabla_7 \rangle_{\alpha\neq 0}.$

  %  \end{enumerate}
    
    \item $\alpha_4= i \alpha_5$. Then choosing $n=0$ we have 
    \begin{align*}
    \alpha_1^* &= x^2 \alpha_1 - x (2 y (\alpha_2 + \alpha_3) + (u - i z) (\alpha_5 + \alpha_7)) -  y (y \alpha_1 + (i u + z) (\alpha_5 + \alpha_7)),\\
    \alpha_2^* &= x^2 \alpha_2 - y (y \alpha_3 + u \alpha_5 -i z \alpha_7) +  x (y \alpha_1 + i u \alpha_5 + z \alpha_7),\\
    \alpha_3^* &= (x + i y) z \alpha_5 + x (y \alpha_1 + x \alpha_3 + i u \alpha_7) -  y (y \alpha_2 + u \alpha_7),\\
    \alpha_4^* &= i (x + i y) (x^2 + y^2) \alpha_5,\\
    \alpha_5^* &= (x + i y) (x^2 + y^2) \alpha_5,\\
	\alpha_6^* &= i (x + i y) (x^2 + y^2) \alpha_7,\\
	\alpha_7^* &= (x + i y) (x^2 + y^2) \alpha_7.
\end{align*}
    Consider $\alpha_1^* = 0, \alpha_3^* = 0$ as a system of linear equations in $u,z.$ Its determinant is $i(x+iy)^2(\alpha_7^2 - \alpha_5^2).$
    \begin{enumerate}
        \item $\alpha_5^2 - \alpha_7^2 \neq 0$. Then by choosing the appropriate values of $u$ and $z$ we get $\alpha_1^* = 0, \alpha_3^* = 0$, $\alpha_2^* = i(i\af_1 + \af_2 + \af_3)(x - iy)^2$ and $\af^*_4,\af^*_5,\af^*_6,\af^*_7$ as above. Therefore, we have two series of representatives of distinct orbits $\langle \nabla_2 + i \alpha \nabla_4 + \alpha \nabla_5 + i \nabla_6 + \nabla_7 \rangle_{\alpha \neq \pm 1}$ and $\langle  i \alpha \nabla_4 + \alpha \nabla_5 + i \nabla_6 + \nabla_7 \rangle_{\alpha \neq \pm 1}$.
    \item $\alpha_5 = \alpha_7$. Then taking $n=0$, $z = \frac{-x y \alpha_1 + y^2 \alpha_2 - x^2 \alpha_3 - i u x \alpha_7 +  u y \alpha_7}{(x + i y)\alpha_7}$ we get $\alpha_3^* = 0$ and
    \begin{align*}
    \alpha_1^* &= (x - i y) (x (\alpha_1- 2i\af_3) - y (i \alpha_1 + 2 \alpha_2)),\\
    \alpha_2^* &= (x^2 + y^2) (\alpha_2-\af_3),\\
    \alpha_4^* &= i (x + i y) (x^2 + y^2) \alpha_7,\\
    \alpha_5^* &= (x + i y) (x^2 + y^2) \alpha_7,\\
	\alpha_6^* &= i (x + i y) (x^2 + y^2) \alpha_7,\\
	\alpha_7^* &= (x + i y) (x^2 + y^2) \alpha_7.
\end{align*}
\begin{enumerate}
    \item $i\af_1 + \af_2 + \af_3\ne 0$, $\af_2 - \af_3\ne 0$ and $i\af_1 + 2\af_2\ne 0$. Then taking $y = \frac{(\af_1 - 2i\af_3)x}{i\af_1 + 2\af_2}$ we get the representative $\la\as 2\nb 2+i \nabla_4 + \nabla_5 + i \nabla_6 + \nabla_7\ra$, where
        \begin{align*}
            \as 2 &= \frac{(i\af_1 + 2\af_2)(\af_2 - \af_3)}{2\af_7 x(i\af_1 + \af_2 + \af_3)}. 
        \end{align*}
So, choosing the appropriate value of $x$, we get the representative $\langle \nabla_2 + i \nabla_4 + \nabla_5 + i \nabla_6 + \nabla_7\rangle.$
    \item $i\af_1 + \af_2 + \af_3 = 0$. Then we have 
         \begin{align*}
    \alpha_1^* &= i(x^2+y^2)(\alpha_2-\af_3),\\
    \alpha_2^* &= (x^2+y^2)(\alpha_2-\af_3),
\end{align*}
Therefore, we get two representatives $\langle i\nabla_1 + \nabla_2 + i \nabla_4 + \nabla_5 + i \nabla_6 + \nabla_7\rangle$ and $\langle i \nabla_4 + \nabla_5 + i \nabla_6 + \nabla_7\rangle.$
    \item $\alpha_2-\af_3 = 0$. Then $\alpha_1^* =  (x-iy)^2(\af_1 - 2i\af_3)$, and we have only one new representative $\langle \nabla_1 + i \nabla_4 + \nabla_5 + i \nabla_6 + \nabla_7\rangle.$
    \item $i\af_1 + 2\af_2=0$. Then
    $
        \alpha_1^* = x(x - i y) (\alpha_1- 2i\af_3),
    $
    so choosing $x=0$ we have two representatives $\langle \nabla_2 + i \nabla_4 + \nabla_5 + i \nabla_6 + \nabla_7\rangle$ and $\langle i \nabla_4 + \nabla_5 + i \nabla_6 + \nabla_7\rangle$ found above.
\end{enumerate}
\item $\alpha_5 = -\alpha_7$. Then taking $n=0$ and $z = \frac{x y \alpha_1 - y^2 \alpha_2 + x^2 \alpha_3 + i u x \alpha_7 -  u y \alpha_7}{(x + i y) \alpha_7}$ we get $\alpha_3^*=0$ and
     \begin{align*}
    \alpha_1^* &= (x^2 - y^2)\alpha_1  - 2 x y (\alpha_2+\af_3),\\
    \alpha_2^* &= (x^2 - y^2)(\alpha_2+\af_3) +2 x y \alpha_1,\\
    \alpha_4^* &= -i(x + i y) (x^2 + y^2) \alpha_7,\\
    \alpha_5^* &= -(x + i y) (x^2 + y^2) \alpha_7,\\
	\alpha_6^* &= i (x + i y) (x^2 + y^2) \alpha_7,\\
	\alpha_7^* &= (x + i y) (x^2 + y^2) \alpha_7.
\end{align*}
    \begin{enumerate}
        \item $(\alpha_1,\alpha_2+\af_3) \neq (0,0)$ and $\alpha_1^2 + (\alpha_2+\af_3)^2 \neq 0$. Then we may assume that $\af_1\ne 0$ and $\af_2+\af_3\ne 0$. In this case the equality $\af^*_1=0$ has two distinct roots $y_1=\mu_1x$ and $y_2=\mu_2x$, where at least one of $\mu_1,\mu_2$ is different from $\pm i$ (otherwise $\alpha_2+\af_3=0$). Let $\mu_1\ne\pm i$. Then choosing $y = \mu_1x$ we get $\alpha_1^* = 0$. Observe that in this case $x^2 + y^2\ne 0$ and $\af_1 = \frac{2(\af_2 + \af_3)\mu_1}{1 - \mu_1^2}$. Substituting this into $\af^*_2$, we obtain $\alpha_2^* = \frac{x^2(1 + \mu_1^2)^2(\af_2 + \af_3)}{1-\mu_1^2}\ne 0.$ Therefore, we have the representative $\langle \nabla_2 - i \nabla_4 - \nabla_5 + i \nabla_6 + \nabla_7\rangle.$
        \item $\alpha_2+\af_3 \neq 0$ and $\alpha_1 = \pm i (\alpha_2+\af_3)\ne 0$. Then we get 
        \begin{align*}
            \alpha_1^* &= \pm i (x \pm i y)^2 (\alpha_2+\af_3),\\
            \alpha_2^* &= (x \pm i y)^2 (\alpha_2+\af_3),
        \end{align*}
        so we obtain two representatives $\langle i\nabla_1 + \nabla_2 - i \nabla_4 - \nabla_5 + i \nabla_6 + \nabla_7\rangle$ and $\langle -i\nabla_1 + \nabla_2 - i \nabla_4 - \nabla_5 + i \nabla_6 + \nabla_7\rangle.$
        \item $(\alpha_1,\alpha_2+\af_3) = (0,0)$. Then we get the representative $\langle - i \nabla_4 - \nabla_5 + i \nabla_6 + \nabla_7\rangle.$
    \end{enumerate}
    %by choosing suitable $u$ and $z$ we have the family of representatives $\langle \alpha_1^{\star}\nabla_1 \pm \alpha_5^{\star}\nabla_4+\alpha_5^{\star} \nabla_5 +i\nabla_6+\nabla_7 \rangle.$

 %   Now, by one more action of a suitable automorphism we have 
    
%    \begin{enumerate}
%     \item if $\alpha_1^{\star}\neq 0,$ then we have the family of representatives $\langle \nabla_1 \pm \alpha\nabla_4+\alpha  \nabla_5 +i\nabla_6+\nabla_7 \rangle .$
%        \item if $\alpha_1^{\star}= 0,$ then we have the family of representatives $\langle \pm \alpha\nabla_4+\alpha  \nabla_5 +i\nabla_6+\nabla_7 \rangle .$

%   \end{enumerate}
    \end{enumerate}
    \item $\alpha_4 = -i\alpha_5$. Choosing $n=0$, we have 
     \begin{align*}
    \alpha_1^* &= x^2 \alpha_1 -  x (2 y (\alpha_2 + \alpha_3) + i z (\alpha_5 - \alpha_7) + u (\alpha_5 + \alpha_7)) - y (y \alpha_1 - i u (\alpha_5 - \alpha_7) + z (\alpha_5 + \alpha_7)),\\
    \alpha_2^* &= x^2 \alpha_2 - y (y \alpha_3 + u \alpha_5 - i z \alpha_7) + x (y \alpha_1 - i u \alpha_5 + z \alpha_7),\\
    \alpha_3^* &= (x - i y) z \alpha_5 + x (y \alpha_1 + x \alpha_3 + i u \alpha_7) - y (y \alpha_2 + u \alpha_7),\\
    \alpha_4^* &= -i(x - iy) (x^2 + y^2) \alpha_5,\\
    \alpha_5^* &= (x - i y) (x^2 + y^2) \alpha_5,\\
	\alpha_6^* &= i (x + i y) (x^2 + y^2) \alpha_7,\\
	\alpha_7^* &= (x + i y) (x^2 + y^2) \alpha_7.
\end{align*}
Consider $\alpha_2^* = 0, \alpha_3^* = 0$ as a linear system in $u, z.$ Its determinant is $-i((x - i y)^2 \alpha_5^2 + (x + i y)^2 \alpha_7^2).$ So, we may choose $x,y$ such that $x^2 + y^2 \neq 0, -i((x - i y)^2 \alpha_5^2 + (x + i y)^2 \alpha_7^2) \neq 0$ and $u$ and $z$ to make $\alpha_2^* = \alpha_3^*=0$. Observe that this does not change the conditions on $\af_4,\af_5,\af_6,\af_7$. So, we may assume that $\af_2=\af_3=0$ from the very beginning. 
%Then the choice of $u$ and $z$ as above gives
%\[
%\alpha_1^* = \frac{(x^2 + y^2)^2 \alpha_1 (\alpha_5^2 + \alpha_7^2)}{(x - i y)^2 \alpha_5^2 + (x + i y)^2 \alpha_7^2}.
%\]
We may also suppose that $\alpha_5 + \alpha_7 \neq 0.$         
\begin{enumerate}
    \item $\alpha_5 \neq 0$. Then taking $y = \frac{\alpha_5 - \alpha_7}{i(\alpha_5+\alpha_7)}x$ we may make $\alpha_5^* = \alpha_7^*$, so we shall assume $\af_5=\af_7$. Now, choosing $n=y = u=z=0,$ we get 
     \begin{align*}
    \alpha_1^* &= x^2 \alpha_1,\\
    \alpha_4^* &= -ix^3 \alpha_7,\\
    \alpha_5^* &= x^3 \alpha_7,\\
	\alpha_6^* &= i x^3 \alpha_7,\\
	\alpha_7^* &= x^3 \alpha_7.
\end{align*}
Therefore, we get two representatives $\langle \nabla_1 - i \nabla_4 + \nabla_5 + i\nabla_6 + \nabla_7 \rangle$ and $\langle - i \nabla_4 + \nabla_5 + i\nabla_6 + \nabla_7 \rangle.$
    \item $\alpha_5 = 0$. Then $\alpha_4 = 0$, so that $\af^*_4=\af^*_5=0$. Choosing $n=0$ and the appropriate values of $u$ and $z$ we have $\af^*_2=\af^*_3=0$ and
        \begin{align*}
    \alpha_1^* &= (x - i y)^2 \alpha_1,\\
	\alpha_6^* &= i (x + i y) (x^2 + y^2) \alpha_7,\\
	\alpha_7^* &= (x + i y) (x^2 + y^2) \alpha_7.
\end{align*}
Therefore, we get two representatives $\langle \nabla_1 + i\nabla_6 + \nabla_7 \rangle$ and $\langle i\nabla_6 + \nabla_7 \rangle,$ which will be joined with the series $\langle \nabla_1 + \alpha \nabla_4 + i \nabla_6 + \nabla_7 \rangle_{\alpha \neq 0}$ and $\langle \alpha \nabla_4 + i \nabla_6 + \nabla_7 \rangle_{\alpha \neq 0}.$ Note that by above representatives $\langle \nabla_1 + i\nabla_6 + \nabla_7 \rangle$ and $\langle \nabla_2 + i\nabla_6 + \nabla_7 \rangle$ define the same orbit.
\end{enumerate}
    
\end{enumerate}

\end{enumerate}

Summarizing, 
we have the following distinct orbits

\[ \langle \nabla_4\rangle,
\langle \nabla_3+ \nabla_4\rangle,
\langle \nabla_1+i\nabla_4+\nabla_5\rangle,
\langle i\nabla_4 +\nabla_5\rangle,\]

\[ \langle \alpha \nabla_4+\beta\nabla_5+\nabla_6 \rangle_{ \beta \in \mathbb C_{\geq 0}}, \langle \nabla_1+ \alpha \nabla_4+\beta\nabla_5+\nabla_6 \rangle_{\alpha\beta \neq 0, \beta \in \mathbb C_{\geq 0}},\]
\[  \langle  \nabla_3+\alpha\nabla_4+\nabla_6 \rangle_{\alpha \neq 0}, \langle  \nabla_1 +\alpha\nabla_3 - \nabla_4 +\nabla_6 \rangle_{\alpha \in \mathbb C_{\geq 0}}, \langle \nabla_1 + \alpha \nabla_2 + i \nabla_5 + \nabla_6 \rangle, \langle  \nabla_2 +\alpha\nabla_5 +\nabla_6 \rangle_{\alpha \in \mathbb C_{\geq 0}}\]

\[\langle  \nabla_1+\alpha\nabla_4+i\nabla_6 + \nabla_7 \rangle, \langle  \alpha\nabla_4+i\nabla_6 + \nabla_7 \rangle,\]
\[\langle  \nabla_2+i\alpha\nabla_4 + \alpha\nabla_5+i\nabla_6 + \nabla_7 \rangle_{\alpha \neq 0}, \langle  i\alpha\nabla_4 + \alpha\nabla_5+i\nabla_6 + \nabla_7 \rangle_{\alpha \neq 0},\]
\[\langle i \nabla_1 + \nabla_2 + i \nabla_4 + \nabla_5 + i\nabla_6 + \nabla_7\rangle,  \langle \pm i \nabla_1 + \nabla_2 - i \nabla_4 - \nabla_5 + i\nabla_6 + \nabla_7\rangle\]

The corresponding algebras are

$$    \begin{array}{lllllllllll}
        \T {4}{18}&:& e_1 e_1 = e_3, & e_1e_3=e_4, & e_2 e_2=e_3; \\
        \T {4}{19}&:& e_1 e_1 = e_3, & e_1e_3=e_4,& e_2e_1=e_4, & e_2 e_2=e_3; \\
        \T {4}{20}&:& e_1 e_1 = e_3+e_4, & e_1e_3=ie_4, & e_2 e_2=e_3, & e_2e_3=e_4; \\
        \T {4}{21}&:& e_1 e_1 = e_3, & e_1e_3=ie_4, & e_2 e_2=e_3,& e_2e_3=e_4; \\
        \T {4}{22}(\alpha,\beta)_{\beta\in \mathbb C_{\geq0}}&:& e_1 e_1 = e_3, & e_1e_3=\alpha e_4,&  e_2 e_2=e_3, &e_2e_3=\beta e_4,& e_3e_1=e_4; \\
        \T {4}{23}(\alpha,\beta)_{%\alpha\beta \neq 0,
        \beta \in \mathbb C_{\geq0}}&:& e_1 e_1 = e_3+e_4, & e_1e_3=\alpha e_4,&  e_2 e_2=e_3, &e_2e_3=\beta e_4,& e_3e_1=e_4; \\
        \T {4}{24}(\alpha)%_{\alpha \neq 0}
        &:& e_1 e_1 = e_3,& e_1e_3 = \alpha e_4, & e_2e_1 = e_4, & e_2e_2 = e_3, & e_3e_1 = e_4; \\
        \T {4}{25}(\alpha)_{\alpha \in \mathbb C_{\geq0}}&:& e_1 e_1 = e_3+e_4, & e_1e_3=- e_4,& e_2e_1 = \alpha e_4, & e_2e_2 = e_3, & e_3e_1 = e_4; \\
        \T {4}{26}(\alpha)&:& e_1 e_1 = e_3+e_4, & e_1e_2=\alpha e_4, & e_2 e_2=e_3, &e_2e_3=i e_4,& e_3e_1=e_4; \\
        \T {4}{27}(\alpha)_{\alpha \in \mathbb C_{\geq0}}&:& e_1 e_1 = e_3, & e_1e_2= e_4, & e_2 e_2=e_3,& e_2e_3=\alpha e_4,& e_3e_1=e_4; \\
        \T {4}{28}(\alpha)&:& e_1 e_1 = e_3 + e_4, & e_1e_3=\alpha e_4, & e_2 e_2=e_3,& e_3e_1=ie_4,& e_3e_2=e_4; \\
        \T {4}{29}(\alpha)&:& e_1 e_1 = e_3, & e_1e_3=\alpha e_4, & e_2 e_2=e_3,& e_3e_1=ie_4,& e_3e_2=e_4; \\
        \T {4}{30}(\alpha)%_{\alpha \neq 0}
        &:& e_1 e_1 = e_3, & e_1 e_2 = e_4, & e_1e_3=i\alpha e_4, & e_2 e_2=e_3,& e_2e_3=\alpha e_4, & e_3e_1=ie_4,& e_3e_2=e_4; \\
        \T {4}{31}(\alpha)%_{\alpha \neq 0}
        &:& e_1 e_1 = e_3, & e_1e_3=i\alpha e_4, & e_2 e_2=e_3,& e_2e_3=\alpha e_4, & e_3e_1=ie_4,& e_3e_2=e_4; \\
        \T {4}{32}&:& e_1 e_1 = e_3 + ie_4, & e_1 e_2 = e_4, & e_1e_3=i e_4, & e_2 e_2=e_3,& e_2e_3= e_4, & e_3e_1=ie_4,& e_3e_2=e_4; \\
        \T {4}{33}&:& e_1 e_1 = e_3 + ie_4, & e_1 e_2 = e_4, & e_1e_3= -ie_4, & e_2 e_2=e_3,& e_2e_3= -e_4, & e_3e_1=ie_4,& e_3e_2=e_4; \\
        \T {4}{34}&:& e_1 e_1 = e_3 -ie_4, & e_1 e_2 = e_4, & e_1e_3= -ie_4, & e_2 e_2=e_3,& e_2e_3= -e_4, & e_3e_1=ie_4,& e_3e_2=e_4.
    \end{array}$$

The algebras above are pairwise non-isomorphic, except $\T {4}{23}(\alpha,0) \cong \T {4}{22}(\alpha,0)$ for $\alpha \neq -1,$ $\T {4}{23}(-1,0) \cong \T {4}{25}(0),$ $\T {4}{23}(0,\beta) \cong \T {4}{22}(0,\beta)$ for $\beta \neq i,$ $\T {4}{23}(0,i) \cong \T {4}{26}(0), \T {4}{24}(0) \cong \T {4}{22}(0,0), \T {4}{30}(0) \cong \T {4}{28}(0), \T {4}{31}(0) \cong \T {4}{29}(0).$ 

\newpage 
\subsubsection{$1$-dimensional  central extensions of $\T {3*}{03}$}
Let us use the following notations 
\begin{align*}
    \nb 1 = \Dl 11, \nb 2 = \Dl 12, \nb 3 = \Dl 13 - \Dl 31, \nb 4 = \Dl 22, \nb 5 = \Dl 23 - \Dl 32, \nb 6 = \Dl 31, \nb 7 = \Dl 32.    
\end{align*}
Take $\0=\sum_{i=1}^7\af_i\nb i\in {\rm H_T^2}(\T {3*}{03}).$
If 
$$
\phi=
\begin{pmatrix}
x &    y  &  0\\
z &    u  &  0\\
v &    w  &  xu-yz
\end{pmatrix}\in\aut{\T {3*}{03}},
$$
then
$$
\phi^T\begin{pmatrix}
\af_1& \af_2& \af_3\\
 0& \af_4& \af_5\\
\af_6-\af_3& \af_7-\af_5& 0
\end{pmatrix} \phi=
\begin{pmatrix}
\af_1^*& \af_2^*-\af^*& \af_3^*\\
\af^*& \af_4^*& \af_5^*\\
\af_6^*-\af_3^*& \af_7^*-\af_5^*& \af^{**}
\end{pmatrix},
$$
where
\begin{align*}
\af^*_1 &= \af_1x^2 + \af_2xz + \af_4z^2 + v(\af_6x + \af_7z),\\
\af^*_2 &= x(2\af_1y + \af_2u) + z(\af_2y + 2\af_4u) + w(\af_6x +\af_7z) + v(\af_6y + \af_7u),\\
\af^*_3 &= (\af_3x+\af_5z)(xu - yz),\\
\af^*_4 &= y(\af_1y + \af_2u) + \af_4u^2 + w(\af_6y + \af_7u),\\
\af^*_5 &= (\af_3y+\af_5u)(xu - yz),\\
\af^*_6 &= (\af_6x+\af_7z)(xu - yz),\\
\af^*_7 &= (\af_6y+\af_7u)(xu - yz).
\end{align*}
Hence, $\phi\langle\0\rangle=\langle\0^*\rangle,$ where $\0^*=\sum\limits_{i=1}^7 \af_i^*  \nb i.$

 We are interested in $\0$ with $(\af_6,\af_7) \neq (0,0).$ Moreover, the condition $\theta \in \mathbf{T}_1 (\T{3*}{03})$ gives us $(\af_3,\af_5, \af_6 -\af_3, \af_7 - \af_5) \neq (0,0,0,0).$

    If $\af_7\neq 0$, then taking $u=-\frac{\af_6y}{\af_7}$ we have $\af^*_7=0$, so we shall assume that $\af^*_6\ne 0$ and $\af^*_7=0$ from the very beginning. Choosing $y=0$, $z=-\frac{\af_3x}{\af_5}$, $u= \frac{\af_6x}{\af_5}$, $v = \frac x{\af_5^2\af_6}(\af_2\af_3\af_5-\af_3^2\af_4 - \af_1\af_5^2)$, $w = \frac x{\af_5^2}(2\af_3\af_4 - \af_2\af_5)$,  we get  the family of representatives 
     $ \langle  \frac{\af_4}{\af_5 x}\nabla_4 +  \nabla_5+  \nabla_6\rangle$. It gives two distinct representatives $ \langle  \nabla_5+  \nabla_6\rangle$ and $ \langle  \nabla_4 +  \nabla_5+  \nabla_6\rangle$ depending on whether $\af_4=0$ or not.
   
   The algebras corresponding to $ \langle  \nabla_4 +  \nabla_5+  \nabla_6\rangle$ and $ \langle  \nabla_5+  \nabla_6\rangle$ are:

$$   
\begin{array}{lllllllll}
%\T {4}{30}(\alpha)&:&  e_1 e_2=e_3, & e_1e_3=\af e_4,& e_2 e_1=-e_3,&  e_2e_2=e_4, & e_3e_1=(1-\af) e_4 \\
%\T {4}{31}(\alpha)&:&  e_1 e_2=e_3, & e_1e_3=\af e_4,& e_2 e_1=-e_3,&     e_3e_1=(1-\af) e_4 \\
\T {4}{35}        &:&  e_1 e_2=e_3, & e_2 e_1=-e_3,& e_2e_2=e_4,& e_2e_3=e_4,& e_3e_1=e_4,&  e_3e_2=-e_4;\\
%\T {4}{33}        &:&  e_1 e_2=e_3, & e_2 e_1=-e_3,& e_2e_2=e_4,&  e_3e_1=e_4\\
\T {4}{36}        &:&  e_1 e_2=e_3, & e_2 e_1=-e_3,&  e_2e_3=e_4,& e_3e_1=e_4,&  e_3e_2=-e_4.
 \end{array} 
 $$

 %\newpage

\subsubsection{$1$-dimensional  central extensions of $\T {3*}{04}$}
Let us use the following notations 
\begin{align*}
    \nb 1 = \Dl 11, \nb 2 = \Dl 12, \nb 3 = \Dl 13, \nb 4 = \Dl 21,
	\nb 5 = \Dl 23, \nb 6 = \Dl 31, \nb 7 = \Dl 32.    
\end{align*}
Take $\0=\sum_{i=1}^7\af_i\nb i\in {\rm H_T^2}(\T {3*}{04}).$
If 
$$
\phi=
 \begin{pmatrix}
    x &      y    &  0\\
    -\lambda y    &    x-y   &  0\\
    z &      u      &  x^2-xy+\lambda y^2
    \end{pmatrix}\in\aut{\T 3{04}},
$$
then
$$
\phi^T\begin{pmatrix}
	\af_1            &  \af_2        & \af_3\\
	\af_4            &      0        & \af_5\\
	\af_6            &  \af_7        & 0
	\end{pmatrix} \phi=
	\begin{pmatrix}
	\af_1^*+\lb\af^*      &  \af_2^*  & \af_3^*\\
	\af_4^*+\af^*         &  \af^*    & \af_5^*\\
	\af_6^*               &  \af_7^*  & 0
	\end{pmatrix},
$$
where
\begin{align*}
	\af^*_1 &= \af_1x^2 + \lb(-\af_1 + \af_2 + \af_4) y^2  -2\lb(\af_2 + \af_4) xy + (\af_3 + \af_6)xz - (\af_5 + \af_7) yz\\
	&\quad- \lb(\af_5 + \af_7) ux + \lb(- \af_3 + \af_5 - \af_6 + \af_7) uy,\\
    \af^*_2 &= \af_2x^2 - \lb\af_4 y^2 + (\af_1 - \af_2)xy + \af_7xz + (\af_6 - \af_7)yz -\af_5\lb uy  + \af_3ux,\\
    \af^*_3 &= (x^2 - xy + \lb y^2)(\af_3x -\lb\af_5 y),\\
    \af^*_4 &= \af_4x^2 + (- \af_1 + (1-\lb)\af_2 + \af_4)y^2 + (\af_1 - \af_2 - 2\af_4)xy + (\af_3 - \af_5)yz + \af_5xz\\
    &\quad - \af_7\lb uy + (- \af_5 + \af_6 - \af_7)ux  + (- \af_3 + \af_5 - \af_6 + \af_7)uy,\\
    \af^*_5 &= (x^2 - xy + \lb y^2)(\af_5x + (\af_3 - \af_5)y),\\
    \af^*_6 &= (x^2 - xy + \lb y^2)(\af_6x -\lb\af_7 y),\\
    \af^*_7 &= (x^2 - xy + \lb y^2)(\af_7x + (\af_6-\af_7)y).
	\end{align*}
Hence, $\phi\langle\0\rangle=\langle\0^*\rangle,$ where $\0^*=\sum\limits_{i=1}^7 \af_i^*  \nb i.$

 We are interested in $\0$ with $(\af_3,\af_5,\af_6,\af_7) \neq (0,0,0,0)$ (if $\lambda \neq 0$) and $(\af_3,\af_6,\af_7) \neq (0,0,0)$ (if $\lambda = 0$). Moreover, the condition $\theta \in \mathbf{T}_1 (\T{3*}{04})$ gives us $(\af_4,\af_5,\af_6,\af_7) \neq (0,0,0,0).$
 
    \begin{enumerate}
    
        \item $\af_7\ne 0$. We have the following subcases.
	\begin{enumerate}
	    \item $\af_6^2 - \af_6\af_7 + \lb\af_7^2\ne 0$. Choosing $x=-\frac{(\af_6-\af_7)y}{\af_7}$, we have $\af^*_7=0$. Since $(\af^*_6)^2 - \af^*_6\af^*_7 + \lb(\af^*_7)^2=(\af_6^2 - \af_6\af_7 + \lb\af_7^2)(x^2 - xy + \lb y^2)^3$, the condition $\af_6^2 - \af_6\af_7 + \lb\af_7^2\ne 0$ is invariant under the action of the automorphism group. Thus, we may assume that $\af_7=0$ and $\af_6\ne 0$ from the very beginning. Then choosing $y=0$ we obtain $\af^*_7=0$ and
        \begin{align*}
	    \af^*_1 &= x(\af_1x  + (\af_3 + \af_6)z - \lb\af_5u),\\
        \af^*_2 &= x(\af_2x  + \af_3u),\\
        \af^*_3 &= \af_3 x^3,\\
        \af^*_4 &= x(\af_4x + \af_5z + (- \af_5 + \af_6)u),\\
        \af^*_5 &= \af_5 x^3,\\
        \af^*_6 &= \af_6 x^3.
	    \end{align*}
	    \begin{enumerate}
	        \item $\af_5\ne 0$ and $\af_3\ne 0$. Then we choose $u=-\frac{\af_2x}{\af_3}$ and $z=-\frac{\af_4x+(- \af_5 + \af_6)u}{\af_5}$ and obtain the family of representatives $\la\af\nb 1+\bt\nb 3+\gm\nb 5+\nb 6\ra_{\bt,\gm\ne 0}$. It gives two families of representatives of distinct orbits: $\la\nb 1+\af\nb 3+\bt\nb 5+\nb 6\ra_{\af,\bt\ne 0}$ and $\la\af\nb 3+\bt\nb 5+\nb 6\ra_{\af,\bt\ne 0}$.
	        
	        \item $\af_5\ne 0$ and $\af_3=0$. Then $\af^*_3=0$ and $\af^*_4$, $\af^*_5$, $\af^*_6$ are as above and
	        \begin{align*}
	            \af^*_1 &= x(\af_1x  + \af_6z - \lb\af_5u),\\
                \af^*_2 &= \af_2x^2.
	        \end{align*}
	        Choosing $z=\frac{\lb\af_5u-\af_1x}{\af_6}$, we have $\af^*_1=0$. Now we have the following subcases:
	        \begin{enumerate}
	            \item $\af_6^2 - \af_5\af_6 + \lb\af_5^2\ne 0$. Then choosing $u=\frac{(\af_1\af_5 - \af_4\af_6)x}{\af_6^2 - \af_5\af_6 + \lb\af_5^2}$ we have $\af^*_4=0$, so we obtain the family of representatives $\la\af\nb 2+\bt\nb 5+\nb 6\ra_{1-\bt+\lb\bt^2\ne 0}$. It determines two families of representatives of distinct orbits: $\la\nb 2+\af\nb 5+\nb 6\ra_{\alpha \neq 0, 1-\af+\lb\af^2\ne 0}$ and $\la\af\nb 5+\nb 6\ra_{\alpha \neq 0, 1-\af+\lb\af^2\ne 0}$.
	            \item $\af_6^2 - \af_5\af_6 + \lb\af_5^2=0$. 
	            
	            If $\lb\not\in\{0,\frac 14\}$, then we have two families of representatives of distinct orbits %$\La\nb 2+\af\nb 4+\frac{1\pm\sqrt{1-4\lb}}{2\lb}\nb 5+\nb 6\Ra$ 
	            $\La \lambda\nb 2+\lambda\af\nb 4+\frac{1+\sqrt{1-4\lb}}{2}\nb 5+\lambda\nb 6\Ra, \La\nb 2+\af\nb 4+\frac{2}{1+\sqrt{1-4\lb}}\nb 5+\nb 6\Ra$ and 4 separate representatives %$\La\nb 4+\frac{1\pm\sqrt{1-4\lb}}{2\lb}\nb 5+\nb 6\Ra$, 
	           $\La\lambda \nb 4+\frac{1+\sqrt{1-4\lb}}{2}\nb 5+\lambda \nb 6\Ra, \La\nb 4+\frac{2}{1+\sqrt{1-4\lb}}\nb 5+\nb 6\Ra,$ $\La\frac{1\pm\sqrt{1-4\lb}}{2\lb}\nb 5+\nb 6\Ra,$ the last two belonging to the family above.
	            
	            If $\lb=\frac 14$, then we have the family of representatives of distinct orbits $\la\nb 2+\af\nb 4+2\nb 5+\nb 6\ra$ and two separate representatives $\la\nb 4+2\nb 5+\nb 6\ra$ and $\la 2\nb 5+\nb 6\ra$. 
	            
	            If $\lb=0$, then we have $\af_5=\af_6$, so we obtain the family of representatives of distinct orbits $\la\nb 2+\af\nb 4+\nb 5+\nb 6\ra$ and two separate representatives $\la\nb 4+\nb 5+\nb 6\ra$ and $\la\nb 5+\nb 6\ra$.
	        
	        \end{enumerate}
	        
	        \item $\af_3\ne 0$ and $\af_5=0$. Then $\af^*_5=0$, $\af^*_2$, $\af^*_3$ and $\af^*_6$ are as above and
	        \begin{align*}
	            \af^*_1 &= x(\af_1x  + (\af_3 + \af_6)z),\\
                \af^*_4 &= x(\af_4x + \af_6u).
	        \end{align*}
	        Choosing $u=-\frac{\af_4x}{\af_6}$, we get $\af^*_4=0$. Now we have two subcases:
	        \begin{enumerate}
	            \item $\af_3 + \af_6\ne 0$. Then choosing $z=-\frac{\af_1x}{\af_3 + \af_6}$, we obtain two families of representatives of distinct orbits $\la \nb 2+\af\nb 3+\nb 6\ra_{\af\not\in \{0,-1\}}$ and $\la \af\nb 3+\nb 6\ra_{\af\not\in \{0,-1\}}$.
	            \item $\af_3 + \af_6=0$. Then choosing $z=0$, we have the family of representatives of distinct orbits $\la\nb 1+\af\nb 2-\nb 3+\nb 6\ra$ and two separate representatives $\la\nb 2-\nb 3+\nb 6\ra$ and $\la-\nb 3+\nb 6\ra$ which will be joined with the families $\la \nb 2+\af\nb 3+\nb 6\ra_{\af\not\in \{0,-1\}}$ and $\la \af\nb 3+\nb 6\ra_{\af\not\in \{0,-1\}}$ found above.
	        \end{enumerate}
	        
	        \item $\af_3=\af_5=0$. Then $\af^*_3=\af^*_5=0$, $\af^*_6$ is as above and 
	        \begin{align*}
	            \af^*_1 &= x(\af_1x + \af_6z),\\
                \af^*_2 &= \af_2x^2,\\
                \af^*_4 &= x(\af_4x + \af_6u).
	        \end{align*}
	        Thus, choosing $z=-\frac{\af_1x}{\af_6}$ and $u=-\frac{\af_4x}{\af_6}$, we have two representatives depending on whether $\af_2=0$ or not: $\la\nb 2+\nb 6\ra$ and $\la\nb 6\ra$. They will be joined with the families $\la\nb 2+\af\nb 5+\nb 6\ra_{\alpha \neq 0, 1-\af+\lb\af^2\ne 0}$ and $\la\af\nb 5+\nb 6\ra_{\alpha \neq 0, 1-\af+\lb\af^2\ne 0}$ found above.
	    \end{enumerate}
	    
	    \item $\af_6^2 - \af_6\af_7 + \lb\af_7^2=0$ and $\af_5\ne 0$ and $\af_3^2 - \af_3\af_5 + \lb\af_5^2\ne 0$. Then choosing $x=-\frac{(\af_3 - \af_5)y}{\af_5}$, we get $\af^*_5=0$. Since $(\af^*_3)^2 - \af^*_3\af^*_5 + \lb(\af^*_5)^2=(\af_3^2 - \af_3\af_5 + \lb\af_5^2)(x^2-xy+\lb y^2)^3\ne 0$, we may assume that $\af_5=0$ and $\af_3\ne 0$ from the very beginning. Then choosing $y=0$ and $z=-\frac{\af_2x+ \af_3u}{\af_7}$ we have $\af^*_2=\af^*_5=0$ and 
	    \begin{align*}
	        \af^*_1 &= \frac x{\af_7}((\af_1\af_7- \af_2\af_3 - \af_2\af_6)x - (\af_3^2  + \af_3\af_6 - \af_6^2 + \af_6\af_7)u),\\
            \af^*_3 &= \af_3x^3,\\
            \af^*_4 &= x(\af_4x - (\af_6 - \af_7)u),\\
            \af^*_6 &= \af_6x^3,\\
            \af^*_7 &= \af_7x^3.
	    \end{align*}
	    \begin{enumerate}
	        \item $\af_6 - \af_7\ne 0$. Then choosing $u=\frac{\af_4x}{\af_6 - \af_7}$ we have $\af^*_4=0$.
	        
	        If $\lb\not\in\{0,\frac 14\}$, then we have four families of representatives of distinct orbits $\La \nb 1+\af\nb 3+ \frac{1\pm\sqrt{1-4\lb}}{2}\nb 6+\nb 7\Ra_{\af\ne 0}$, $\La \af\nb 3+ \frac{1\pm\sqrt{1-4\lb}}{2}\nb 6+\nb 7\Ra_{\af\ne 0}$.
	        
	       	If $\lb=\frac 14$, then $\af_6^2 - \af_6\af_7 + \lb\af_7^2=0$ implies $\af_6=\frac 12\af_7$, so $\af_6 - \af_7\ne 0$ is satisfied. Thus, we have two families of representatives of distinct orbits $\la\nb 1+\af\nb 3+\frac 12\nb 6+\nb 7\ra_{\af\ne 0}$ and $\la\af\nb 3+\frac 12\nb 6+\nb 7\ra_{\af\ne 0}$.
	        	        
	        If $\lb = 0$, then $\af_6^2 - \af_6\af_7 + \lb\af_7^2=0$ and $\af_6 - \af_7\ne 0$ imply that $\af_6=0$. Hence, we have two families of representatives of distinct orbits $\la\nb 1+\af\nb 3+\nb 7\ra_{\af\ne 0}$ and $\la\af\nb 3+\nb 7\ra_{\af\ne 0}$.
	        
	        \item $\af_6 - \af_7=0$. Then $\af_6^2 - \af_6\af_7 + \lb\af_7^2=0$ and $\af_7\ne 0$ imply that $\lb = 0$. Now,
	        \begin{align*}
	        \af^*_1 &= \frac x{\af_7}((\af_1\af_7- \af_2\af_3 - \af_2\af_7)x - \af_3(\af_3  + \af_7)u),\\
            \af^*_3 &= \af_3x^3,\\
            \af^*_4 &= \af_4x^2,\\
            \af^*_6 &= \af_7x^3,\\
            \af^*_7 &= \af_7x^3.
	    \end{align*}
	        \begin{enumerate}
	            \item $\af_3  + \af_7\ne 0$. Then choosing $u=\frac{(\af_1\af_7- \af_2\af_3 - \af_2\af_7)x}{\af_3(\af_3  + \af_7)}$, we have $\af^*_1=0$. Thus, we obtain two families of representatives of distinct orbits $\la\af\nb 3+\nb 4+\nb 6+\nb 7\ra_{\af\ne 0,-1}$ and $\la\af\nb 3+\nb 6+\nb 7\ra_{\af\ne 0,-1}$.
	            \item $\af_3  + \af_7=0$. Then we obtain the family of representatives of distinct orbits $\la\nb 1-\nb 3+\af\nb 4+\nb 6+\nb 7\ra$ and two separate representatives $\la-\nb 3+\nb 4+\nb 6+\nb 7\ra$ and  $\la-\nb 3+\nb 6+\nb 7\ra,$ which will be joined with the families $\la\af\nb 3+\nb 4+\nb 6+\nb 7\ra_{\af\ne 0,-1}$ and $\la\af\nb 3+\nb 6+\nb 7\ra_{\af\ne 0,-1}$ found above.
	        \end{enumerate}
	    \end{enumerate}
	    \item $\af_6^2 - \af_6\af_7 + \lb\af_7^2=0$ and $\af_5\ne 0$ and $\af_3^2 - \af_3\af_5 + \lb\af_5^2=0$. Choosing $x,y$ such that $(x^2 - xy + \lb y^2)(\af_7x + (\af_6-\af_7)y)=1$, we have $(\af_3 - \af_5)y + \af_5x\ne 0$, since otherwise $x^2 - xy + \lb y^2=(\af_3^2 - \af_3\af_5 + \lb\af_5^2)\frac{y^2}{\af_5^2}=0$. Now, the suitable value of $z$ gives $\af^*_4=0$, so we shall assume $\af_7=1$ and $\af_4=0$ from the very beginning. The equality $\af_6^2 - \af_6\af_7 + \lb\af_7^2=0$ takes the form $\af_6^2 - \af_6 + \lb=0$, whence $\lb=\af_6-\af_6^2$. On the other hand, $\af_3^2 - \af_3\af_5 + \lb\af_5^2=0$ implies $\lb=\frac{\af_3\af_5-\af_3^2}{\af_5^2}$. Therefore, $\af_6-\af_6^2=\frac{\af_3\af_5-\af_3^2}{\af_5^2}$, i.e. $\af_3^2-\af_3\af_5+\af_5^2(\af_6-\af_6^2)=0$. This equation has two solutions in $\af_3$, namely, $\af_3=\af_5\af_6$ and $\af_3=\af_5(1-\af_6)$.
	    
	    \begin{enumerate}
	    	\item $\af_3=\af_5\af_6$. Then
	    	\begin{align*}
	    		\af^*_2 &= (x + (\af_6-1)y)z + \af_5\af_6(x + (\af_6-1)y)u  + \af_2x^2 + \af_1xy - \af_2xy,\\
	    		\af^*_4 &= \af_5(x + (\af_6-1)y)z + ((\af_6 - \af_5 - 1)x + (\af_6^2  - \af_5\af_6  + \af_5 - 2\af_6 + 1)y)u\\
	    		&\quad + (\af_1-\af_2)xy  + (\af_2 + \af_2\af_6^2 - \af_1 - \af_2\af_6)y^2. 
	    	\end{align*}
	    	We consider $\af^*_2=\af^*_4=0$ as a system of linear equations in $z$ and $u$. Its determinant is $(\af_5 + 1)(\af_6 - \af_5\af_6 - 1)( x + (\af_6-1)y)^2$. So, we have the following cases:
	    	\begin{enumerate}
	    		\item $\af_5 + 1\ne 0$ and $\af_6 - \af_5\af_6 - 1\ne 0$. Then choosing $x,y$ such that $x + (\af_6-1)y\ne 0$, $x - \af_6y\ne 0$ and $z,u$ such that $\af^*_2=\af^*_4=0$ we obtain 
	    		\begin{align*}
	    			\af^*_1 &= (\af_1 - \af_2\af_6)(x - \af_6y)^2,\\
	    			\af^*_3 &= \af_5\af_6(x - \af_6y)(x + (\af_6-1)y)^2,\\
	    			\af^*_5 &= \af_5(x - \af_6y)(x + (\af_6-1)y)^2,\\
	    			\af^*_6 &= \af_6(x - \af_6y)(x + (\af_6-1)y)^2,\\
	    			\af^*_7 &= (x - \af_6y)(x + (\af_6-1)y)^2.
	    		\end{align*}
	    		\begin{itemize}
	    		    \item Let $\af_1 - \af_2\af_6\ne 0$. Taking $y=0$ and $x=\af_1 - \af_2\af_6$ we obtain the family of representatives $\la\nb 1+\af\bt\nb 3+\af\nb 5+\bt\nb 6+\nb 7\ra$, where $\af\not\in \{0,-1\}$, $\bt-\af\bt-1\ne 0$ and $\bt^2 - \bt + \lb=0$.
	    		
	    		If $\lb\not\in\{0,\frac 14\}$, then $\bt=\frac{1\pm\sqrt{1-4\lb}}2$. Observe that $\bt\ne 0$, so the condition $\bt-\af\bt-1\ne 0$ is equivalent to $\af\ne 1-\frac 1\bt=\frac{\sqrt{1-4\lb}\mp 1}{\sqrt{1-4\lb}\pm 1}$. Thus, we obtain two families $\La\nb 1+\frac{1\pm\sqrt{1-4\lb}}2\af\nb 3+\af\nb 5+\frac{1\pm\sqrt{1-4\lb}}2\nb 6+\nb 7\Ra_{\af\not\in\left\{0,-1,\frac{\sqrt{1-4\lb}\mp 1}{\sqrt{1-4\lb}\pm 1}\right\}}$ of representatives of distinct orbits.
	    		
	    		If $\lb=\frac 14$, then $\bt=\frac 12$, and $\af\ne 1-\frac 1\bt$ becomes $\af\ne -1$. Thus, we obtain the family $\la\nb 1+\frac 12\af\nb 3+\af\nb 5+\frac 12\nb 6+\nb 7\ra_{\af\not\in\{0,-1\}}$ of representatives of distinct orbits.
	    		
	    		If $\lb=0$, then $\bt=0$ or $\bt=1$. If $\bt=0$, then the condition $\bt-\af\bt-1\ne 0$ becomes $-1\ne 0$; and if $\bt=1$, then it is $\af\ne 0$. Thus, we obtain two families $\la\nb 1+\af\nb 5+\nb 7\ra_{\af\not\in\{0,-1\}}$ and $\la\nb 1+\af\nb 3+\af\nb 5+\nb 6+\nb 7\ra_{\af\not\in\{0,-1\}}$ of representatives of distinct orbits.
	    		    \item Let  $\af_1 - \af_2\af_6=0$. Then we obtain the same families as in the previous case, but without $\nb 1$.
	    		\end{itemize}
	   
	    		\item $\af_6 - \af_5\af_6 - 1=0$. Then clearly $\af_6\ne 0$ and $\af_5=1-\frac 1{\af_6}$, so $\af_3=\af_6-1$. Moreover, $\af_6\ne 1$ since $\af_5\ne 0$. In particular, $\lb\ne 0$. Choosing $z$ such that $\af^*_4=0$, we obtain
	    		\begin{align*}
	    			\af^*_1 &= \frac{x - \af_6y}{\af_6 - 1}(\af_1(\af_6-1)x + \af_6(\af_2(2\af_6^2 + 1) - \af_6(\af_1 + 2\af_2))y),\\
	    			\af^*_2 &= \frac{x - \af_6y}{\af_6 - 1}(\af_2(\af_6-1)x  + (\af_2\af_6(\af_6-1) - \af_1 + \af_2)y),\\
	    			\af^*_3 &= (\af_6 - 1)(x - \af_6y)(x + (\af_6-1)y)^2,\\
	    			\af^*_5 &= \left(1 - \frac 1{\af_6}\right)(x - \af_6y)(x + (\af_6-1)y)^2,\\
	    			\af^*_6 &= \af_6(x - \af_6y)(x + (\af_6-1)y)^2,\\
	    			\af^*_7 &= (x - \af_6y)(x + (\af_6-1)y)^2.
	    		\end{align*}
	    		Let $(\alpha_1,\af_2) =(0,0)$. Then we get the representative $\La(\af-1)\nb 3+\left(1-\frac 1{\af}\right)\nb 5+\af\nb 6+\nb 7\Ra$, where $\af^2 - \af + \lb=0$.
				
				If $\lb\not\in\{0,\frac 14\}$, then we obtain two representatives $\La\frac{-1\pm\sqrt{1-4\lb}}2\nb 3+\frac{\sqrt{1-4\lb}\mp 1}{\sqrt{1-4\lb}\pm 1}\nb 5+\frac{1\pm\sqrt{1-4\lb}}2\nb 6+\nb 7 \Ra$ which will be joined with the families $\La\frac{1\pm\sqrt{1-4\lb}}2\af\nb 3+\af\nb 5+\frac{1\pm\sqrt{1-4\lb}}2\nb 6+\nb 7\Ra_{\af\not\in\left\{0,-1,\frac{\sqrt{1-4\lb}\mp 1}{\sqrt{1-4\lb}\pm 1}\right\}}$.
				
				If $\lb=\frac 14$, then we obtain the representative $\la -\frac 12\nb 3-\nb 5+\frac 12\nb 6+\nb 7\ra$ which will be joined with the family $\la\frac 12\af\nb 3+\af\nb 5+\frac 12\nb 6+\nb 7\ra_{\af\not\in\{0,-1\}}.$

				Let $(\alpha_1,\alpha_2) \neq (0,0).$ If $\alpha_2 =0,$ then  
				\begin{align*}
				\af^*_1 &= \frac{x - \af_6y}{\af_6 - 1}((\af_6-1)x - \af_6^2y)\af_1,\\
				\af^*_2 &= -\frac{x - \af_6y}{\af_6 - 1}y\af_1,
				\end{align*}
				so choosing $y \neq 0$ we may suppose for the rest of the case that $\alpha_2 \neq 0.$
\begin{itemize}
    \item Let $\af_1-\af_2\af_6\ne 0$ and $\af_1 - \af_2 - 2\af_2\af_6^2 + 2\af_2\af_6\ne 0$. Then choosing $z=\frac{\af_6(\af_2\af_6(\af_6-1) - \af_1 + \af_2)y}{\af_6 - 1}-(\af_6-1)u$ and $x=\frac{(\af_1 - \af_2 -\af_2\af_6(\af_6-1))y}{\af_2(\af_6 - 1)}$ we obtain $\af^*_2=\af^*_4=0$ and 
	    		\begin{align*}
	    			\af^*_1 &= (\af_1 - \af_2 - 2\af_2\af_6^2 + 2\af_2\af_6)^2(\af_1 - \af_2\af_6)\frac{y^2}{\af_2^2(\af_6 - 1)^2},\\
	    			\af^*_3 &= (\af_1 - \af_2 - 2\af_2\af_6^2 + 2\af_2\af_6)(\af_1 - \af_2\af_6)^2\frac{y^3}{\af_2^3(\af_6 - 1)^2},\\
	    			\af^*_5 &= (\af_1 - \af_2 - 2\af_2\af_6^2 + 2\af_2\af_6)(\af_1 - \af_2\af_6)^2\frac{y^3}{\af_2^3\af_6(\af_6 - 1)^2},\\
	    			\af^*_6 &= (\af_1 - \af_2 - 2\af_2\af_6^2 + 2\af_2\af_6)(\af_1 - \af_2\af_6)^2\frac{\af_6y^3}{\af_2^3(\af_6 - 1)^3},\\
	    			\af^*_7 &= (\af_1 - \af_2 - 2\af_2\af_6^2 + 2\af_2\af_6)(\af_1 - \af_2\af_6)^2\frac{y^3}{\af_2^3(\af_6 - 1)^3}.
	    		\end{align*}
	    		Choosing $y=\frac{\af_2(\af_6 - 1)(\af_1 - \af_2 - 2\af_2\af_6^2 + 2\af_2\af_6)}{\af_1 - \af_2\af_6}$ we obtain the  representative $\La\nb 1+(\af-1)\nb 3+\left(1-\frac 1{\af}\right)\nb 5+\af\nb 6+\nb 7\Ra$, where $\af^2 - \af + \lb=0$.
	    		
	    		If $\lb\not\in\{0,\frac 14\}$, then $\af=\frac{1\pm\sqrt{1-4\lb}}2$, so $\af-1=\frac{-1\pm\sqrt{1-4\lb}}2$ and $1-\frac 1{\af}=\frac{\sqrt{1-4\lb}\mp 1}{\sqrt{1-4\lb}\pm 1}$, and we obtain the following two representatives $\La \nb 1+\frac{-1\pm\sqrt{1-4\lb}}2\nb 3+\frac{\sqrt{1-4\lb}\mp 1}{\sqrt{1-4\lb}\pm 1}\nb 5+\frac{1\pm\sqrt{1-4\lb}}2\nb 6+\nb 7 \Ra$ which will be joined with the families $\La\nb 1+\frac{1\pm\sqrt{1-4\lb}}2\af\nb 3+\af\nb 5+\frac{1\pm\sqrt{1-4\lb}}2\nb 6+\nb 7\Ra_{\af\not\in\left\{0,-1,\frac{\sqrt{1-4\lb}\mp 1}{\sqrt{1-4\lb}\pm 1}\right\}}$ found above.
	    		
	    		If $\lb=\frac 14$, then $\af=\frac 12$, so $\af-1=-\frac 12$ and $1-\frac 1{\af}=-1$, and we obtain the representative $\la\nb 1-\frac 12\nb 3-\nb 5+\frac 12\nb 6+\nb 7\ra$ which will be joined with the family $\la\nb 1+\frac 12\af\nb 3+\af\nb 5+\frac 12\nb 6+\nb 7\ra_{\af\not\in\{0,-1\}}$ found above.
    \item Let $\af_1 - \af_2 - 2\af_2\af_6^2 + 2\af_2\af_6=0$. Then $\af_1=(2\af_6^2 - 2\af_6 + 1)\af_2$. %Taking $z=-(\af_6 - 1)u - \frac{\af_2x(x + 2\af_6(\af_6 - 1)y)}{x+(\af_6 - 1)y}$ we obtain $\af^*_2=0$ and
	    %		\begin{align*}
	    %		\af^*_1 &= 2\af_2(x - \af_6y)^2(\af_6 - 1)^2,\\
	  %  		\af^*_3 &= (\af_6 - 1)(x - \af_6y)(x + (\af_6-1)y)^2,\\
	 %   		\af^*_4 &=-\af_2\left(1 - \frac 1{\af_6}\right)(x - \af_6y)^2,\\
	 %   		\af^*_5 &= \left(1 - \frac 1{\af_6}\right)(x - \af_6y)(x + (\af_6-1)y)^2,\\
	%    		\af^*_6 &= \af_6(x - \af_6y)(x + (\af_6-1)y)^2,\\
	%    		\af^*_7 &= (x - \af_6y)(x + (\af_6-1)y)^2.
	 %   		\end{align*}
	 %   		Choosing $y=0$ and $x=-\af_2\left(1 - \frac 1{\af_6}\right)$ we obtain the representative $\La 2\af(1-\af)\nb 1+(\af-1)\nb 3+\nb 4+\left(1 - \frac 1{\alpha}\right)\nb 5+\af\nb 6+\nb 7\Ra$, where $\af^2 - \af + \lb=0$. Observe that $2\af(1-\af)=2\lb$.
	    		
	 %   		If $\lb\not\in\{0,\frac 14\}$, then we obtain two representatives $\La 2\lb\nb 1+\frac{-1\pm\sqrt{1-4\lb}}2\nb 3+\nb 4+\frac{\sqrt{1-4\lb}\mp 1}{\sqrt{1-4\lb}\pm 1}\nb 5+\frac{1\pm\sqrt{1-4\lb}}2\nb 6+\nb 7 \Ra$.
	    		
	 %   		If $\lb=\frac 14$, then we obtain the representative $\la \frac 12\nb 1-\frac 12\nb 3+\nb 4-\nb 5+\frac 12\nb 6+\nb 7\ra$.

	In this case $\alpha_3^*, \alpha_5^*, \alpha_6^*, \alpha_7^*$ are as above and
		\begin{align*}
	    		\af^*_1 &= \alpha_2 (x - y \alpha_6)^2 (2\af_6^2 - 2\af_6 + 1)\\
	    		\af^*_2 &=\alpha_2 (x - y \alpha_6)^2.
	    		\end{align*}
	Therefore, we obtain the representative $\La (2 \alpha^2-2\alpha + 1)\nabla_1+\nabla_2+ (\alpha-1)\nabla_3 + (1-\frac{1}{\alpha})\nabla_5 +\alpha\nabla_6+\nabla_7\Ra$, where $\af^2 - \af + \lb=0$. Note that $2 \alpha^2-2\alpha + 1 = 1- 2\lambda.$
	
	If $\lb\not\in\{0,\frac 14\},$ then we obtain two representatives $\La (1-2\lb)\nabla_1 + \nabla_2 +\frac{-1+\sqrt{1-4\lb}}2\nb 3+\frac{\sqrt{1-4\lb}- 1}{\sqrt{1-4\lb}+ 1}\nb 5+\frac{1+\sqrt{1-4\lb}}2\nb 6+\nb 7 \Ra, \La (1-2\lb)\nabla_1 + \nabla_2 +\frac{-1-\sqrt{1-4\lb}}2\nb 3+\frac{\sqrt{1-4\lb}+ 1}{\sqrt{1-4\lb}- 1}\nb 5+\frac{1-\sqrt{1-4\lb}}2\nb 6+\nb 7 \Ra = \La (1-2\lb)\lb\nabla_1 + \lb\nabla_2 -\lb\frac{1+\sqrt{1-4\lb}}2\nb 3-\Big(\frac{\sqrt{1-4\lb}+ 1}{2}\Big)^2\nb 5+\lb\frac{1-\sqrt{1-4\lb}}2\nb 6+\lb\nb 7 \Ra$. %$\La (1-2\lb)\nabla_1 + \nabla_2 +\frac{-1\pm\sqrt{1-4\lb}}2\nb 3+\frac{\sqrt{1-4\lb}\mp 1}{\sqrt{1-4\lb}\pm 1}\nb 5+\frac{1\pm\sqrt{1-4\lb}}2\nb 6+\nb 7 \Ra$
	
		If $\lb=\frac 14$, then we obtain the representative $\la \frac 12\nb 1 + \nabla_2 -\frac 12\nb 3-\nb 5+\frac 12\nb 6+\nb 7\ra$.
	
    \item 	Let $\af_1-\af_2\af_6=0$. We may assume that $\af_6\ne\frac 12$ (and hence $\lb\ne\frac 14$), since otherwise $\af_1=(2\af_6^2 - 2\af_6 + 1)\af_2$, which has already been considered.
    
    %Then $\af_1=\af_2\af_6$. Choosing $z=-(\af_6 - 1)u - \af_2x$ we obtain $\af^*_2=0$ and
	%			\begin{align*}
	%			\af^*_1 &= -\af_2(\af_6 - 1)(x - \af_6y)(x + (\af_6-1)y),\\
	%			\af^*_3 &= (\af_6 - 1)(x - \af_6y)(x + (\af_6-1)y)^2,\\
	%			\af^*_4 &= -\af_2\left(1 - \frac 1{\af_6}\right)(x - \af_6y)(x + (\af_6-1)y),\\
	%			\af^*_5 &= \left(1 - \frac 1{\af_6}\right)(x - \af_6y)(x + (\af_6-1)y)^2,\\
	%			\af^*_6 &= \af_6(x - \af_6y)(x + (\af_6-1)y)^2,\\
	%			\af^*_7 &= (x - \af_6y)(x + (\af_6-1)y)^2.
%				\end{align*}
%				Choosing $y=0$ and $x=-\af_2\left(1 - \frac 1{\af_6}\right)$, we obtain the representative $\La\af\nb 1+(\af-1)\nb 3+\nb 4+\left(1-\frac 1{\af}\right)\nb 5+\af\nb 6+\nb 7\Ra$, where $\af^2 - \af + \lb=0$.
				
	%			If $\lb\not\in\{0,\frac 14\}$, then we obtain two representatives $\La \frac{1\pm\sqrt{1-4\lb}}2\nb 1+\frac{-1\pm\sqrt{1-4\lb}}2\nb 3+\nb 4+\frac{\sqrt{1-4\lb}\mp 1}{\sqrt{1-4\lb}\pm 1}\nb 5+\frac{1\pm\sqrt{1-4\lb}}2\nb 6+\nb 7 \Ra$.
				
	%			If $\lb=\frac 14$, then we obtain the representative $\la \frac 12\nb 1-\frac 12\nb 3+\nb 4-\nb 5+\frac 12\nb 6+\nb 7\ra$ found before.
				
	%			\bigskip
				
	In this case $\alpha_3^*, \alpha_5^*, \alpha_6^*, \alpha_7^*$ are as above and
		\begin{align*}
	    		\af^*_1 &= \alpha_2\alpha_6 (x - y \alpha_6)(x+(\alpha_6 - 1)y)\\
	    		\af^*_2 &=\alpha_2 (x - y \alpha_6)(x+(\alpha_6 - 1)y).
	    		\end{align*}
	Therefore, we obtain the representative $\La \alpha\nabla_1+\nabla_2+ (\alpha-1)\nabla_3 + (1-\frac{1}{\alpha})\nabla_5 +\alpha\nabla_6+\nabla_7\Ra$, where $\af^2 - \af + \lb=0$. 
	
		If $\lb\not\in\{0,\frac 14\},$ then we obtain two representatives $\La \frac{1+\sqrt{1-4\lb}}2\nabla_1 + \nabla_2 +\frac{-1+\sqrt{1-4\lb}}2\nb 3+\frac{\sqrt{1-4\lb}- 1}{\sqrt{1-4\lb}+1}\nb 5+\frac{1+\sqrt{1-4\lb}}2\nb 6+\nb 7 \Ra, \La \frac{1-\sqrt{1-4\lb}}2\nabla_1 + \nabla_2 -\frac{1+\sqrt{1-4\lb}}2\nb 3+\frac{\sqrt{1-4\lb}+ 1}{\sqrt{1-4\lb}- 1}\nb 5+\frac{1-\sqrt{1-4\lb}}2\nb 6+\nb 7 \Ra = \La \lb\frac{1-\sqrt{1-4\lb}}2\nabla_1 + \lb\nabla_2 -\lb\frac{1+\sqrt{1-4\lb}}2\nb 3-\Big(\frac{\sqrt{1-4\lb}+ 1}{2}\Big)^2\nb 5+\lb\frac{1-\sqrt{1-4\lb}}2\nb 6+\lb\nb 7 \Ra$
		%$\La \frac{1\pm\sqrt{1-4\lb}}2\nabla_1 + \nabla_2 +\frac{-1\pm\sqrt{1-4\lb}}2\nb 3+\frac{\sqrt{1-4\lb}\mp 1}{\sqrt{1-4\lb}\pm 1}\nb 5+\frac{1\pm\sqrt{1-4\lb}}2\nb 6+\nb 7 \Ra$.
	    
	    %If $\lb=\frac 14$, then we obtain the representative $\la \frac 12\nb 1 + \nabla_2 -\frac 12\nb 3-\nb 5+\frac 12\nb 6+\nb 7\ra$ found before.

\end{itemize}
				\item $\af_5 + 1=0$. Then $\af_3=-\af_6$. Observe that we may assume that $\af_6\ne\frac 12$ (and hence $\lb\ne\frac 14$), since otherwise $\af_6 - \af_5\af_6 - 1=0,$ which was considered above. Choosing $z$ such that $\af^*_4=0$, we have
				\begin{align*}
					\af^*_1 &= \af_1x^2 + 2\af_2\af_6(\af_6 - 1)xy + \af_6(\af_6-1)(\af_1 - \af_2)y^2,\\
					\af^*_2 &= \af_2x^2 + 2(\af_1 - \af_2)xy + (\af_2\af_6(\af_6-1) - \af_1 + \af_2)y^2,\\
					\af^*_3 &= -\af_6(x - \af_6y)(x + (\af_6-1)y)^2,\\
					\af^*_5 &= -(x - \af_6y)(x + (\af_6-1)y)^2,\\
					\af^*_6 &= \af_6(x - \af_6y)(x + (\af_6-1)y)^2,\\
					\af^*_7 &= (x - \af_6y)(x + (\af_6-1)y)^2.
				\end{align*}
				
				If $(\alpha_1,\alpha_2)=(0,0)$, then we obtain the representative $\la-\bt\nb 3-\nb 5+\bt\nb 6+\nb 7\ra$, where $\bt^2 - \bt + \lb=0$.
				
				If $\lb\not\in\{0,\frac 14\}$, then we have two representatives $\La\frac{-1\mp\sqrt{1-4\lb}}2\nb 3-\nb 5+\frac{1\pm\sqrt{1-4\lb}}2\nb 6+\nb 7\Ra$ which will be joined with the families $\La\frac{1\pm\sqrt{1-4\lb}}2\af\nb 3+\af\nb 5+\frac{1\pm\sqrt{1-4\lb}}2\nb 6+\nb 7\Ra_{\af\ne 0,-1}$ found above.
				
%				If $\lb=\frac 14$, we have the representative $\la-\frac 12\nb 3-\nb 5+\frac 12\nb 6+\nb 7\ra$ which was found above.
				
				If $\lb=0$, then we have two representatives $\la-\nb 3-\nb 5+\nb 6+\nb 7\ra$ and $\la-\nb 5+\nb 7\ra$ which will be joined the families $\la\af\nb 3+\af\nb 5+\nb 6+\nb 7\ra_{\af\ne 0,-1}$ and $\la\af\nb 5+\nb 7\ra_{\af\ne 0,-1}$ found above.
				
				Let $(\alpha_1,\alpha_2) \neq (0,0).$ Suppose that $\af_2=0$. Then we have $\af^*_2 =  \af_1 y (2x - y).$ So, whenever $(\alpha_1,\alpha_2) \neq (0,0)$, we may find $x,y$ such that $\af^*_2\neq 0$ and $\af^*_7=1$. In the rest of the case we will suppose that $\alpha_2 \neq 0.$
				
				The determinant of the equation $\af^*_2=0$ is $4(\af_2\af_6 + \af_1 - \af_2)(\af_1-\af_2\af_6)$.
				
				\begin{itemize}
				    \item Let $(\af_2\af_6 + \af_1 - \af_2)(\af_1-\af_2\af_6)\ne 0.$ 
				    
				    Suppose that $\af_2 = 2\af_1.$ Observe that in this case $(\af_2\af_6 + \af_1 - \af_2)(\af_2\af_6 - \af_1)\ne 0$ is equivalent to $\af_1^2(2\af_6 - 1)^2\ne 0$. Then $\alpha_2^* - 2\alpha_1^* = -2\alpha_1(2x-y)y(2\alpha_6-1)^2,$ so choosing appropriate $x,y$ we may suppose that $\af_2\ne 2\af_1.$
				    
				    The equation $\af^*_2=0$ has two  solutions $x_1=\mu_1 y$ and $x_2=\mu_2 y$, where $\mu_1,\mu_2\in\Co$, $\mu_1\ne\mu_2$. As in the case 1(a) if $\mu_1^2 - \mu_1 + \lb=\mu_2^2 - \mu_2 + \lb=0$, then $\mu_1+\mu_2=1$. But $\mu_1+\mu_2=-\frac{2(\af_1 - \af_2)}{\af_2}$, whence $2(\af_1 - \af_2)=-\af_2$, which contradicts the assumption that $\af_2\ne 2\af_1$. Thus, we may choose $x=\mu_iy$ to make $\af^*_2=0$. Since in this case the condition $(\af_1,\af_2,\af_4)\ne (0,0,0)$ is invariant under automorphisms, we have $\af^*_1\ne 0$ for such a choice of $x$. Thus, we obtain the family of representatives $\la\af\nb 1-\bt\nb 3-\nb 5+\bt\nb 6+\nb 7\ra$, where $\bt^2 - \bt + \lb=0$ and $\af\ne 0$. Then taking $y=0$ and $x=\af$, we obtain the representative $\la\nb 1-\bt\nb 3-\nb 5+\bt\nb 6+\nb 7\ra.$  It will be joined with the family $\la\nb 1+\af\bt\nb 3+\af\nb 5+\bt\nb 6+\nb 7\ra_{\af\not\in \{0,-1\}}.$
				\item 	Let $\af_1-\af_2\af_6=0$. Then
				\begin{align*}
				\af^*_1 =\af_2\af_6(x + (\af_6-1)y)^2,
				\af^*_2 =\af_2(x + (\af_6-1)y)^2.
				\end{align*}
			 	So, taking $x=\af_6y+\af_2$, we obtain the representative $\la\bt\nb 1+\nb 2-\bt\nb 3-\nb 5+\bt\nb 6+\nb 7\ra$, where $\bt^2 - \bt + \lb=0$. If $\lb\not\in\{0,\frac 14\}$, then we have two representatives $\La\frac{1\pm\sqrt{1-4\lb}}2\nb 1+\nb 2-\frac{1\pm\sqrt{1-4\lb}}2\nb 3-\nb 5+\frac{1\pm\sqrt{1-4\lb}}2\nb 6+\nb 7\Ra$. %If $\lb=\frac 14$, we have the representative $\la\frac 12\nb 1+\nb 2-\frac 12\nb 3-\nb 5+\frac 12\nb 6+\nb 7\ra$ found above.
			 	If $\lb=0$, then we have two representatives $\la\nb 1+\nb 2-\nb 3-\nb 5+\nb 6+\nb 7\ra$ and $\la\nb 2-\nb 5+\nb 7\ra$.
			 	\item 		Let $\af_2\af_6 + \af_1 - \af_2=0$. Then
			 	\begin{align*}
			 	\af^*_1 =\af_2(1-\af_6)(x - \af_6y)^2,
			 	\af^*_2 =\af_2(x - \af_6y)^2.
			 	\end{align*}
			 	Taking $x,y$ such that $(x + (\af_6-1)y)^2=\af_2(x - \af_6y)\ne 0$, we obtain the representative $\la(1-\bt)\nb 1+\nb 2-\bt\nb 3-\nb 5+\bt\nb 6+\nb 7\ra$, where $\bt^2 - \bt + \lb=0$. If $\lb\not\in\{0,\frac 14\}$, then we have two representatives $\La\frac{1\mp\sqrt{1-4\lb}}2\nb 1+\nb 2-\frac{1\pm\sqrt{1-4\lb}}2\nb 3-\nb 5+\frac{1\pm\sqrt{1-4\lb}}2\nb 6+\nb 7\Ra$. 
			 	%If $\lb=\frac 14$, we have the representative $\la\frac 12\nb 1+\nb 2-\frac 12\nb 3-\nb 5+\frac 12\nb 6+\nb 7\ra$ found above.
			 	If $\lb=0$, then we have two representatives $\la\nb 2-\nb 3-\nb 5+\nb 6+\nb 7\ra$ and $\la\nb 1+\nb 2-\nb 5+\nb 7\ra$.
				\end{itemize}
				%Let $\af_2\ne 0$, $(\af_2\af_6 + \af_1 - \af_2)(\af_2\af_6 - \af_1)\ne 0$ and $\af_2= 2\af_1$. 
				%Then
				%\begin{align*}
				%\af^*_2 = \frac{\af_1}2((2x-(1+(2\af_6-1)i)y)(2x-(1-(2\af_6-1)i)y)).
				%\end{align*}
			%	Choosing $x=\frac{(1+(2\af_6-1)i)y}2$ and $z=\af_1((1 - i)\af_6 - 1)y + \af_6u$ we obtain $\af^*_2=\af^*_4=0$ and 
			%	\begin{align*}
			%		\af^*_1 &=\frac i2\af_1(2\af_6 - 1)^3y^2,\\
			%		\af^*_3 &=\frac{i+1}4(2\af_6 - 1)^3\af_6y^3,\\
		%			\af^*_5 &=\frac{i+1}4(2\af_6 - 1)^3y^3,\\
			%		\af^*_6 &=-\frac{i+1}4(2\af_6 - 1)^3\af_6y^3,\\
			%		\af^*_7 &=-\frac{i+1}4(2\af_6 - 1)^3y^3.
			%	\end{align*}
			%	Then choosing $y=-(i+1)\af_1$ we obtain the representative $\la\nb 1-\bt\nb 3-\nb 5+\bt\nb 6+\nb 7\ra$, where $\bt^2 - \bt + \lb=0$. It was found above.

	    	\end{enumerate}
    	
    		\item $\af_3=\af_5(1-\af_6)$. Observe that we may assume that $\af_6\ne\frac 12$ (and hence $\lb\ne\frac 14$), since otherwise $\af_3=\af_5\af_6$, which has already been considered.
    		
  		Consider $\af^*_2=\af^*_4=0$ as a system of linear equations in $z$ and $u$. Then its determinant is the following polynomial in $x$ and $y$:
    		\begin{align*}
    			D(x,y)&=(\af_5^2\af_6 - \af_5^2 - \af_5 + \af_6 - 1)x^2\\
    			&\quad-(2\af_5^2\af_6^2 - 2\af_5^2\af_6 - 2\af_6^2 - \af_5 + 4\af_6 - 2)xy\\
    			&\quad+(\af_5^2\af_6^2 + \af_5\af_6 + \af_6^2 - 2\af_6 + 1)(\af_6 - 1)y^2.
    		\end{align*}
    		
    		Suppose that $\af_5^2\af_6 - \af_5^2 - \af_5 + \af_6 - 1 = 0.$ In this case $\af_5^2+1\ne 0$, since otherwise $(\af_5^2+1)\af_6 = \af_5^2 + \af_5 + 1$ would imply $\af_5=0$, whence $\af_5^2+1=1$, a contradiction. Thus, $\af_6=\frac{\af_5^2 + \af_5 + 1}{\af_5^2+1}$ and hence $\af_3=-\frac{\af_5^2}{\af_5^2+1}$. Then $D(x,y) = \frac{y \alpha_5 (1 + \alpha_5)^3 ((1-\alpha_5)x + \alpha_5y )}{(1 + \alpha_5^2)^2}.$ If $\af_5+1=0,$ then $\af_5=-1$, $\af_6=\frac 12$, $\af_3=-\frac 12=\af_5\af_6$. This case has been considered above. Therefore, we may suppose that $D(x,y)$ is not identically zero, and hence we may solve $\af^*_2=\af^*_4=0$ in $z$ and $u$. Choosing these values of $z$ and $u$, we may suppose that $\alpha_2 = \alpha_4 = 0$ from the very beginning.
    		
Then taking the appropriate values of $z$ and $u$ we have 	
\begin{align*}  	       
                    \af^*_1&=\frac{\alpha_1(\af_5^2\af_6 - \af_5^2 - \af_5 + \af_6 - 1)(x + y(\af_6-1))^2(x-y\af_6)^2}{D(x,y)},\\	
                    \af^*_3&=\alpha_5 (1 - \alpha_6) (x - y \alpha_6)^2 (x + y(\af_6-1)),\\
    				\af^*_5&=\alpha_5 (x - y \alpha_6)^2 (x + y(\af_6-1)),\\
    				\af^*_6&=\alpha_6 (x - y \alpha_6) (x + y(\af_6-1))^2,\\
    				\af^*_7&=(x - y \alpha_6) (x + y(\af_6-1))^2.
    			\end{align*}
    			
    		Observe that for $x,y$ such that $\alpha_7^* = 1$ and $D(x,y)\ne 0$ we have $(\af_5^*)^2\af_6^* - (\af_5^*)^2 - \af_5^* + \af_6^* - 1 = \frac{D(x,y)}{(x + y(\alpha_6-1))^2}\ne 0,$ so we may assume that $\af_5^2\af_6 - \af_5^2 - \af_5 + \af_6 - 1 \neq 0$ from the very beginning.
    		
    		Let $\alpha_5 \neq -1$.  Taking $x = \frac{y (\alpha_6(\alpha_5-1)+1)}{\alpha_5+1},$ we have $ D(x,y) = \frac{y^2 \alpha_5^2 (2 \alpha_6-1)^2}{(\alpha_5+1)^2}\ne 0$, $x^2 - xy + \lambda y^2 = -\frac{y^2 \alpha_5 (2 \alpha_6-1)^2}{(\alpha_5+1)^2}\ne 0$, $\alpha_5^* = -\alpha_7^*$, so we get the representative $\langle \alpha \nabla_1 + (\beta-1)\nabla_3 - \nabla_5 + \beta \nabla_6 + \nabla_7 \rangle_{\beta^2 - \beta + \lambda=0},$ where
            \[
            \alpha = -\frac{\alpha_1 (\alpha_5+1) (\af_5^2\af_6 - \af_5^2 - \af_5 + \af_6 - 1)}{y \alpha_5^2 (2 \alpha_6-1)^2}
            \]
            
    \begin{itemize}
        \item  If $\alpha_1 \neq 0,$ then we take $y = -\frac{\alpha_1 (\alpha_5+1) (\af_5^2\af_6 - \af_5^2 - \af_5 + \af_6 - 1)}{\alpha_5^2 (2 \alpha_6-1)^2} \neq 0,$ so that $x^2 - xy + \lambda y^2 \neq 0$, $D(x,y) \neq 0$, and we get the representative $\langle \nabla_1 + (\alpha-1)\nabla_3 - \nabla_5 + \alpha \nabla_6 + \nabla_7 \rangle_{\alpha^2 - \alpha + \lambda=0}.$
        
            If $\lambda \neq 0$, then we get two distinct representatives $\La \nb 1 - \frac{1\mp\sqrt{1-4\lb}}2\nb 3-\nb 5+\frac{1\pm\sqrt{1-4\lb}}2\nb 6+\nb 7\Ra.$
            
            If $\lambda = 0$, then we get two distinct representatives $\La \nb 1 - \nb 3-\nb 5 +\nb 7\Ra$ and $\La \nb 1 -\nb 5+ \nb 6+\nb 7\Ra.$
            
            \item If $\alpha_1 = 0,$ then we get the same representatives but without $\nabla_1.$
    \end{itemize}       

    	   \end{enumerate}
	    
	    \item $\af_6^2 - \af_6\af_7 + \lb\af_7^2=0$ and $\af_5=0$. We may assume that $\af_3=0$, since the case $\af_5=0$ and $\af_3\ne 0$ was considered in 1(b). Then $\af^*_3=\af^*_5=0$. A suitable choice of $x,y$ and $z$ gives $\af^*_7=1$ and $\af^*_2=0$, so we shall assume that $\af_7=1$ and $\af_2=0$. As above, $\lb=\af_6-\af_6^2$. Now, choosing $z=-\frac{\af_1xy + (\af_4\af_6^2 - \af_4\af_6)y^2}{x+(\af_6 - 1)y}$ we obtain $\af^*_2=0$.
	    
	    \begin{enumerate}
	    	\item $\af_6\ne 1$. Then choosing $u=-\frac{\af_4x^2 + (\af_1 - 2\af_4)xy - (\af_1 - \af_4)y^2}{(\af_6 - 1)(x + (\af_6 - 1)y)}$ we have $\af^*_4=0$ and
	    	\begin{align*}
	    		\af^*_1&=(\af_1-\af_4\af_6)(x-\af_6y)^2,\\
	    		\af^*_6 &= \af_6(x - \af_6y)(x + (\af_6-1)y)^2,\\
	    		\af^*_7 &= (x - \af_6y)(x + (\af_6-1)y)^2.
	    	\end{align*}
	    	\begin{enumerate}
	    		\item $\af_1-\af_4\af_6\ne 0$. Then choosing $y=0$ and $x=\af_1-\af_4\af_6$ we obtain the representative $\la\nb 1+\af\nb 6+\nb 7\ra$, where $\af^2 - \af + \lb=0$. If $\lb\not\in\{0,\frac 14\}$, then we have two representatives $\La\nb 1+\frac{1\pm\sqrt{1-4\lb}}2\nb 6+\nb 7\Ra$. We will join them with the families $\La\nb 1+\frac{1\pm\sqrt{1-4\lb}}2\af\nb 3+\af\nb 5+\frac{1\pm\sqrt{1-4\lb}}2\nb 6+\nb 7\Ra_{\af\ne 0}$found above. If $\lb=\frac 14$, then we have the representative $\la\nb 1+\frac 12\nb 6+\nb 7\ra$, which will be joined with the family $\la\nb 1+\frac 12\af\nb 3+\af\nb 5+\frac 12\nb 6+\nb 7\ra_{\af\ne 0}$ found above. If $\lb=0$, then $\af_6=0$, since $\af_6\ne 1$ by assumption. So, we have the representative $\la\nb 1+\nb 7\ra$. It will be joined with $\la\nb 1+\af\nb 5+\nb 7\ra_{\af\ne 0}$.
	    		
	    		\item $\af_1-\af_4\af_6=0$. Then we have the same representatives as above, but without $\nb 1$. We join them with the families found above. 
	    	\end{enumerate}
    	
    		\item $\af_6=1$, so that $\lb =0$. Then 
    		\begin{align*}
    			\af^*_1 &=\af_1x(x - y),\\
    			\af^*_4 &=(\af_4(x-y) + \af_1y)(x - y),\\
    			\af^*_6 &= x^2(x - y),\\
    			\af^*_7 &= x^2(x - y).
    		\end{align*}
    		\begin{enumerate}
    			\item $\af_1\ne 0$ and $\af_1-\af_4\ne 0$. Then choosing $x=\af_1$ and $y=-\frac{\af_1\af_4}{\af_1-\af_4}$ we obtain the representative $\la\nb 1+\nb 6+\nb 7\ra$ which will be joined with the family $\la\nb 1+\af\nb 3+\af\nb 5+\nb 6+\nb 7\ra_{\af\ne 0}$ found above.
    			
    			\item $\af_1=0$. Then we have two representatives $\la\nb 4+\nb 6+\nb 7\ra$ and $\la\nb 6+\nb 7\ra$ depending on whether $\af_4=0$ or not. Both of them belong to the families found above.
    			
    			\item $\af_1-\af_4=0$. Then
    			\begin{align*}
    			\af^*_1 &=\af_1x(x - y),\\
    			\af^*_4 &=\af_1x(x - y),\\
    			\af^*_6 &= x^2(x - y),\\
    			\af^*_7 &= x^2(x - y).
    			\end{align*}
    			Thus, we have two representatives $\la\nb 1+\nb 4+\nb 6+\nb 7\ra$ and $\la\nb 6+\nb 7\ra$. The second one was found above.
    		\end{enumerate}
	    \end{enumerate}
	\end{enumerate}
	
        \item $\alpha_7=0$. We may assume that $\af_6=0$, since the case $\af_7=0$ and $\af_6\ne 0$ was considered in 1(a).
        \begin{enumerate}
            \item Let $\alpha_3^2 - \alpha_3\alpha_5 + \lambda\alpha_5^2 \neq 0.$ If $\alpha_5 \neq 0,$ we may take $x = \frac{(\alpha_5 - \alpha_3)y}{\alpha_5}$ and get $\af^*_5=0$. Observe that $(\af^*_3)^2 - \af^*_3\af^*_5 + \lb(\af^*_5)^2=(\af_3^2 - \af_3\af_5 + \lb\af_5^2)(x^2 - xy + \lb y^2)^3$, so the condition $\af_3^2 - \af_3\af_5 + \lb\af_5^2\ne 0$ is invariant under the automorphisms. Thus, we shall assume that $\af_5=0$ and $\af_3\ne 0$ from the very beginning. Choosing $y=0$, we have
            \begin{align*}
\af^*_1 &= x(\alpha_1x + \alpha_3z),\\
\af^*_2 &= x(\alpha_2x + \alpha_3u),\\
\af^*_3 &= \alpha_3x^3,\\
\af^*_4 &= \alpha_4x^2.
\end{align*}
Taking $z = -\frac{\alpha_1x}{\alpha_3}, u =  -\frac{\alpha_2x}{\alpha_3},$ we get two representatives $\langle \nabla_3 + \nabla_4 \rangle$ and $\langle \nabla_3 \rangle$ depending on whether $\af_4=0$ or not. 
    \item Let $\alpha_3^2 - \alpha_3\alpha_5 + \lambda\alpha_5^2 = 0, \alpha_5 \neq 0.$ Taking $x,y$ such that $(x^2-xy+\lambda y^2)(\alpha_5x + (\alpha_3-\alpha_5)y) = 1,$ we may suppose that $\alpha_5^* = 1$ and $\lambda = \alpha_3 - \alpha_3^2.$ Consider $\alpha_2^* = 0, \alpha_4^* = 0$ as a linear system in $u, z.$ Its determinant is $\alpha_3(x + (\alpha_3-1)y)^2.$ If $\alpha_3 = 0,$ then $\lambda = 0$ and we get a Leibniz cocycle. Therefore, we may suppose that this determinant is nonzero and the system $\alpha_2^* = 0, \alpha_4^* = 0$ has a unique solution. We get the family of representatives $\langle \alpha_1\nabla_1 + \alpha_3\nabla_3 + \nabla_5 \rangle,$ where $\alpha_3^2 - \alpha_3 + \lambda = 0,$ except for $(\lambda,\alpha_3) = (0,0)$ which gives a Leibniz cocycle. Therefore, we may suppose that $\alpha_2 = \alpha_4 = 0$ from the very beginning and
    \begin{align*}
    \af^*_1 &= \alpha_1(x-\alpha_3y)^2,\\
    \af^*_3 &= \alpha_3(x-\alpha_3y)(x+(\alpha_3-1)y)^2,\\
    \af^*_5 &= (x-\alpha_3y)(x+(\alpha_3-1)y)^2.
    \end{align*}
    If $\lb\not\in\{0,\frac 14\}$, then we have the following representatives: $\La \nabla_1 +  \frac{1\pm\sqrt{1-4\lambda}}{2}\nabla_3 + \nabla_5 \Ra$ and $\La \frac{1\pm\sqrt{1-4\lambda}}{2}\nabla_3 + \nabla_5 \Ra$. If $\lb=\frac 14$, then we have the representatives $\la\nb 1+\frac 12\nb 3+\nb 5\ra$ and $\la\frac 12\nb 3+\nb 5\ra$.  If $\lb=0$, then we have the representatives $\langle \nabla_1 + \nabla_3 + \nabla_5 \rangle$ and  $\langle \nabla_3 + \nabla_5 \rangle$. 
        \end{enumerate}
    \end{enumerate}
%\newpage

	\begin{landscape}
	\tiny{
	\begin{tabular}{|c|c|c|}
	    \hline
	    $\lb=0$                      & $\lb=\frac 14$                & $\lb\not\in\{0,\frac 14\}$\\
	    \hline
	    $\begin{array}{l}
	       \la\nb 1+\af\nb 3+\bt\nb 5+\nb 6\ra_{\af,\bt\ne 0},\\
	       \la\af\nb 3+\bt\nb 5+\nb 6\ra_{\af,\bt\ne 0},\\
	       \la\nb 2+\af\nb 5+\nb 6\ra_{\af\ne 1},\\
	       \la\af\nb 5+\nb 6\ra,\\
	       \la\nb 2+\af\nb 4+\nb 5+\nb 6\ra,\\
	       \la\nb 4+\nb 5+\nb 6\ra,\\
	       \la \nb 2+\af\nb 3+\nb 6\ra_{\af\ne 0},\\
	       \la \af\nb 3+\nb 6\ra_{\af\ne 0},\\
	       \la\nb 1+\af\nb 2-\nb 3+\nb 6\ra,\\
	       \la\nb 1+\af\nb 3+\nb 7\ra_{\af\ne 0},\\
	       \la\af\nb 3+\nb 7\ra_{\af\ne 0},\\
	       \la\af\nb 3+\nb 4+\nb 6+\nb 7\ra,\\
	       \la\af\nb 3+\nb 6+\nb 7\ra,\\
	       \la\nb 1-\nb 3+\af\nb 4+\nb 6+\nb 7\ra,\\
	       \la\nb 1+\af\nb 5+\nb 7\ra,\\
	       \la\nb 1+\af\nb 3+\af\nb 5+\nb 6+\nb 7\ra,\\
	       \la\af\nb 5+\nb 7\ra,\\
	       \la\af\nb 3+\af\nb 5+\nb 6+\nb 7\ra_{\af\ne 0},\\
	       \la\nb 1+\nb 2-\nb 3-\nb 5+\nb 6+\nb 7\ra,\\
	       \la\nb 2-\nb 5+\nb 7\ra,\\
	       \la\nb 2-\nb 3-\nb 5+\nb 6+\nb 7\ra,\\
	       \la\nb 1+\nb 2-\nb 5+\nb 7\ra,\\
	       \la\nb 1-\nb 3-\nb 5+\nb 7\ra,\\
	       \la\nb 1-\nb 5+\nb 6+\nb 7\ra,\\
	       \la-\nb 3-\nb 5+\nb 7\ra,\\
	       \la-\nb 5+\nb 6+\nb 7\ra,\\
	       \la\nb 1+\nb 4+\nb 6+\nb 7\ra,\\
	       \la \nb 3 + \nb 4 \ra,\\
	       \la \nb 3 \ra,\\
	       \la \nb 1 + \nb 3 + \nb 5 \ra,\\
	       \la \nb 3 + \nb 5 \ra
	    \end{array}$
	    &
	    $\begin{array}{l}
	         \la\nb 1+\af\nb 3+\bt\nb 5+\nb 6\ra_{\af,\bt\ne 0},\\
	         \la\af\nb 3+\bt\nb 5+\nb 6\ra_{\af,\bt\ne 0},\\
	         \la\nb 2+\af\nb 5+\nb 6\ra_{\af\ne 2},\\
	         \la\af\nb 5+\nb 6\ra,\\
	         \la\nb 2+\af\nb 4+2\nb 5+\nb 6\ra,\\
	         \la\nb 4+2\nb 5+\nb 6\ra,\\
	         \la \nb 2+\af\nb 3+\nb 6\ra_{\af\ne 0},\\
	         \la \af\nb 3+\nb 6\ra_{\af\ne 0},\\
	         \la\nb 1+\af\nb 2-\nb 3+\nb 6\ra,\\
	         \la\nb 1+\af\nb 3+\frac 12\nb 6+\nb 7\ra_{\af\ne 0},\\
	         \la\af\nb 3+\frac 12\nb 6+\nb 7\ra_{\af\ne 0},\\
	         \la\nb 1+\frac 12\af\nb 3+\af\nb 5+\frac 12\nb 6+\nb 7\ra,\\
	         \la\frac 12\af\nb 3+\af\nb 5+\frac 12\nb 6+\nb 7\ra,\\
	         \la \frac 12\nb 1+ \nb 2 -\frac 12\nb 3-\nb 5+\frac 12\nb 6+\nb 7\ra,\\
	        % \la\frac 12\nb 1+\nb 2-\frac 12\nb 3-\nb 5+\frac 12\nb 6+\nb 7\ra,\\
	         \la \nb 3 + \nb 4 \ra,\\
	         \la \nb 3 \ra,\\
	         \la\nb 1+\frac 12\nb 3+\nb 5\ra,\\
	         \la\frac 12\nb 3+\nb 5\ra
	    \end{array}$
	    &
	   $\begin{array}{l}
	       \la\nb 1+\af\nb 3+\bt\nb 5+\nb 6\ra_{\af,\bt\ne 0},\\
	       \la\af\nb 3+\bt\nb 5+\nb 6\ra_{\af,\bt\ne 0},\\
	       \la\nb 2+\af\nb 5+\nb 6\ra_{\af\ne\frac{1\pm\sqrt{1-4\lb}}{2\lb}},\\
	       \la\af\nb 5+\nb 6\ra,\\
	       \La \lambda\nb 2+\lambda\af\nb 4+\frac{1+\sqrt{1-4\lb}}{2}\nb 5+\lambda\nb 6\Ra,\\
	       %\La \nb 2+\af\nb 4+\frac{1+\sqrt{1-4\lb}}{2\lb}\nb 5+\nb 6\Ra,\\
	       \La\nb 2+\af\nb 4+\frac{2}{1+\sqrt{1-4\lb}}\nb 5+\nb 6\Ra,\\
	       %\La \nb 2+\af\nb 4+\frac{1+\sqrt{1-4\lb}}{2\lb}\nb 5+\nb 6\Ra,\\
	       \La\lambda \nb 4+\frac{1+\sqrt{1-4\lb}}{2}\nb 5+\lambda \nb 6\Ra,\\
	       \La\nb 4+\frac{2}{1+\sqrt{1-4\lb}}\nb 5+\nb 6\Ra,\\
	       \la \nb 2+\af\nb 3+\nb 6\ra_{\af\ne 0},\\
	       \la \af\nb 3+\nb 6\ra_{\af\ne 0},\\
	       \la\nb 1+\af\nb 2-\nb 3+\nb 6\ra,\\
	       \La \nb 1+\af\nb 3+ \frac{1+\sqrt{1-4\lb}}{2}\nb 6+\nb 7\Ra_{\af\ne 0},\\
	       \La \nb 1+\af\nb 3+ \frac{1-\sqrt{1-4\lb}}{2}\nb 6+\nb 7\Ra_{\af\ne 0},\\
	       \La \af\nb 3+ \frac{1+\sqrt{1-4\lb}}{2}\nb 6+\nb 7\Ra_{\af\ne 0},\\
	       \La \af\nb 3+ \frac{1-\sqrt{1-4\lb}}{2}\nb 6+\nb 7\Ra_{\af\ne 0},\\	       
	       \La\nb 1+\frac{1+\sqrt{1-4\lb}}2\af\nb 3+\af\nb 5+\frac{1+\sqrt{1-4\lb}}2\nb 6+\nb 7\Ra,\\
	       \La\nb 1+\frac{1-\sqrt{1-4\lb}}2\af\nb 3+\af\nb 5+\frac{1-\sqrt{1-4\lb}}2\nb 6+\nb 7\Ra,\\
	       \La\frac{1+\sqrt{1-4\lb}}2\af\nb 3+\af\nb 5+\frac{1+\sqrt{1-4\lb}}2\nb 6+\nb 7\Ra,\\
	       \La\frac{1-\sqrt{1-4\lb}}2\af\nb 3+\af\nb 5+\frac{1-\sqrt{1-4\lb}}2\nb 6+\nb 7\Ra,\\
	       \La (1-2\lb)\nb 1+ \nb 2 + \frac{-1+\sqrt{1-4\lb}}2\nb 3+\frac{\sqrt{1-4\lb}- 1}{\sqrt{1-4\lb}+ 1}\nb 5+\frac{1+\sqrt{1-4\lb}}2\nb 6+\nb 7 \Ra,\\
	       \La (1-2\lb)\lb\nabla_1 + \lb\nabla_2 -\lb\frac{1+\sqrt{1-4\lb}}2\nb 3-\Big(\frac{\sqrt{1-4\lb}+ 1}{2}\Big)^2\nb 5+\lb\frac{1-\sqrt{1-4\lb}}2\nb 6+\lb\nb 7 \Ra\\
	       %\La (1-2\lb)\nb 1+ \nb 2 -\frac{1+\sqrt{1-4\lb}}2\nb 3+\frac{\sqrt{1-4\lb}+ 1}{\sqrt{1-4\lb}- 1}\nb 5+\frac{1-\sqrt{1-4\lb}}2\nb 6+\nb 7 \Ra,\\
	       \La \frac{1+\sqrt{1-4\lb}}2\nb 1+ \nb 2 + \frac{-1+\sqrt{1-4\lb}}2\nb 3+\frac{\sqrt{1-4\lb}- 1}{\sqrt{1-4\lb}+ 1}\nb 5+\frac{1+\sqrt{1-4\lb}}2\nb 6+\nb 7 \Ra,\\
	       \La \lb\frac{1-\sqrt{1-4\lb}}2\nabla_1 + \lb\nabla_2 -\lb\frac{1+\sqrt{1-4\lb}}2\nb 3-\Big(\frac{\sqrt{1-4\lb}+ 1}{2}\Big)^2\nb 5+\lb\frac{1-\sqrt{1-4\lb}}2\nb 6+\lb\nb 7 \Ra\\
	       %\La \frac{1-\sqrt{1-4\lb}}2\nb 1+ \nb 2 -\frac{1+\sqrt{1-4\lb}}2\nb 3+\frac{\sqrt{1-4\lb}+ 1}{\sqrt{1-4\lb}- 1}\nb 5+\frac{1-\sqrt{1-4\lb}}2\nb 6+\nb 7 \Ra,\\
	       \La\frac{1+\sqrt{1-4\lb}}2\nb 1+\nb 2-\frac{1+\sqrt{1-4\lb}}2\nb 3-\nb 5+\frac{1+\sqrt{1-4\lb}}2\nb 6+\nb 7\Ra,\\
	       \La\frac{1-\sqrt{1-4\lb}}2\nb 1+\nb 2-\frac{1-\sqrt{1-4\lb}}2\nb 3-\nb 5+\frac{1-\sqrt{1-4\lb}}2\nb 6+\nb 7\Ra,\\
	       \La\frac{1-\sqrt{1-4\lb}}2\nb 1+\nb 2-\frac{1+\sqrt{1-4\lb}}2\nb 3-\nb 5+\frac{1+\sqrt{1-4\lb}}2\nb 6+\nb 7\Ra,\\
	       \La\frac{1+\sqrt{1-4\lb}}2\nb 1+\nb 2-\frac{1-\sqrt{1-4\lb}}2\nb 3-\nb 5+\frac{1-\sqrt{1-4\lb}}2\nb 6+\nb 7\Ra,\\
	       \La\nb 1-\frac{1-\sqrt{1-4\lb}}2\nb 3-\nb 5+\frac{1+\sqrt{1-4\lb}}2\nb 6+\nb 7\Ra,\\
	       \La\nb 1-\frac{1+\sqrt{1-4\lb}}2\nb 3-\nb 5+\frac{1-\sqrt{1-4\lb}}2\nb 6+\nb 7\Ra,\\
	       \La-\frac{1-\sqrt{1-4\lb}}2\nb 3-\nb 5+\frac{1+\sqrt{1-4\lb}}2\nb 6+\nb 7\Ra,\\
	       \La-\frac{1+\sqrt{1-4\lb}}2\nb 3-\nb 5+\frac{1-\sqrt{1-4\lb}}2\nb 6+\nb 7\Ra,\\
	       \la \nb 3 + \nb 4 \ra,\\
	       \la \nb 3 \ra,\\
	       \La \nb 1 + \frac{1+\sqrt{1-4\lambda}}{2}\nb 3 + \nb 5 \Ra,\\
	       \La \nb 1 + \frac{1-\sqrt{1-4\lambda}}{2}\nb 3 + \nb 5 \Ra,\\
	       \La \frac{1+\sqrt{1-4\lambda}}{2}\nb 3 + \nb 5 \Ra,\\
	       \La \frac{1-\sqrt{1-4\lambda}}{2}\nb 3 + \nb 5 \Ra.
	       \end{array}$
	    \\
	    \hline
	\end{tabular}
	}

	\end{landscape}

\newpage

Denote $\Theta=\frac{1+\sqrt{1-4\lambda}}{2}.$ The orbits above correspond to the following algebras:	

{\tiny
\[\begin{array}{lllllllllll}

        \D{4}{01}(\lambda,\alpha,\beta)&:& e_1 e_1 = \lambda e_3 + e_4, & e_1 e_3 = \alpha e_4, & e_2 e_1=e_3, & e_2 e_2 = e_3, & e_2 e_3 = \beta e_4, & e_3e_1 = e_4; \\
	    \D{4}{02}(\lambda,\alpha,\beta)&:& e_1 e_1 = \lambda e_3, & e_1 e_3 = \alpha e_4, & e_2 e_1=e_3 & e_2 e_2 = e_3, & e_2 e_3 = \beta e_4, & e_3e_1 = e_4; \\
	    \D{4}{03}(\lambda,\alpha)&:& e_1 e_1 = \lambda e_3, & e_1 e_2 = e_4, & e_2 e_1=e_3, & e_2 e_2 = e_3, & e_2 e_3 = \alpha e_4, & e_3e_1 = e_4; \\
	    \D{4}{04}(\lambda,\alpha)&:& e_1 e_1 = \lambda e_3, & e_2 e_1=e_3, & e_2 e_2 = e_3, & e_2 e_3 = \alpha e_4, & e_3e_1 = e_4; \\
	    \D{4}{05}(\lambda,\alpha)&:& e_1 e_1 = \lambda e_3, & e_1 e_2 = \lambda e_4, & e_2 e_1=e_3 + \lambda\alpha e_4, & e_2 e_2 = e_3, & e_2 e_3 = 	\Theta e_4, & e_3e_1 = \lambda e_4; \\
	    %\D{4}{05}(\lambda,\alpha)&:& e_1 e_1 = \lambda e_3, & e_1 e_2 = e_4, & e_2 e_1=e_3 + \alpha e_4, & e_2 e_2 = e_3, & e_2 e_3 = 	\frac{1+\sqrt{1-4\lb}}{2\lb} e_4, & e_3e_1 = e_4; \\
	    \D{4}{06}(\lambda,\alpha)&:& e_1 e_1 = \lambda e_3, & e_1 e_2 = e_4, & e_2 e_1=e_3 + \alpha e_4, & e_2 e_2 = e_3, & e_2 e_3 = 	\Theta^{-1} e_4, & e_3e_1 = e_4; \\
	    %\D{4}{06}(\lambda,\alpha)&:& e_1 e_1 = \lambda e_3, & e_1 e_2 = e_4, & e_2 e_1=e_3 + \alpha e_4, & e_2 e_2 = e_3, & e_2 e_3 = 	\frac{1-\sqrt{1-4\lb}}{2\lb} e_4, & e_3e_1 = e_4; \\
	    \D{4}{07}(\lambda)&:& e_1 e_1 = \lambda e_3, & e_2 e_1=e_3 + \lambda e_4, & e_2 e_2 = e_3, & e_2 e_3 = 	\Theta e_4, & e_3e_1 = \lambda e_4; \\
	    %\D{4}{07}(\lambda)&:& e_1 e_1 = \lambda e_3, & e_2 e_1=e_3 + e_4, & e_2 e_2 = e_3, & e_2 e_3 = 	\frac{1+\sqrt{1-4\lb}}{2\lb} e_4, & e_3e_1 = e_4; \\
	    \D{4}{08}(\lambda)&:& e_1 e_1 = \lambda e_3, &  e_2 e_1=e_3 + e_4, & e_2 e_2 = e_3, & e_2 e_3 = 	\Theta^{-1} e_4, & e_3e_1 = e_4; \\
	    %\D{4}{08}(\lambda)&:& e_1 e_1 = \lambda e_3, &  e_2 e_1=e_3 + e_4, & e_2 e_2 = e_3, & e_2 e_3 = 	\frac{1-\sqrt{1-4\lb}}{2\lb} e_4, & e_3e_1 = e_4; \\
	    \D{4}{09}(\lambda,\alpha)&:& e_1 e_1 = \lambda e_3, & e_1 e_2 = e_4,& e_1 e_3 = \alpha e_4, & e_2 e_1=e_3, & e_2 e_2 = e_3, & e_3e_1 = e_4; \\
	    \D{4}{10}(\lambda,\alpha)&:& e_1 e_1 = \lambda e_3,& e_1 e_3 = \alpha e_4, &  e_2 e_1=e_3, & e_2 e_2 = e_3, & e_3e_1 = e_4; \\
        \D{4}{11}(\lambda,\alpha)&:& e_1 e_1 = \lambda e_3 + e_4,& e_1e_2 = \alpha e_4, & e_1 e_3 = -e_4, &  e_2 e_1=e_3, & e_2 e_2 = e_3, & e_3e_1 = e_4; \\
        \D{4}{12}(\lambda,\alpha)&:& e_1 e_1 = \lambda e_3 + e_4,& e_1e_3 = \alpha e_4,  &  e_2 e_1=e_3, & e_2 e_2 = e_3, & e_3e_1 = \Theta e_4, & e_3e_2 = e_4; \\
        \D{4}{13}(\lambda,\alpha)&:& e_1 e_1 = \lambda e_3 + e_4,& e_1e_3 = \alpha e_4,  &  e_2 e_1=e_3, & e_2 e_2 = e_3, & e_3e_1 = (1-\Theta)e_4, & e_3e_2 = e_4; \\
        \D{4}{14}(\lambda,\alpha)&:& e_1 e_1 = \lambda e_3,& e_1e_3 = \alpha e_4,  &  e_2 e_1=e_3, & e_2 e_2 = e_3, & e_3e_1 = \Theta e_4, & e_3e_2 = e_4; \\
        \D{4}{15}(\lambda,\alpha)&:& e_1 e_1 = \lambda e_3,& e_1e_3 = \alpha e_4,  &  e_2 e_1=e_3, & e_2 e_2 = e_3, & e_3e_1 = (1-\Theta)e_4, & e_3e_2 = e_4; \\
        \D{4}{16}(\alpha)&:& e_1e_3 = \alpha e_4,  &  e_2 e_1=e_3 + e_4, & e_2 e_2 = e_3, & e_3e_1 = e_4, & e_3e_2 = e_4; \\
        \D{4}{17}(\alpha)&:& e_1 e_1 = e_4, & e_1e_3 = -e_4,  &  e_2 e_1=e_3 + \alpha e_4, & e_2 e_2 = e_3, & e_3e_1 = e_4, & e_3e_2 = e_4; \\
        \D{4}{18}(\lambda,\alpha)&:& e_1 e_1 = \lambda e_3 + e_4,& e_1e_3 = \Theta\alpha e_4,  &  e_2 e_1=e_3, & e_2 e_2 = e_3, & e_2 e_3 = \alpha e_4, & e_3e_1 = \Theta e_4,\\
        && e_3e_2 = e_4; \\
        \D{4}{19}(\lambda,\alpha)&:& e_1 e_1 = \lambda e_3 + e_4,& e_1e_3 = (1-\Theta)\alpha e_4,  &  e_2 e_1=e_3, & e_2 e_2 = e_3, & e_2 e_3 = \alpha e_4, & e_3e_1 = (1-\Theta)e_4,
        \\&& e_3e_2 = e_4; \\
        \D{4}{20}(\lambda,\alpha)&:& e_1 e_1 = \lambda e_3,& e_1e_3 = \Theta\alpha e_4,  &  e_2 e_1=e_3, & e_2 e_2 = e_3, & e_2 e_3 = \alpha e_4, & e_3e_1 = \Theta e_4,
        \\&& e_3e_2 = e_4; \\
        \D{4}{21}(\lambda,\alpha)&:& e_1 e_1 = \lambda e_3,& e_1e_3 = (1-\Theta)\alpha e_4,  &  e_2 e_1=e_3, & e_2 e_2 = e_3, & e_2 e_3 = \alpha e_4, & e_3e_1 = (1-\Theta)e_4, \\&& e_3e_2 = e_4; \\
        \D{4}{22}(\lambda)&:& e_1 e_1 = \lambda e_3 + (1-2\lambda)e_4,& e_1 e_2 = e_4,& e_1e_3 = (\Theta - 1) e_4,  &  e_2 e_1=e_3, & e_2 e_2 = e_3, & e_2 e_3 = (1-\Theta^{-1}) e_4,
        \\&& e_3e_1 = \Theta e_4, & e_3e_2 = e_4; \\
        \D{4}{23}(\lambda)&:& e_1 e_1 = \lambda e_3 + \lambda(1-2\lambda) e_4& e_1 e_2 = \lambda e_4& e_1e_3 = -\lambda\Theta e_4 &  e_2 e_1=e_3 & e_2 e_2 = e_3 & e_2 e_3 = -\Theta^2 e_4, 
        \\&& e_3e_1 = \lambda(1-\Theta)e_4 & e_3e_2 = \lambda e_4; \\
        \D{4}{24}(\lambda)&:& e_1 e_1 = \lambda e_3 + \Theta e_4,& e_1 e_2 = e_4,& e_1e_3 = (\Theta - 1) e_4,  &  e_2 e_1=e_3, & e_2 e_2 = e_3, & e_2 e_3 = (1-\Theta^{-1}) e_4, 
        \\&& e_3e_1 = \Theta e_4, & e_3e_2 = e_4; \\
        \D{4}{25}(\lambda)&:& e_1 e_1 = \lambda e_3 + \lambda(1-\Theta)e_4,& e_1 e_2 = \lambda e_4,& e_1e_3 = -\lambda\Theta e_4,  &  e_2 e_1=e_3, & e_2 e_2 = e_3, & e_2 e_3 = -\Theta^2 e_4, 
        \\&& e_3e_1 = \lambda (1-\Theta)e_4, & e_3e_2 = \lambda e_4; \\
        \D{4}{26}(\lambda)&:& e_1 e_1 = \lambda e_3 + \Theta e_4,& e_1 e_2 = e_4,& e_1e_3 = -\Theta e_4,  &  e_2 e_1=e_3, & e_2 e_2 = e_3, & e_2 e_3 = -e_4,
        \\&& e_3e_1 = \Theta e_4, & e_3e_2 = e_4; \\
        \D{4}{27}(\lambda)&:& e_1 e_1 = \lambda e_3 + (1-\Theta)e_4,& e_1 e_2 = e_4,& e_1e_3 = (\Theta-1) e_4,  &  e_2 e_1=e_3, & e_2 e_2 = e_3, & e_2 e_3 = -e_4,
        \\&& e_3e_1 = (1-\Theta)e_4, & e_3e_2 = e_4; \\
        \D{4}{28}(\lambda)&:& e_1 e_1 = \lambda e_3 + (1-\Theta)e_4,& e_1 e_2 = e_4,& e_1e_3 = -\Theta e_4,  &  e_2 e_1=e_3, & e_2 e_2 = e_3, & e_2 e_3 = -e_4, 
        \\&& e_3e_1 = \Theta e_4, & e_3e_2 = e_4; \\
        \D{4}{29}(\lambda)&:& e_1 e_1 = \lambda e_3 + \Theta e_4,& e_1 e_2 = e_4,& e_1e_3 = (\Theta-1) e_4,  &  e_2 e_1=e_3, & e_2 e_2 = e_3, & e_2 e_3 = -e_4, 
        \\&& e_3e_1 = (1-\Theta)e_4, & e_3e_2 = e_4; \\
        \D{4}{30}(\lambda)&:& e_1 e_1 = \lambda e_3 + e_4, & e_1e_3 = (\Theta-1) e_4,  &  e_2 e_1=e_3, & e_2 e_2 = e_3, & e_2 e_3 = -e_4, & e_3e_1 = \Theta e_4,
        \\&& e_3e_2 = e_4; \\
        \D{4}{31}(\lambda)&:& e_1 e_1 = \lambda e_3 + e_4, & e_1e_3 = -\Theta e_4,  &  e_2 e_1=e_3, & e_2 e_2 = e_3, & e_2 e_3 = -e_4, & e_3e_1 = (1-\Theta)e_4,
        \\&& e_3e_2 = e_4; \\
        \D{4}{32}(\lambda)&:& e_1 e_1 = \lambda e_3, & e_1e_3 = (\Theta-1) e_4,  &  e_2 e_1=e_3, & e_2 e_2 = e_3, & e_2 e_3 = -e_4, & e_3e_1 = \Theta e_4, 
        \\&& e_3e_2 = e_4; \\
        \D{4}{33}(\lambda)&:& e_1 e_1 = \lambda e_3, & e_1e_3 = -\Theta e_4,  &  e_2 e_1=e_3, & e_2 e_2 = e_3, & e_2 e_3 = -e_4, & e_3e_1 = (1-\Theta)e_4, 
        \\&& e_3e_2 = e_4; \\
       \D{4}{34}&:& e_1 e_1 = e_4, &  e_2 e_1=e_3 + e_4, & e_2 e_2 = e_3, & e_3 e_1 = e_4,& e_3 e_2 = e_4; \\
        \D{4}{35}(\lambda)&:& e_1 e_1 = \lambda e_3, & e_1e_3 = e_4,  &  e_2 e_1=e_3 + e_4, & e_2 e_2 = e_3; \\
        \D{4}{36}(\lambda)&:& e_1 e_1 = \lambda e_3, & e_1e_3 = e_4,  &  e_2 e_1=e_3, & e_2 e_2 = e_3; \\
        \D{4}{37}(\lambda)&:& e_1 e_1 = \lambda e_3 + e_4, & e_1e_3 = \Theta e_4, &  e_2 e_1=e_3, & e_2 e_2 = e_3,&  e_2 e_3 = e_4; \\
        \D{4}{38}(\lambda)&:& e_1 e_1 = \lambda e_3 + e_4, & e_1e_3 = (1-\Theta)e_4, &  e_2 e_1=e_3, & e_2 e_2 = e_3,&  e_2 e_3 = e_4; \\
        \D{4}{39}(\lambda)&:& e_1 e_1 = \lambda e_3, & e_1e_3 = \Theta e_4, &  e_2 e_1=e_3, & e_2 e_2 = e_3,&  e_2 e_3 = e_4; \\
        \D{4}{40}(\lambda)&:& e_1 e_1 = \lambda e_3, & e_1e_3 = (1-\Theta)e_4, &  e_2 e_1=e_3, & e_2 e_2 = e_3,&  e_2 e_3 = e_4.
\end{array}\]}

The algebras above are pairwise non-isomorphic, except for {\tiny \begin{gather*}
\D{4}{01}(\lambda,0,\beta) \cong \D{4}{02}(\lambda,0,\beta) \cong \D{4}{04}(\lambda,\beta),\quad \D{4}{01}(\lambda,\alpha,0)_{\alpha \neq -1} \cong \D{4}{02}(\lambda,\alpha,0) \cong \D{4}{10}(\lambda,\alpha),\quad \D{4}{01}(\lambda,-1,0) \cong \D{4}{11}(\lambda,0),\quad\\
\D{4}{03}(\lambda,0) \cong \D{4}{09}(\lambda,0),\quad \D{4}{03}\left(\lambda,(1-\Theta)^{-1}\right)_{\lambda \neq 0} \cong \D{4}{05}(\lambda,0)_{\lambda \neq 0}, \D{4}{03}\left(\lambda,\Theta^{-1}\right)\cong \D{4}{06}(\lambda,0),\quad \D{4}{04}(\lambda,0) \cong \D{4}{10}(\lambda,0),\quad\\
\D{4}{05}(1/4,\alpha) \cong \D{4}{06}(1/4,\alpha),\quad \D{4}{07}(1/4) \cong \D{4}{08}(1/4),\quad\\
\D{4}{05}(0,\alpha) \cong \D{4}{07}(0) \cong \D{4}{23}(0) \cong  \D{4}{25}(0) \cong  \D{4}{40}(0),\quad\\
\D{4}{12}(\lambda,0) \cong \D{4}{18}(\lambda,0),\quad  \D{4}{12}(1/4,\alpha) \cong \D{4}{13}(1/4,\alpha),\quad \D{4}{12}(0,\alpha)_{\alpha \neq -1} \cong \D{4}{14}(0,\alpha),\quad \D{4}{12}(0,-1) \cong \D{4}{17}(0),\quad\\
\D{4}{13}(\lambda,0) \cong \D{4}{19}(\lambda,0),\quad \D{4}{14}(\lambda,0) \cong \D{4}{20}(\lambda,0),\quad \D{4}{14}(1/4,\alpha) \cong \D{4}{15}(1/4,\alpha),\quad \D{4}{15}(\lambda,0) \cong \D{4}{21}(\lambda,0),\quad\\
\D{4}{18}(1/4,\alpha) \cong \D{4}{19}(1/4,\alpha),\quad \D{4}{18}(0,0) \cong \D{4}{22}(0) \cong  \D{4}{24}(0),\quad \D{4}{18}(1/4,-1) \cong \D{4}{19}(1/4,-1) \cong \D{4}{30}(1/4) \cong \D{4}{31}(1/4),\quad\\
\D{4}{20}(1/4,\alpha) \cong \D{4}{21}(1/4,\alpha),\quad \D{4}{20}(1/4,-1) \cong \D{4}{21}(1/4,-1) \cong \D{4}{32}(1/4) \cong \D{4}{33}(1/4),\quad\\
\D{4}{22}(1/4) \cong \D{4}{23}(1/4) \cong \D{4}{24}(1/4) \cong \D{4}{25}(1/4) \cong \D{4}{26}(1/4) \cong \D{4}{27}(1/4) \cong \D{4}{28}(1/4) \cong \D{4}{29}(1/4),\quad\\
 \D{4}{37}(1/4) \cong \D{4}{38}(1/4),\quad \D{4}{39}(1/4) \cong \D{4}{40}(1/4).
\end{gather*}}

Moreover, the algebras $\D{4}{05}(0,\alpha) \cong \D{4}{07}(0) \cong \D{4}{23}(0) \cong  \D{4}{25}(0) \cong  \D{4}{40}(0), \D{4}{38}(0)$ are Leibniz.
%\newpage

\newpage

\subsubsection{$1$-dimensional  central extensions of $\T {3*}{05}$}\label{ext-T_05^3*}
	Let us use the following notations 
	\begin{align*}
	\nb 1 = \Dl 13, \nb 2 = \Dl 21, \nb 3 = \Dl 22 - 3\Dl 31.    
	\end{align*}
	Take $\0=\sum_{i=1}^3\af_i\nb i\in {\rm H_T^2}(\T {3*}{05}).$
	If 
	$$
	\phi=
	\begin{pmatrix}
	x &    0  &  0\\
	y &  x^2  &  0\\
	z &   xy  &  x^3
	\end{pmatrix}\in\aut{\T {3*}{05}},
	$$
	then
	$$
	\phi^T\begin{pmatrix}
	0      &  0    & \af_1\\
	\af_2  & \af_3 & 0\\
	-3\af_3&  0    & 0
	\end{pmatrix} \phi=
	\begin{pmatrix}
	\af^*      &  \af^{**}    & \af^*_1\\
	\af^*_2    & \af^*_3      & 0\\
	-3\af^*_3  &  0           & 0
	\end{pmatrix},
	$$
	where
	\begin{align*}
	\af^*_1 &= \af_1x^4,\\
	\af^*_2 &= x^2(\af_2x - 2\af_3y),\\
	\af^*_3 &= \af_3x^4.
	\end{align*}
	Hence, $\phi\langle\0\rangle=\langle\0^*\rangle,$ where $\0^*=\sum\limits_{i=1}^3 \af_i^*  \nb i.$
	
	We are interested in $\0$ with $(\af_2,\af_3) \neq (0,0).$ Moreover, the condition $\theta \in \mathbf{T}_1 (\T{3}{05})$ gives us $(\alpha_1, \alpha_3) \neq (0,0).$
	
	Let $\theta$ be as above. Consider two mutually exclusive cases:
	
	\begin{enumerate}
	    \item $\alpha_3 = 0.$ The conditions above imply that $\alpha_1 \neq 0, \alpha_2 \neq 0.$
	    Then choosing $x=\frac{\af_2}{\af_1},$ we have the representative $\langle \nabla_1+\nabla_2 \rangle.$      

        \item $\alpha_3 \neq 0.$ Choosing
        $x = \frac{1}{\sqrt[4]{\af_3}}$ and $y=\frac{\af_2}{ 2 \sqrt[4]{\af_3^5}}$ we have the family of representatives  $\langle\alpha\nabla_1 + \nabla_3 \rangle.$ 
\end{enumerate}

Summarizing, we have the following distinct orbits: 
   
\[ \langle \nabla_1+\nabla_2 \rangle, \ \langle\alpha\nabla_1 + \nabla_3 \rangle.\]

They correspond to the following algebras:	        
    
\[\begin{array}{lllllllllll}

        \T{4}{37}&:& e_1e_1 = e_2,& e_1e_2 = e_3,& e_1e_3 = e_4,& e_2e_1 = e_4; \\
	    \T{4}{38}(\alpha)&:& e_1e_1 = e_2,& e_1e_2 = e_3,& e_1e_3 = \alpha e_4,& e_2e_2 = e_4,& e_3e_1 = -3e_4.
	
% govno=37б hueta=38
\end{array}\]

\subsubsection{$1$-dimensional  central extensions of $\T {3}{01}(\lambda)$}

\begin{enumerate}
    \item 	$\lambda \neq 0.$

    Let us use the following notations 
	\begin{align*}
	\nb 1 = \Dl 12, \nb 2 = (\lambda-1)\Dl 13 + 3 \Dl 31, \nb 3 = \red{\Dl 13 + \Dl 22}.    
	\end{align*}
	Take $\0=\sum_{i=1}^3\af_i\nb i\in {\rm H_T^2}(\T {3}{01}).$
	If 
	$$
	\phi=
	\begin{pmatrix}
x &               0  &  0\\
y &             x^2  &  0\\
z &   (\lambda+1)xy  &  x^3
\end{pmatrix}
\in\aut{\T {3}{01}},
	$$
	then
	$$
	\phi^T\begin{pmatrix}
	0      &  \af_1    & (\lambda-1)\af_2\red{+\af_3}\\
	0  & \red{\af_3} & 0\\
	\red{3\af_2}&  0    & 0
	\end{pmatrix} \phi=
	\begin{pmatrix}
	\alpha^*      &  \af_1^* + \lambda\alpha^{**}    & (\lambda-1)\af_2^*\red{+\af_3^*}\\
	\alpha^{**}  & \red{\af_3^*} & 0\\
	\red{3\af_2^*}&  0    & 0
	\end{pmatrix},
	$$
	where
	\begin{align*}
	%\af_1^* &= (\lambda + 1)(\lambda - 1)\af_2x^2y + ((\lambda - 1)\af_3y + \af_1x)x^2 - (3(\lambda + 1)\lambda(\af_2 -\af_3)x^2y + (\lambda - 1)\af_3x^2y)\lambda\\
	%\af^*_1 &= x^2(((\lambda + 1)(\lambda - 1)\af_2 - (\lambda-1)^2\af_3 - 3(\lambda + 1)(\af_2 -\af_3))y + \af_1x),\\
	\af^*_1 &= x^2(\af_1x+({\red 2}\af_3 - (\lb + 1)(2\lb + 1)\af_2)y),\\
	\af^*_2 &= \af_2x^4,\\
	\af^*_3 &= \af_3x^4.
	\end{align*}
	Hence, $\phi\langle\0\rangle=\langle\0^*\rangle,$ where $\0^*=\sum\limits_{i=1}^3 \af_i^*  \nb i.$
	
	The condition $\theta \in \mathbf{T}_1 (\T{3}{01}(\lambda))$ gives us ${\red(\alpha_2, \alpha_3)} \neq (0,0).$
	
		\begin{enumerate}
		\item $(\lb + 1)(2\lb + 1)\ne 0$.
		\begin{enumerate}
			\item $\af_3\ne 0$.
			\begin{enumerate}
				\item $\af_2 \ne \frac{\red 2}{(\lb + 1)(2\lb + 1)}\af_3$. Then choosing $x=\frac 1{\sqrt[4]{\af_3}}$ and $y=-\frac{\af_1x}{{\red 2}\af_3 - (\lb + 1)(2\lb + 1)\af_2}$, we obtain the family of representatives of distinct orbits
				$\la\af\nb 2+\nb 3\ra$, where $\af\ne \frac{\red 2}{(\lb + 1)(2\lb + 1)}$.
				\item $\af_2=\frac{\red 2}{(\lb + 1)(2\lb + 1)}\af_3$. Then we have two representatives 
				$\La\frac{\red 2}{(\lb + 1)(2\lb + 1)}\nb 2+\nb 3\Ra$ and $\La\nb 1+\frac{\red 2}{(\lb + 1)(2\lb + 1)}\nb 2+\nb 3\Ra$ depending on whether $\af_1=0$ or not. For convenience, we shall join the representative $\La\frac{\red 2}{(\lb + 1)(2\lb + 1)}\nb 2+\nb 3\Ra$ with the family $\la\af\nb 2+\nb 3\ra_{\af\ne \frac{\red 2}{(\lb + 1)(2\lb + 1)}}$ found above.  
			\end{enumerate}
			\item $\af_3=0$. Then $\af_2\ne 0$. Choosing $y=\frac{\af_1x}{(\lb + 1)(2\lb + 1)\af_2}$, we obtain the representative $\la\nb 2\ra$.
		\end{enumerate}

	\item  $(\lb + 1)(2\lb + 1)=0$. 
		\begin{enumerate}
			\item $\af_3\ne 0$. Then choosing $x=\frac 1{\sqrt[4]{\af_3}}$ and $y=-\frac{\af_1x}{{\red 2}\af_3}$, we obtain the family of representatives of distinct orbits $\la\af\nb 2+\nb 3\ra$.
			\item $\af_3=0$. Then $\af_2\ne 0$. So, we obtain two representatives $\la\nb 2\ra$ and $\la\nb 1+\nb 2\ra$ depending on whether $\af_1=0$ or not.
		\end{enumerate}

	\end{enumerate}

Summarizing, in the case (1) we have the following distinct orbits: 
$$
     \la\nb 1+\nb 2\ra_{\lb\in\{-1,-\frac 12\}},
    \left\la\nb 1+\frac{\red 2}{(\lb + 1)(2\lb + 1)}\nb 2+\nb 3\right\ra_{\lb\not\in\{-1,-\frac 12,0\}},
     \la\nb 2\ra_{\lb\ne 0},
    \la\af\nb 2+\nb 3\ra_{\lb\ne 0}.
$$

%\[ \la\af\nb 2+\nb 3\ra, \la\nb 2\ra, \
%\la\af\nb 1+\frac{2\lb^2  + 5\lb  - 1}{(\lb + 1)(2\lb + 1)}\nb 2+\nb 3\ra_{\af\neq 0, \lb \neq -1,-1/2}, \
%\la\af\nb 1+\nb 2\ra_{\af\neq 0,\lb=-1 \mbox{ or } \lb=-1/2}. \]

%Entonces, después de un cambio de la base, puede recibir

%\[ \la\af\nb 2+\nb 3\ra, \la\af\nb 1+\nb 2\ra. \]

	\item $\lb=0.$ 
	Let us use the following notations 
	\begin{align*}
	\nb 1 = \Dl 12, \nb 2 = -\Dl 13 + 3 \Dl 31, \nb 3 = {\red\Dl 13 + \Dl 22}, \nb 4 = \Dl 23.    
	\end{align*}
	Take $\0=\sum_{i=1}^4\af_i\nb i\in {\rm H_T^2}(\T {3}{01}(0)).$
	If 
	$$
	\phi=
	\begin{pmatrix}
    x &               0  &  0\\
    y &             x^2  &  0\\
    z &   xy  &  x^3
    \end{pmatrix}
\in\aut{\T {3}{01}(0)},
	$$
	then
	$$
	\phi^T\begin{pmatrix}
	0      &  \af_1    & -\af_2{\red+\af_3}\\
	0  & \red{\af_3} & \af_4\\
	{\red 3\af_2}&  0    & 0
	\end{pmatrix} \phi=
	\begin{pmatrix}
	\alpha^*      &  \af^*_1    & -\af^*_2{\red+\af_3^*}\\
	\alpha^{**}  & {\red\af^*_3} & \af^*_4\\
	{\red 3\af^*_2}&  0    & 0
	\end{pmatrix},
	$$
	where
	\begin{align*}
	\af^*_1 &= x (x^2 \af_1 - x y (\af_2 {\red-2}\af_3) + y^2 \af_4),\\
    \af^*_2 &= {\red \af_2 x^4},\\
    \af^*_3 &= x^3 (x \af_3 {\red+} y \af_4),\\
	\af^*_4 &= \af_4x^5.
	\end{align*}
	Hence, $\phi\langle\0\rangle=\langle\0^*\rangle,$ where $\0^*=\sum\limits_{i=1}^4 \af_i^*  \nb i.$
	
	The condition $\theta \in \mathbf{T}_1 (\T{3}{01}(0))$ gives us $(\alpha_2, {\red\alpha_3},\af_4) \neq (0,0,0).$

	\begin{enumerate}
	    \item $\alpha_4\neq0$. Then choosing $y={\red-}\frac{\af_3x}{\alpha_4}$ we have the family of representatives 
	     $\langle \alpha_1^{\star} \nabla_1 + {\red\alpha_2^*} \nabla_2+ \af^*_4\nabla_4 \rangle$,
	     where
	     \begin{align*}
	        \af^\star_1=& \frac{x^3}{\af_4}(\af_1\af_4{\red+}\af_2\af_3{\red-\af_3^2}).
	     \end{align*}
	     
	     \begin{enumerate}
	         \item ${\red\af_2}\ne 0$. Then we have the family of representatives of distinct orbits $\langle  \af \nabla_1+ \nabla_2+ \nabla_4 \rangle$.
	         
	         \item ${\red\af_2}=0$. Then we have two representatives  $\langle   \nabla_4 \rangle$ and $\langle \nabla_1+ \nabla_4 \rangle$ depending on whether $\af_1\af_4{\red+}\af_2\af_3{\red-\af_3^2}=0$ or not.
    
	     \end{enumerate}

	    \item $\alpha_4=0.$ 
	    \begin{enumerate}
	        \item $\af_2 {\red-2}\af_3\ne 0$. Then choosing $y = \frac{\af_1x}{\af_2 {\red-2}\af_3}$ we have the representative $\la\nb 2\ra$ and the family of representatives of distinct orbits $\la\af\nb 2+\nb 3\ra_{\af\ne{\red 2}}$ depending on whether $\af_3=0$ or not.
	        \item $\af_2 {\red-2}\af_3=0$. Then we have two representatives $\la{\red 2}\nb 2+\nb 3\ra$ and $\la\nb 1{\red +2}\nb 2+\nb 3\ra$ depending on whether $\af_1=0$ or not. The representative $\la{\red 2}\nb 2+\nb 3\ra$ will be joined with the family $\la\af\nb 2+\nb 3\ra_{\af\ne{\red 2}}$.
	    \end{enumerate}
	    \end{enumerate}
	    
Summarizing, in the case (2)  we have the following distinct orbits: 
$$ 
\la\nb 1{\red +2}\nb 2+\nb 3\ra_{\lb=0}, \langle  \af \nabla_1+ \nabla_2+ \nabla_4 \rangle_{\lb=0}, \langle \nabla_1+ \nabla_4 \rangle_{\lb=0}, \la\nb 2\ra_{\lb=0}, \la\af\nb 2+\nb 3\ra_{\lb=0}, \langle   \nabla_4 \rangle_{\lb=0}.
$$
\end{enumerate}
Now, taking into account the both cases (1) and (2), we have the following distinct orbits:

$
\begin{array}{l}
     \la\nb 1+\nb 2\ra_{\lb\in\{-1,-\frac 12\}},\\ 
     \left\la\nb 1+\frac{\red 2}{(\lb + 1)(2\lb + 1)}\nb 2+\nb 3\right\ra_{\lb\not\in\{-1,-\frac 12\}},
\end{array}
\begin{array}{l}
    \langle  \af \nabla_1+ \nabla_2+ \nabla_4 \rangle_{\lb=0},\\
    \langle \nabla_1+ \nabla_4 \rangle_{\lb=0},\\
\end{array}
\begin{array}{l}
    \la\nb 2\ra,\\
    \la\af\nb 2+\nb 3\ra,
\end{array}
\begin{array}{l}
    \langle \nabla_4 \rangle_{\lb=0}.
\end{array}
$

Observe that $\left\la\nb 1+\frac{\red 2}{(\lb + 1)(2\lb + 1)}\nb 2+\nb 3\right\ra=\la(\lb + 1)(2\lb + 1)\nb 1+{\red 2}\nb 2+(\lb + 1)(2\lb + 1)\nb 3\ra$, the latter being $\la\nb 2\ra$, when $\lb\in\{-1,-\frac 12\}$. Now, if $\lb\not\in\{-1,-\frac 12\}$, then $\orb\la\nb 2\ra=\orb\la\nb 1+\nb 2\ra$. Thus, we may reorganize our orbits as follows:

$
\begin{array}{l}
     \la\nb 1+\nb 2\ra,\\ 
     \la(\lb + 1)(2\lb + 1)\nb 1+{\red 2}\nb 2+(\lb + 1)(2\lb + 1)\nb 3\ra,
\end{array}
\begin{array}{l}
    \langle  \af \nabla_1+ \nabla_2+ \nabla_4 \rangle_{\lb=0},\\
    \langle \nabla_1+ \nabla_4 \rangle_{\lb=0},
\end{array}
\begin{array}{l}
     \la\af\nb 2+\nb 3\ra,\\
     \la\nb 4 \ra_{\lb=0}.
\end{array}
$

The corresponding algebras are:	    

%$$    
%\begin{array}{lllllllll}
%        \T 4{**}(\lb, \alpha)   &:& e_1 e_1 = e_2, & e_1 e_2=\lb e_3+\af e_4, & e_2 e_1=e_3,& e_1e_3=-e_4,&e_3e_1=3e_4 \\   
%        \T 4{**}(\lb, \alpha)   &:& e_1 e_1 = e_2, & e_1 e_2=\lb e_3, & e_2 e_1=e_3,& e_1e_3=-\af e_4,& e_2e_2=-e_4, &e_3e_1=3(\af-1) e_4 \\   
%        \T 4{**}(\alpha)  &:& e_1 e_1 = e_2,& e_1e_3=-\af e_4, & e_2 e_1=e_3,& e_2e_2=-e_4,& e_2e_3=e_4,& e_3e_1= 3(\af-1) e_4 \\
%        \T 4{**}   &:& e_1 e_1 = e_2,& e_1e_3=- e_4, & e_2 e_1=e_3,& e_2e_3=e_4,& e_3e_1=3 e_4\\
%        \T 4{**}   &:& e_1 e_1 = e_2,& e_2 e_1=e_3,& e_2e_3=e_4 
%\end{array}
%$$

$$
\begin{array}{lllllllll}
%bunda-39б pizdes-40б jopa-41 popka-42 MJ-43 popochka-44
     \T 4{39}(\lb)   &:& e_1 e_1 = e_2, & e_1 e_2=\lb e_3+e_4, & e_1e_3 = (\lb-1)e_4,\\
     && e_2 e_1=e_3, & e_3e_1 = 3e_4;\\ 
     
     \T 4{40}(\lb)   &:& e_1 e_1 = e_2, & e_1 e_2=\lb e_3+(\lb + 1)(2\lb + 1)e_4, & e_1e_3 = {\red(2\lb^2  + 5\lb  - 1)}e_4,\\
     && e_2 e_1=e_3, & e_2e_2={\red(\lb + 1)(2\lb + 1)}e_4, & e_3e_1 = {\red 6}e_4;\\ 
     
     \T 4{41}(\af)   &:& e_1 e_1 = e_2, & e_1 e_2=\af e_4, & e_1e_3=-e_4,\\
     && e_2 e_1=e_3, & e_2e_3 = e_4, & e_3e_1=3e_4;\\
     
     \T 4{42}   &:& e_1 e_1 = e_2, & e_1 e_2=e_4, & e_2 e_1=e_3, & e_2e_3 = e_4;\\
     
     \T 4{43}(\lb,\af)   &:& e_1 e_1 = e_2, & e_1 e_2=\lb e_3, & e_1e_3 = (\af(\lb-1){\red+1})e_4,\\
     && e_2 e_1=e_3, & e_2e_2={\red e_4}, & e_3e_1 = 3{\red\af} e_4;\\ 
     
     \T 4{44}   &:& e_1 e_1 = e_2, & e_2 e_1=e_3,    & e_2e_3 = e_4.
\end{array}
$$

 \section{The geometric classification of nilpotent terminal algebras}

\subsection{Degenerations of algebras}
Given an $n$-dimensional vector space ${\bf V}$, the set ${\rm Hom}({\bf V} \otimes {\bf V},{\bf V}) \cong {\bf V}^* \otimes {\bf V}^* \otimes {\bf V}$ 
is a vector space of dimension $n^3$. This space inherits the structure of the affine variety $\mathbb{C}^{n^3}.$ 
Indeed, let us fix a basis $e_1,\dots,e_n$ of ${\bf V}$. Then any $\mu\in {\rm Hom}({\bf V} \otimes {\bf V},{\bf V})$ is determined by $n^3$ structure constants $c_{i,j}^k\in\mathbb{C}$ such that
$\mu(e_i\otimes e_j)=\sum_{k=1}^nc_{i,j}^ke_k$. A subset of ${\rm Hom}({\bf V} \otimes {\bf V},{\bf V})$ is {\it Zariski-closed} if it can be defined by a set of polynomial equations in the variables $c_{i,j}^k$ ($1\le i,j,k\le n$).

The general linear group ${\rm GL}({\bf V})$ acts by conjugation on the variety ${\rm Hom}({\bf V} \otimes {\bf V},{\bf V})$ of all algebra structures on ${\bf V}$:
$$ (g * \mu )(x\otimes y) = g\mu(g^{-1}x\otimes g^{-1}y),$$ 
for $x,y\in {\bf V}$, $\mu\in {\rm Hom}({\bf V} \otimes {\bf V},{\bf V})$ and $g\in {\rm GL}({\bf V})$. Clearly, the ${\rm GL}({\bf V})$-orbits correspond to the isomorphism classes of algebras structures on ${\bf V}$. Let $T$ be a set of polynomial identities which is invariant under isomorphism. Then the subset $\mathbb{L}(T)\subset {\rm Hom}({\bf V} \otimes {\bf V},{\bf V})$ of the algebra structures on ${\bf V}$ which satisfy the identities in $T$ is ${\rm GL}({\bf V})$-invariant and Zariski-closed. It follows that $\mathbb{L}(T)$ decomposes into ${\rm GL}({\bf V})$-orbits. The ${\rm GL}({\bf V})$-orbit of $\mu\in\mathbb{L}(T)$ is denoted by $O(\mu)$ and its Zariski closure by $\overline{O(\mu)}$.

Let ${\bf A}$ and ${\bf B}$ be two $n$-dimensional algebras satisfying the identities from $T$ and $\mu,\lambda \in \mathbb{L}(T)$ represent ${\bf A}$ and ${\bf B}$ respectively.
We say that ${\bf A}$ {\it degenerates} to ${\bf B}$ and write ${\bf A}\to {\bf B}$ if $\lambda\in\overline{O(\mu)}$.
Note that in this case we have $\overline{O(\lambda)}\subset\overline{O(\mu)}$. Hence, the definition of a degeneration does not depend on the choice of $\mu$ and $\lambda$. It is easy to see that any algebra degenerates to the algebra with zero multiplication. If ${\bf A}\to {\bf B}$ and ${\bf A}\not\cong {\bf B}$, then ${\bf A}\to {\bf B}$ is called a {\it proper degeneration}. We write ${\bf A}\not\to {\bf B}$ if $\lambda\not\in\overline{O(\mu)}$ and call this a {\it non-degeneration}. Observe that the dimension of the subvariety $\overline{O(\mu)}$ equals $n^2-\dim\der({\bf A})$. Thus if ${\bf A}\to {\bf B}$ is a proper degeneration, then we must have $\dim\der({\bf A})>\dim\der({\bf B})$.

Let ${\bf A}$ be represented by $\mu\in\mathbb{L}(T)$. Then  ${\bf A}$ is  {\it rigid} in $\mathbb{L}(T)$ if $O(\mu)$ is an open subset of $\mathbb{L}(T)$.
Recall that a subset of a variety is called {\it irreducible} if it cannot be represented as a union of two non-trivial closed subsets. A maximal irreducible closed subset of a variety is called an {\it irreducible component}.
It is well known that any affine variety can be represented as a finite union of its irreducible components in a unique way.
The algebra ${\bf A}$ is rigid in $\mathbb{L}(T)$ if and only if $\overline{O(\mu)}$ is an irreducible component of $\mathbb{L}(T)$.

In the present work we use the methods applied to Lie algebras in \cite{GRH,GRH2}.
To prove %primary 
degenerations, we will construct families of matrices parametrized by $t$. Namely, let ${\bf A}$ and ${\bf B}$ be two algebras represented by the structures $\mu$ and $\lambda$ from $\mathbb{L}(T)$, respectively. Let $e_1,\dots, e_n$ be a basis of ${\bf V}$ and $c_{i,j}^k$ ($1\le i,j,k\le n$) be the structure constants of $\lambda$ in this basis. If there exist $a_i^j(t)\in\mathbb{C}$ ($1\le i,j\le n$, $t\in\mathbb{C}^*$) such that the elements $E_i^t=\sum_{j=1}^na_i^j(t)e_j$ ($1\le i\le n$) form a basis of ${\bf V}$ for any $t\in\mathbb{C}^*$, and the structure constants $c_{i,j}^k(t)$ of $\mu$ in the basis $E_1^t,\dots, E_n^t$ satisfy $\lim\limits_{t\to 0}c_{i,j}^k(t)=c_{i,j}^k$, then ${\bf A}\to {\bf B}$. In this case  $E_1^t,\dots, E_n^t$ is called a {\it parametric basis} for ${\bf A}\to {\bf B}$.

To prove a non-degeneration ${\bf A}\not\to {\bf B}$ we will use the following lemma (see \cite{GRH}).

\begin{lemma}\label{main}
Let $\mathcal{B}$ be a Borel subgroup of ${\rm GL}({\bf V})$ and $\mathcal{R}\subset \mathbb{L}(T)$ be a $\mathcal{B}$-stable closed subset.
If ${\bf A} \to {\bf B}$ and ${\bf A}$ can be represented by $\mu\in\mathcal{R}$ then there is $\lambda\in \mathcal{R}$ that represents ${\bf B}$.
\end{lemma}

In particular, it follows from Lemma \ref{main} that ${\bf A}\not\to {\bf B}$, whenever $\dim({\bf A}^2)<\dim({\bf B}^2)$.
%Each time when we will need to prove some primary non-degeneration $\mu\not\to\lambda$, we will define $\mathcal{R}$ by a set of polynomial equations in structure constants $c_{ij}^k$ in such a way that the structure constants of $\mu$ in the basis $e_1,\dots,e_n$ satisfy these equations. We will omit everywhere the verification of the fact that $\mathcal{R}$ is stable under the action of the subgroup of lower triangular matrices and of the fact that $\lambda\not\in\mathcal{R}$ for any choice of a basis of ${\bf V}.$ 
%To simplify our equations, we will use the notation $A_i=\langle e_i,\dots,e_n\rangle,\ i=1,\ldots,n$ and write simply $A_pA_q\subset A_r$ instead of $c_{ij}^k=0$ ($i\geq p$, $j\geq q$, $k\geq r$).

When the number of orbits under the action of ${\rm GL}({\bf V})$ on  $\mathbb{L}(T)$ is finite, the graph of primary degenerations gives the whole picture. In particular, the description of rigid algebras and irreducible components can be easily obtained. Since the variety of $4$-dimensional nilpotent terminal algebras contains infinitely many non-isomorphic algebras, we have to fulfill some additional work. Let ${\bf A}(*):=\{{\bf A}(\alpha)\}_{\alpha\in I}$ be a family of algebras and ${\bf B}$ be another algebra. Suppose that, for $\alpha\in I$, ${\bf A}(\alpha)$ is represented by a structure $\mu(\alpha)\in\mathbb{L}(T)$ and ${\bf B}$ is represented by a structure $\lambda\in\mathbb{L}(T)$. Then by ${\bf A}(*)\to {\bf B}$ we mean $\lambda\in\overline{\cup\{O(\mu(\alpha))\}_{\alpha\in I}}$, and by ${\bf A}(*)\not\to {\bf B}$ we mean $\lambda\not\in\overline{\cup\{O(\mu(\alpha))\}_{\alpha\in I}}$.

Let ${\bf A}(*)$, ${\bf B}$, $\mu(\alpha)$ ($\alpha\in I$) and $\lambda$ be as above. To prove ${\bf A}(*)\to {\bf B}$ it is enough to construct a family of pairs $(f(t), g(t))$ parametrized by $t\in\mathbb{C}^*$, where $f(t)\in I$ and $g(t)=\left(a_i^j(t)\right)_{i,j}\in {\rm GL}({\bf V})$. Namely, let $e_1,\dots, e_n$ be a basis of ${\bf V}$ and $c_{i,j}^k$ ($1\le i,j,k\le n$) be the structure constants of $\lambda$ in this basis. If we construct $a_i^j:\mathbb{C}^*\to \mathbb{C}$ ($1\le i,j\le n$) and $f: \mathbb{C}^* \to I$ such that $E_i^t=\sum_{j=1}^na_i^j(t)e_j$ ($1\le i\le n$) form a basis of ${\bf V}$ for any  $t\in\mathbb{C}^*$, and the structure constants $c_{i,j}^k(t)$ of $\mu\big(f(t)\big)$ in the basis $E_1^t,\dots, E_n^t$ satisfy $\lim\limits_{t\to 0}c_{i,j}^k(t)=c_{i,j}^k$, then ${\bf A}(*)\to {\bf B}$. In this case, $E_1^t,\dots, E_n^t$ and $f(t)$ are called a {\it parametric basis} and a {\it parametric index} for ${\bf A}(*)\to {\bf B}$, respectively. In the construction of degenerations of this sort, we will write $\mu\big(f(t)\big)\to \lambda$, emphasizing that we are proving the assertion $\mu(*)\to\lambda$ using the parametric index $f(t)$.

%To prove an assertion of the form ${\bf A}(*)\not\to {\bf B}$, one can use the fact that if $\bf{C}\to \bf{A}(\alpha)$ for any $\alpha\in I$ and $C\not\to \bf{B}$, then ${\bf A}(*)\not\to {\bf B}$.

Through a series of degenerations summarized in the table below by the corresponding parametric bases and indices, we obtain the main result of the second part of the paper.

\begin{theorem}\label{main-geo}
The variety of $4$-dimensional complex nilpotent terminal algebras has 3 irreducible components: one of dimension $17$ determined by the family of algebras $\D 4{01}(\lb,\af,\bt)$ and two of dimension 15 determined by the families of algebras $\T 4{41}(\af)$ and $\T 4{43}(\lb,\af)$. 
\end{theorem}

\begin{proof}[{\bf Proof}]
Thanks to \cite{kppv} the algebras 
$\mathfrak{L}_2,$ $\mathfrak{L}_5,$ $\mathfrak{L}_{11}$ and $\mathfrak{N}_3(\alpha)$  
 define irreducible components in the variety of $4$-dimensional nilpotent Leibniz  algebras. 
 Note that in \cite{kppv} right Leibniz algebras were considered, and here we use their opposite versions which are left Leibniz algebras (and hence are terminal):
$$ 
    \begin{array}{lllllllll}
    
\mathfrak{N}_3(\alpha) & :&  e_1e_1 = e_4 &  e_1e_2 = -\alpha e_4 & e_2e_1 = \alpha e_4 & e_2e_2 = e_4&  e_3e_3 = e_4  \\

\mathfrak{L}_2  &: &  e_1e_1 = e_2 & e_1e_2 = e_3 &  e_1e_3 = e_4 \\  

\mathfrak{L}_5  &:&    e_1e_1 = e_3&   e_1e_2 = e_3& e_2e_2 = e_4&  e_1e_3=e_4 \\  

\mathfrak{L}_{11}  &:&  e_1e_1 = e_4& e_1e_2 = e_3 &  e_1e_3=e_4 & e_2e_1=-e_3 &   e_2e_2=e_4 & e_3e_1=-e_4.    \\ 
\end{array}
$$
 The list of all $4$-dimensional nilpotent non-Leibniz terminal  algebras was found in Theorem~\ref{main-alg}. 
 All these algebras  degenerate from one of the families: $\D 4{01}(\lb,\af,\bt)$, $\T 4{41}(\af)$ or $\T 4{43}(\lb,\af)$, as it is shown in the table below. By a straightforward computation we have $\dim\der(\D 4{01}(\lb,\af,\bt))=2$ for $\af,\bt\ne 0$, and $\dim\der(\D 4{01}(\lb,\af,\bt))>2$ otherwise. Hence, the closure of the orbit of $\D 4{01}(\lb,\af,\bt)$ has dimension $4^2-2+3=17$. A similar argument yields that the dimensions of the orbit closures of $\T 4{41}(\af)$ and $\T 4{43}(\lb,\af)$ are both equal to $15<17$. In particular, this shows that $\D 4{01}(\lb,\af,\bt)$ cannot degenerate from $\T 4{41}(\af)$ or $\T 4{43}(\lb,\af)$. Moreover, $\T 4{41}(\af)$ and $\T 4{43}(\lb,\af)$\ do not degenerate from $\D 4{01}(\lb,\af,\bt)$, since the squares of $\T 4{41}(\af)$ and $\T 4{43}(\lb,\af)$ have dimension 3, while the square of $\D 4{01}(\lb,\af,\bt)$ has dimension $2<3$. This completes the proof of the theorem.

\end{proof}

%\section{The geometric classification of nilpotent terminal algebras}

{\tiny

\begin{longtable}{|lcl|ll|}
\multicolumn{5}{c}{  {\bf Table.} {\it  Degenerations of $4$-dimensional nilpotent terminal algebras.}}\\ 

\hline
$\T {4}{03}\left(\frac{1}{1 - t}\right)$ &$\to$&  $\mathfrak{N}_3(\alpha)$&
$E^t_1= \frac{2 \alpha t}{\sqrt{t-1}} e_1 -  \alpha t \sqrt{t-1} e_2 + \alpha t \sqrt{t-1} e_3$ &
$E^t_3= t e_3 + \frac{t}{4}  \left(1 + \frac{1 - t}{\alpha^2} - t\right) e_4$\\

&&&
$E^t_2= -t \sqrt{t-1} e_2 + t \sqrt{t-1}  e_3$ &
$E^t_4= t^2 e_4$ \\

\hline
$\T {4}{38}\left(t^{-1}\right)$ &$\to$& $\mathfrak{L}_2$&

$E^t_1= e_1$ &
$E^t_3= e_3$  \\
&&& 
$E^t_2= e_2$&
$E^t_4= t^{-1}  e_4$\\

\hline
$\D {4}{01}\left( 2t^{-1},  t^{-1}, t \right)$ &$\to$& $\mathfrak{L}_5$&

$E^t_1= 2 e_1 - 2 e_2 - 2 e_3$ &
$E^t_3= 4 t e_3 - 4t^{-1} e_4$\\
&&& 
$E^t_2= -2 t e_2 - 2 e_3$ &
$E^t_4= 8 e_4$\\

\hline
$\D {4}{01}\left( -\frac 12, 1, \frac{1 + t^2}{4}\right)$ &$\to$& $\mathfrak{L}_{11}$&
$E^t_1= 4t^{-1} e_1 - 2t^{-1} e_2$ &
$E^t_3= -8t^{-1} e_3 + 32t^{-3}  e_4$\\
&&&
$E^t_2= 4  e_2 - 8t^{-2} e_3$ &
$E^t_4= 32t^{-2} e_4$\\

\hline

\red $\T {4}{43} \left(0,0\right)$   &\red $\to$& \red $\T {4}{01}$ & 

\red $E^t_1= te_1$ &
\red $E^t_3= e_3$  \\
&&& 
\red $E^t_2= t^2e_2$&
\red $E^t_4= e_4$\\
\hline

\red$\T {4}{43} \left(0, 0\right)$   &\red$\to$& \red$\T {4}{02}$ & 

\red $E^t_1= te_1 + \frac 12 t^{-2}e_2$ &
\red $E^t_3= t^3e_3 + \frac 12 e_4$ \\ 
&&& 
\red $E^t_2= t^2e_2 + \frac 12 t^{-1}e_3 + \frac 14 t^{-4}e_4$ &
\red $E^t_4= e_4$\\
\hline

$\D {4}{01} \left(0, t + \af^{-1},t\right)$   &$\to$& $\T {4}{03} (\af)$ & 
$E^t_1= \af^2 t e_2$ &
$E^t_3= \af^3 t^2 (e_1 - e_2)$\\
&&&
$E^t_2= \af^4 t^2 e_3$ &
$E^t_4= \af^6 t^4 e_4$
 \\
\hline

$\T {4}{03} \left(\af \right)$   &$\to$& $\T {4}{04} (\af)$ & 
$E^t_1= t^{-1} e_1$ &
$E^t_3= t^{-1} e_3$\\
&&&
$E^t_2= t^{-2} e_2$ &
$E^t_4= t^{-3} e_4$
 \\
\hline

$\T {4}{07} \left(\frac{1}{t-1} \right)$   &$\to$& $\T {4}{05}$  & 

$E^t_1= \frac{(t-1)^2}{1 - 3 t + t^2} e_1  + \frac{(t-1)^3}{(1 - 3 t + t^2)^2} e_2+  t e_3$ &

$E^t_2= \frac{(t-1)^4}{ (1 - 3 t + t^2)^2} e_2 + \frac{(t-1)^4 t}{ (1 - 3 t + t^2)^2} e_4$ \\
&&&
$E^t_3=   - \frac{(t-1)^4}{ t (1 - 3 t + t^2)^2} e_2+e_3$ &
$E^t_4= -\frac{(t-1)^3}{ (1 - 3 t + t^2)^2} e_4$

 \\
\hline

$\T {4}{03} \left(t \right)$   &$\to$& $\T {4}{06}$  & 

$E^t_1= \frac{t - 1}{1 + t^2} e_1 + \frac{t^2 - 1}{(1 + t^2)^2} e_2 + \frac{1 - t}{1 + t^2} e_3$ &
$E^t_3= \frac{(1 - t) t}{1 + t^2} e_3$ \\
&&& 
$E^t_2= \frac{(t - 1)^2}{(1 + t^2)^2} e_2 +  \frac{(1 - t)^3}{(1 + t^2)^3} e_4$ &
$E^t_4= \frac{(1 - t)^3 t}{(1 + t^2)^3} e_4$\\
\hline

$\T {4}{03} \left( \alpha \right)$   &$\to$& $\T {4}{07} (\alpha)$  & 
$E^t_1= -\frac{1}{ 1 +\alpha} e_1 - \frac{1}{ (1 +\alpha)^2} e_2 + \frac{1}{ 1 +\alpha} e_3$ &
$E^t_2= \frac{1}{(1 +\alpha)^2} e_2 + \frac{1}{ (1 +\alpha)^3} e_4$\\
&&&
$E^t_3=   t^{-1} (e_2-(1 +\alpha) e_3)$ &
$E^t_4= -\frac{1}{ t(1 +\alpha)} e_4$ \\

\hline

$\T {4}{03} \left(-1 \right)$   &$\to$& $\T {4}{08}$  & 

$E^t_1= t^{-1} e_1$ & $E^t_3= t^{-2} e_3$ \\
&&&
$E^t_2= t^{-2} e_2$ &
$E^t_4= t^{-4} e_4$\\

\hline

$\T {4}{03} \left(t+t^4 \right)$   &$\to$& $\T {4}{09}$  & 
$E^t_1= \frac{t-1}{t^2 + t^5} e_1 + \frac{t-1}{t^3 (1 + t^3)^2}e_2 + \frac{1 - t}{t^2 + t^5} e_3$ &
$E^t_3= -\frac{(t-1 ) (-1 + t - t^2 - t^3 + 2 t^4 + t^7)}{ t (1 + t^3)^2 (-1 + t + t^4)} e_3$ \\
&&& 
$E^t_2= \frac{(t-1)^2}{t^4 (1 + t^3)^2} e_2 + \frac{t-1}{(1 + t^3)^2} e_3$& 
$E^t_4= -\frac{(t-1)^2 (-1 + t - t^2 - t^3 + 2 t^4 + t^7)}{ t^5 (1 + t^3)^4} e_4$\\ 
\hline
  
$\T {4}{03} (\af )$   &$\to$& $\T {4}{10}(\af )$  & 

$E^t_1= t^{-2} e_1$ & $E^t_3= t^{-3} e_3$\\
&&&
$E^t_2= t^{-4} e_2$  & $E^t_4= t^{-7} e_4$\\
\hline
  
$\T {4}{03} (1+t^{-2})$   &$\to$& $\T {4}{11}$   & 
$E^t_1= t^3 e_1$ & $E^t_3= t^5 e_3$ \\

&&&
$E^t_2= t^6 e_2 - t^6 e_3$ &
$E^t_4= t^9 e_4$\\
\hline

$\T {4}{41} (t^{-1})$    &$\to$& $\T {4}{12}$   & 

$E^t_1= 3 e_1 + 3 e_2$ &
$E^t_3= 3t^{-1} e_3 + 9t^{-1} e_4$ \\
&&& 
$E^t_2= 9 e_2 + 9 e_3 + 9t^{-1} e_4$ &
$E^t_4= 27t^{-1} e_4$\\
\hline

$\T {4}{12}$    &$\to$& $\T {4}{13}$   & 

$E^t_1=  e_1$ & $E^t_3= t^{-1}  e_3$ \\
&&& $E^t_2=  e_2$  & $E^t_4= t^{-1}  e_4$\\
\hline

$\T {4}{03} (1/t )$   &$\to$& $\T {4}{14}$   & 

$E^t_1= t^{-1}   e_1$ & $E^t_3= t^{-2}  e_3$ \\
&&& $E^t_2= t^{-2}  e_2$  & $E^t_4= t^{-5}  e_4$\\
\hline

$\T {4}{03} (\frac{1 - t}{\af} )$   &$\to$& $\T {4}{15}(\af)$   & 

$E^t_1= e_1 + \frac{\af}{ t ( \af -1+ t)} e_2 -  \frac{\af}{t (  \af-1 + t)} e_3$ &
$E^t_3= \frac{1 - t}{\af-1 + t} e_3 + \frac{\af (\af t -1+ t^2)}{  t^3 ( \af-1 + t)^2} e_4$ \\
&&& 
$E^t_2= e_2 + \frac{(t-1) t}{\af-1 + t} e_3$ &
$E^t_4= \frac{-1 + t}{t ( \af-1 + t)} e_4$\\
\hline

$\T {4}{03} (1 + t )$   &$\to$& $\T {4}{16}$   &

$E^t_1= \frac{t}{(1 + t) X^2} e_1 + \frac{1 - t}{(1 + t) X^3} e_2 +  \frac{1}{(1 + t) X^2}e_3$ &
$E^t_3= \frac{t}{(1 + t) X^3} e_3$ \\
\multicolumn{3}{|l|}{$X=-1 + t + t^2$}&
$E^t_2= \frac{t^2}{(1 + t)^2 X^4} e_2 + \frac{ 1 + t - t^2}{(1 + t)^2 X^5} e_4$ &
$E^t_4= \frac{t^2}{(1 + t)^2 X^6} e_4$\\
\hline

$\T {4}{03} (\frac{1}{1 - t^2} )$   &$\to$& $\T {4}{17}(\af)$   & 

$E^t_1= t^2 X e_1 + t^2(-1 + t^2) X^2 e_2 + (t^2 - t^4) X^2 e_3$ &
$E^t_3=  t^3 X^2 e_3$ \\

\multicolumn{3}{|l|}{$X=(-1 + \af + t^2)^{-1}$}&
$E^t_2= t^4  X^2 e_2 + t^4 ( t^2-1)(\af + t^2) X^4 e_4$ &
$E^t_4=  t^6 X^4 e_4$\\
\hline

$\T 4{19}$ &$\to$& $\T 4{18}$  & 

$E_1^t = \sqrt{1-t^2}e_1 + te_2 - te_3$ &  $E_3^t = e_3$\\
&&& $E_2^t = e_2$ & $E_4^t = \sqrt{1-t^2}e_4$\\
\hline

$\T 4{23}(t^{-1},-t-1)$ &$\to$& $\T 4{19}$  & 

$E_1^t = t^2e_1 - \frac{t^3}{t + 1}e_3$ & $E_3^t = t^4e_3$ \\
&&& $E_2^t = t^2e_2$ & $E_4^t = t^5e_4$\\
\hline

$\T 4{23}(it^{-1},t^{-1})$ &$\to$& $\T 4{20}$  & 

$E_1^t = te_1$ & $E_3^t = t^2e_3$ \\
&&& $E_2^t = te_2$ & $E_4^t = t^2e_4$\\
\hline

$\T 4{20}$ &$\to$& $\T 4{21}$  & 

$E_1^t = t^{-1}e_1$ &  $E_3^t = t^{-2}e_3$\\
&&& $E_2^t = t^{-1}e_2$ & $E_4^t = t^{-3}e_4$\\
\hline

$\T 4{23}(\af,\bt)$ &$\to$& $\T 4{22}(\af,\bt)$  & 

$E_1^t = t^{-1}e_1$ & $E_3^t = t^{-2}e_3$ \\
&&& $E_2^t = t^{-1}e_2$ & $E_4^t = t^{-3}e_4$\\
\hline

$\D 4{01}\left(\bt^2t^{-2},t+\af,t\right)$ &$\to$& $\T 4{23}(\af,\bt)$  & 

$E_1^t = t^2\bt^{-2}(e_1 - e_2)$ & $E_3^t = t^2\bt^{-2}e_3$ \\
&&& $E_2^t = t\bt^{-1}e_2$ & $E_4^t = t^4\bt^{-4}e_4$\\
\hline

$\T 4{23}(\af,-(\af + 1)t)$ &$\to$& $\T 4{24}(\af)$  & 

$E_1^t = te_1 - t(\af + 1)^{-1}e_3$ & $E_3^t = t^2e_3$ \\
&&& $E_2^t = te_2$ & $E_4^t = t^3e_4$\\
\hline

$\T 4{23}(-1,t)$ &$\to$& $\T 4{25}(\af)$  & 

$E_1^t = e_1 + \af t^{-1}e_3$ & $E_3^t = e_3$ \\
&&& $E_2^t = e_2$ & $E_4^t = e_4$\\
\hline

$\T 4{23}(t,i)$ &$\to$& $\T 4{26}(\af)$  & 

$E_1^t = \frac i{\af + i}e_1 -\frac{\af i}{(\af + i)^2t}e_3$ & $E_3^t = -\frac 1{(\af + i)^2}e_3 + \frac{\af}{(\af + i)^3t}e_4$ \\
&&& $E_2^t = \frac i{\af + i}e_2 - \frac{\af}{(\af + i)^2t}e_3$ & $E_4^t = -\frac i{(\af + i)^3}e_4$\\
\hline

$\T 4{23}\left(\frac{(\af^2 + 1)t}{\af - t},\af\right)$ &$\to$& $\T 4{27}(\af)$  & 

$E_1^t =  te_1 +\frac{(t-\af)t}{\af(\af^2 + 1)}e_3$ & $E_3^t = t^2e_3 - \frac{\af(t-\af)t^2}{\af^2 + 1}e_4$ \\
&&& $E_2^t = te_2 - \frac{(t-\af)t}{\af^2 + 1}e_3$ & $E_4^t = t^3e_4$\\
\hline

$\T 4{23}\left(\frac{i\af(1-t)}t,\frac{\af\sqrt{1-t}}t\right)$ &$\to$& $\T 4{28}(\af)$  & 

$E_1^t =  \frac{\af^2(t - 1)}t\left(-i e_1 + \frac 1{\sqrt{1-t}} e_2 - (\af - i) e_3\right)$ & $E_3^t = -\frac{\af^4(t - 1)}t e_3$ \\
&&& $E_2^t = \frac{\af^2(t - 1)}t(-e_1 + i\sqrt{1-t} e_2 +  e_3)$ & $E_4^t = \frac{\af^6(t - 1)^2}{t^2}e_4$\\
\hline

$\T 4{28}(\af)$ &$\to$& $\T 4{29}(\af)$  & 

$E_1^t = t^{-1}e_1$ & $E_3^t = t^{-2}e_3$ \\
&&& $E_2^t = t^{-1}e_2$ & $E_4^t = t^{-3}e_4$\\
\hline

$\T 4{23}\left(1,(1 - \af)\sqrt{t - 1}\right)$ &$\to$& $\T 4{30}(\af)$  & 
 \multicolumn{2}{l|}{$E_1^t = \frac{(\af - 1)t(t - 1)}D\left(-e_1+\frac 1{\sqrt{t - 1}} e_2-\frac{(\af - 2)t - \af  + 1}D e_3\right)$}\\
 \multicolumn{3}{|l|}{$D=\af^2t^2 - 2\af^2t - 2\af t^2 + \af^2 + 2\af t + t^2 + 2t - 1$}
& \multicolumn{2}{l|}{$E_2^t = \frac{i(\af - 1)t(t - 1)}D\left(e_1+\sqrt{t - 1}e_2-\frac{(\af-1)t^2 - (2\af-3)t + \af - 1}D e_3\right)$}\\
&&& \multicolumn{2}{l|}{$E_3^t = \frac{(\af - 1)^2(t - 1)t^3}{D^2}\left(e_3+\frac{((\af-1) t - \af)(\af - 1)(t - 1)}De_4\right)$} \\
&&& \multicolumn{2}{l|}{$E_4^t = \frac{i(\af - 1)^3(t - 1)^2t^4}{D^3}e_4$}\\
\hline

$\T 4{30}(\af)$ &$\to$& $\T 4{31}(\af)$ & 

$E_1^t = -\frac{i(\af^2 - 1)}{\af t}e_1 - \frac i{\af t}e_3$ & $E_3^t = -\frac{(\af^2 - 1)^2}{\af^2t^2}e_3 - \frac{i(\af + 1)^2(\af - 1)}{\af t^2}e_4$\\
&&& $E_2^t = -\frac{i(\af^2 - 1)}{\af t}e_2 + \frac 1t e_3$ & $E_4^t = i\frac{(\af^2 - 1)^3}{\af^3 t^3}e_4$\\
\hline

$\T 4{23}(\sqrt{t-1} + 1,1)$ &$\to$& $\T 4{32}$  & 

$E_1^t = \frac 1{(\sqrt{t-1} + 1)^2}\left(e_1 - \sqrt{t-1}e_2 -\frac{\sqrt{t-1}}{(\sqrt{t-1} + 1)^2}e_3\right)$ & $E_3^t = \frac t{(\sqrt{t-1} + 1)^4}e_3$ \\
&&& $E_2^t = \frac i{(\sqrt{t-1} + 1)^2}\left(-e_1 - \frac 1{\sqrt{t-1}}e_2 + \frac {\sqrt{t - 1}}{(\sqrt{t - 1} + 1)^2}e_3\right)$ & $E_4^t = -\frac {it}{(\sqrt{t - 1} + 1)^6}e_4$\\

\hline
$\T 4{23}(\sqrt{t-1} - 1,1)$ &$\to$& $\T 4{33}$  & 

$E_1^t = \frac 1t\left(e_1 - \sqrt{t - 1}e_2 -\frac{\sqrt{t - 1}}t e_3\right)$ & $E_3^t = \frac 1t e_3$ \\
&&& $E_2^t = \frac it\left(-e_1 -\frac 1{\sqrt{t - 1}}e_2 + \frac{\sqrt{t - 1}}t e_3\right)$ & $E_4^t = -\frac i{t^2}e_4$\\
\hline

$\T 4{23}(\sqrt{t-1} - 1,1)$ &$\to$& $\T 4{34}$  & 

\multicolumn{2}{l|}{$E_1^t = \frac{\sqrt{t-1} - 1}{t\sqrt{t-1} - t + 2}\left(e_1 -\sqrt{t-1}e_2 -\frac{\sqrt{t-1}(t\sqrt{t-1} - t + 4)}{t(t\sqrt{t-1} - t + 2)}e_3\right)$}\\
&&& 
\multicolumn{2}{l|}{$E_2^t = \frac{i(\sqrt{t-1}-1)}{t\sqrt{t-1} - t + 2}\left(-e_1 -\frac 1{\sqrt{t-1}}e_2 +\frac{t^2 - t\sqrt{t-1} - 3t + 4\sqrt{t-1}}{t(t\sqrt{t-1} - t + 2)}e_3\right)$}\\
&&& 
\multicolumn{2}{l|}{$E_3^t = \frac{(\sqrt{t-1} - 1)^2}{(t\sqrt{t-1} - t + 2)^2}\left(te_3 +\frac{2\sqrt{t-1}(t - \sqrt{t-1})}{t\sqrt{t-1} - t + 2}e_4\right)$}\\
&&& 
\multicolumn{2}{l|}{$E_4^t = -\frac{it(\sqrt{t-1} - 1)^3}{(t\sqrt{t-1} - t + 2)^3}e_4$}\\
\hline

$\D 4{01}\left(\frac 14,\frac 1{2t},\frac 1t\right)$ &$\to$& $\T 4{35}$  & 

$E_1^t = -t^{-1}(2e_1 - e_2 - 2e_3)$ & $E_3^t = -2t^{-1}e_3$ \\
&&& $E_2^t = t^{-1}(2e_1 - (2t + 1)e_2 - 4e_3)$ & $E_4^t = 4t^{-2}e_4$\\
\hline

$\T 4{35}$ &$\to$& $\T 4{36}$  & 

$E_1^t = t^{-1}e_1$ & $E_3^t = t^{-2}e_3$ \\
&&& $E_2^t = t^{-1}e_2$ & $E_4^t = t^{-3}e_4$\\
\hline

$\T 4{38}\left(t^{-1}\right)$ &$\to$& $\T 4{37}$  & 

$E_1^t = e_1 - \frac 1{2t}e_2$ & $E_3^t = e_3 - \frac 1{2t}(1 + \frac 1t)e_4$ \\
&&& $E_2^t = e_2 - \frac 1{2t}e_3 + \frac 1{4t^2}e_4$ & $E_4^t = \frac 1te_4$\\
\hline

\red $\T 4{43}\left((1-\af)(t + 1)t^{-1}, t(\af - 1)^{-1}\right)$ &\red $\to$& \red $\T 4{38}(\af)$  & 

\red $E^t_1= e_1$ & 
\red $E^t_3= (1-\af)t^{-1}e_3$ \\
&&& 
\red $E^t_2= e_2$ &
\red $E^t_4= e_4$\\
\hline

\red $\T 4{43}\left(\lb+t, -t^{-1}\right)$ &\red $\to$& \red $\T 4{39}(\lb)$  & 

\multicolumn{2}{l|}{\red $E_1^t = e_1 - De_2$}\\
 \multicolumn{3}{|l|}{\red $D=(2t^2 + (4\lb +5)t + 2\lb^2 + 3\lb + 1)^{-1}$}
& \multicolumn{2}{l|}{\red $E_2^t = e_2 - (t + \lb + 1)De_3 + D^2 e_4$}\\
&&& \multicolumn{2}{l|}{\red $E_3^t = e_3 + t^{-1}(2t + 3\lb + 3)De_4$} \\
&&& \multicolumn{2}{l|}{\red $E_4^t = -t^{-1}e_4$}\\
\hline
\red
$\T 4{43}\left(\lb  + t, \frac 2{(\lb + 1)(2\lb + 1)}\right)$ &\red$\to$& \red$\T 4{40}(\lb)$ & 

 \multicolumn{2}{l|}{\red $E_1^t = e_1 - \frac 12(\lb + 1)(2\lb + 1)D e_2$}\\
 \multicolumn{3}{|l|}{\red $D=(t(2t + 4\lb + 3))^{-1}$}
& \multicolumn{2}{l|}{\red $E_2^t = e_2 - \frac 12(\lb + 1)(2\lb + 1)(t + \lb + 1)De_3 + \frac 14(\lb + 1)^2(2\lb + 1)^2D^2e_4$}\\
&&& \multicolumn{2}{l|}{\red $E_3^t = e_3 - \frac 12(6t + 2\lb^2 + 9\lb + 7)De_4$} \\
&&& \multicolumn{2}{l|}{\red $E_4^t =  \frac 1{(\lb + 1)(2\lb + 1)}e_4$}\\
\hline

$\T 4{41}(t^{-2})$ &$\to$& $\T 4{42}$  & 

$E_1^t = t^{-1}e_1 + t^{-1}e_2$ &  $E_3^t = t^{-3}e_3 + 3t^{-3}e_4$\\
&&& $E_2^t = t^{-2}e_2 + t^{-2}e_3 + t^{-4}e_4$ & $E_4^t = t^{-5}e_4$\\
\hline

$\T 4{41}(0)$ &$\to$& $\T 4{44}$  & 

$E_1^t = t^{-1}e_1 + t^{-1}e_2$ & $E_3^t = t^{-3}e_3 + 3t^{-3}e_4$ \\
&&& $E_2^t = t^{-2}e_2 + t^{-2}e_3$ & $E_4^t = t^{-5}e_4$\\
\hline

$\D 4{01}\left(\lambda, \alpha,\beta\right)$ &$\to$& $\D 4{02}\left(\lambda, \alpha,\beta\right)$ 
& $E_1^t = t^{-1}e_1$  & $E_3^t = t^{-2}e_3$\\
 &&& $E_2^t = t^{-1}e_2$ & $E_4^t = t^{-3}e_4$\\
\hline

$\D 4{01}\left( \frac{\lambda - t}{(t-1)^2}, t - \af  t ,\af  - \af  t \right)$ &$\to$& $\D 4{03}\left(\lambda, \alpha \right)$ 
& 

\multicolumn{2}{l|}{$E^t_1= \frac{(1 - \af ) \af  (-1 + t)^2 t}{1 + \af ^2 \lambda + t - \af  (1 + t)} e_1 + 
\frac{(-1 + \af ) \af  (-1 + t) t^2}{1 + \af ^2 \lambda + t - \af  (1 + t)} e_2 - 
\frac{(-1 + \af )^2 \af  (-1 + t)^2 t}{ (1 + \af ^2 \lambda + t - \af  (1 + t))^2} e_3$} \\

&&& \multicolumn{2}{l|}{$E^t_2= \frac{(-1 + \af ) \af  (-1 + t) t}{ 1 + \af ^2 \lambda + t - \af  (1 + t)}e_2 - 
\frac{(-1 + \af ) \af ^2 (-1 + t)^2 t}{ (1 + \af ^2 \lambda + t - \af  (1 + t))^2} e_3$ }\\

&&&
\multicolumn{2}{l|}{$E^t_3= \frac{(-1 + \af )^2 \af ^2 (-1 + t)^2 t^2}{(1 + \af ^2 \lambda + t - \af  (1 + t))^2} e_3 + 
\frac{ (-1 + \af )^2 \af ^4 (-1 + t)^4 t^2}{  (1 + \af ^2 \lambda + t - \af  (1 + t))^3} e_4$}\\

&&&\multicolumn{2}{l|}{$E^t_4= -\frac{ (-1 + \af )^3 \af ^3 (-1 + t)^4 t^3}{  (1 + \af ^2 \lambda + t - \af  (1 + t))^3} e_4$ }\\

\hline

$\D 4{01}\left(\lambda, 0,\alpha \right)$ &$\to$& $\D 4{04}\left(\lambda, \alpha\right)$

& $E_1^t = t^{-1}e_1$  & $E_3^t = t^{-2}e_3$\\
 &&& $E_2^t = t^{-1}e_2$ & $E_4^t = t^{-3}e_4$\\
 \hline

$\D 4{01}\left(\lambda, t , \frac{1 + \sqrt{1 - 4 \lambda}}{2 \lambda }\right)$ &$\to$& $\D 4{05}\left(\lambda, \alpha\right)$ 
&

$E^t_1= -\frac{1 + \sqrt{1 - 4 \lambda}}{D} e_1 + 
\frac{(1 + \sqrt{1 - 4 \lambda}) (D-2 \lambda \af )}{ t D^2} e_3$ &
  
$E^t_3= \frac{(1 + \sqrt{1 - 4 \lambda})^2}{  D^2} e_3 - \frac{(1 + \sqrt{1 - 4 \lambda})^4}{ 2 \lambda t D^3} e_4$ \\

\multicolumn{3}{|l|}{$D=1 + \sqrt{1 - 4 \lambda} + 2 \lambda (-1 + \af + \af t)$}&
$E^t_2= -\frac{1 + \sqrt{1 - 4 \lambda}}{  D} e_2 + \frac{(1 + \sqrt{1 - 4 \lambda})^2}{ t D^2} e_3$ &

$E^t_4= -\frac{(1 + \sqrt{1 - 4 \lambda})^3}{  D^3} e_4$
    \\
\hline

$\D 4{01}\left(\lambda, t , \frac{1 -\sqrt{1 - 4 \lambda}}{2 \lambda }\right)$ &$\to$& $\D 4{06}\left(\lambda, \alpha\right)$ 
&

$E^t_1= -\frac{1 - \sqrt{1 - 4 \lambda}}{D} e_1 + 
\frac{(1 - \sqrt{1 - 4 \lambda}) (D-2 \lambda \af )}{ t D^2} e_3$ &

$E^t_3= \frac{(1 - \sqrt{1 - 4 \lambda})^2}{  D^2} e_3 - \frac{(1 - \sqrt{1 - 4 \lambda})^4}{ 2 \lambda t D^3} e_4$ \\

\multicolumn{3}{|l|}{$D=1 - \sqrt{1 - 4 \lambda} + 2 \lambda (-1 + \af + \af t)$}&
$E^t_2= -\frac{1 - \sqrt{1 - 4 \lambda}}{  D} e_2 + \frac{(1 - \sqrt{1 - 4 \lambda})^2}{ t D^2} e_3$ &

$E^t_4= -\frac{(1 - \sqrt{1 - 4 \lambda})^3}{  D^3} e_4$   
    \\
\hline

$\D 4{01}\left(\lambda, t , \frac{1 + \sqrt{1 - 4 \lambda}}{2 \lambda }\right)$ &$\to$& $\D 4{07}\left(\lambda\right)$ 
&

\multicolumn{2}{l|}{$E^t_1=-\frac{D  - (1 + \sqrt{1 - 4 \lambda}) \lambda}{D \lambda (1 + t)} e_1 + 
\frac{2 (D -(1 + \sqrt{1 - 4 \lambda}) \lambda) (D  -2 (1 + \sqrt{1 - 4 \lambda})\lambda + 2 \lambda^2)}{D^3 \lambda (1 + t)^2} e_3$} \\
  
\multicolumn{3}{|l|}{$D=1 + \sqrt{1 - 4 \lambda} - 2 \lambda$}&
\multicolumn{2}{l|}{$E^t_2= -\frac{D  - (1 + \sqrt{1 - 4 \lambda}) \lambda}{D \lambda (1 +t)}e_2$}  \\

&&&
$E^t_3= \frac{(D  - (1 + \sqrt{1 - 4\lambda}) \lambda)^2}{D^2 \lambda^2 (1 + t)^2} e_3$ &
$E^t_4= \frac{(D - (1 + \sqrt{1 - 4 \lambda}) \lambda)^3}{ D^3 \lambda^3 (1 + t)^3} e_4$\\
\hline
 
$\D 4{01}\left(\lambda, t , \frac{1 - \sqrt{1 - 4 \lambda}}{2 \lambda }\right)$ &$\to$& $\D 4{08}\left(\lambda\right)$ 
&

\multicolumn{2}{l|}{$E^t_1=-\frac{D  - (1 - \sqrt{1 - 4\lambda}) \lambda}{D\lambda (1 + t)} e_1 + 
\frac{2 (D - (1 - \sqrt{1 - 4\lambda}) \lambda) (D  - 2 (1 - \sqrt{1 - 4\lambda}) \lambda + 2\lambda^2)}{D^3\lambda (1 + t)^2} e_3$} \\
  
\multicolumn{3}{|l|}{$D=1 - \sqrt{1 - 4\lambda} - 2\lambda$}&
\multicolumn{2}{l|}{$E^t_2= -\frac{D - (1- \sqrt{1 - 4\lambda})\lambda}{D\lambda (1 +t)}e_2$}  \\

&&&
$E^t_3= \frac{(D  - (1 - \sqrt{1 - 4\lambda})\lambda)^2}{D^2\lambda^2 (1 + t)^2} e_3$ &
$E^t_4=- \frac{(D - (1 - \sqrt{1 - 4\lambda})\lambda)^3}{ D^3\lambda^3 (1 + t)^3} e_4$\\
\hline
    
$\D 4{01}\left( \frac{\lambda  - t}{(t-1)^2}, \af  - t^2,  t - t^2 \right)$ &$\to$& $\D 4{09}\left(\lambda, \af \right)$ 
&

\multicolumn{2}{l|}{$E^t_1=(1-t)^2 t ( t^2-\af) D  e_1  + (1 -t) t^2 ( t^2-\af) D  e_2  + ( t-1 )^3
   t (\af  - t^2) D^2 e_3$  } \\
  
\multicolumn{3}{|l|}{$D=(-1 + \af  (-1 + t) + t - \lambda  t^2)^{-1}$}&
\multicolumn{2}{l|}{$E^t_2= (1 - t) t (-\af  + t^2) D e_2  - (-1 + t)^2 t^2 (-\af  + t^2) D^2 e_3$ }  \\

&&&
\multicolumn{2}{l|}{$E^t_3=(1 - t)^2 t^2 (\af  - t^2)^2 D^2 e_3 + (1 - t)^4 t^4 (\af  - t^2)^2 D^3 e_4$ } \\\

&&&
\multicolumn{2}{l|}{$E^t_4=( 1-t )^4 ( t^3-\af  t)^3 D^3 e_4$}\\
\hline

$\D 4{09}\left(\lambda, \alpha \right)$ &$\to$& $\D 4{10}\left(\lambda, \alpha \right)$ 
& $E_1^t = t^{-1}e_1$  & $E_3^t = t^{-2}e_3$\\
 &&& $E_2^t = t^{-1}e_2$ & $E_4^t = t^{-3}e_4$\\
\hline

$\D 4{01}\left( \frac{\lambda  - t}{(t-1)^2},  t^2-1,  t^2 - t\right)$ &$\to$& $\D 4{11}\left(\lambda, \af \right)$ 
&

\multicolumn{2}{l|}{$E^t_1=(1 - t)^3 (1 + t) D  e_1 + (1 - t)^2 t (1 + t) D e_2 - \af (1 - t)^3 (1 + t)^2 D^2t^{-1}  e_3$ } \\
  
\multicolumn{3}{|l|}{$D=(1 + \af \lambda t - t^2)^{-1}$}&
\multicolumn{2}{l|}{$E^t_2= (1 - t)^2 (1 + t) D  e_2 -  \af (1 - t)^3 (1 + t) D^2 e_3$  }  \\

&&&
\multicolumn{2}{l|}{$E^t_3= (1 - t)^4 (1 + t)^2 D^2 e_3 +  \af (1 - t)^6 t (1 + t)^2 D^3 e_4 $} \\\

&&&
\multicolumn{2}{l|}{$E^t_4=(1 - t)^7 (1 + t)^3 D^3 e_4$}\\
\hline

$\D 4{01}\left(\lb\left(1-\frac {\af t}{\Theta^2}\right), \frac 1\Theta\left(\af+\frac{\lb}{t}\right), \frac 1t\right)$ &$\to$& $\D 4{12}\left(\lb,\af\right)$
& $E_1^t = -\af\Theta t^{-1}e_1 + \af\lb t^{-1}e_2 + \af(\Theta - \af)t^{-1}e_3$  & $E_3^t = -\af^3 t^{-1}e_3$\\
 &&& $E_2^t = -\af t^{-1}e_1 + \af(\af t + \lb)(\Theta t)^{-1}e_2 + \af t^{-1}e_3$ & $E_4^t = \af^4 t^{-2}e_4$\\
\hline

$\D 4{01}\left(\lb\left(1-\frac {\af t}{\Psi^2}\right), \frac 1{\Psi}\left(\af+\frac{\lb}{t}\right), \frac 1t\right)$ &$\to$& $\D 4{13}\left(\lb,\af\right)$ 
& $E_1^t = -\af\Psi t^{-1}e_1 + \af\lb t^{-1}e_2 + \af(\Psi - \af)t^{-1}e_3$  & $E_3^t = -\af^3 t^{-1}e_3$\\
\multicolumn{3}{|l|}{$\Psi=1-\Theta$}&
 $E_2^t = -\af t^{-1}e_1 + \af(\af t + \lb)(\Psi t)^{-1}e_2 + \af t^{-1}e_3$ & $E_4^t = \af^4 t^{-2}e_4$\\
\hline

$\D 4{12}\left(\lb,\af\right)$ &$\to$& $\D 4{14}\left(\lb,\af\right)$ 
& $E_1^t = t^{-1}e_1$  & $E_3^t = t^{-2}e_3$\\
 &&& $E_2^t = t^{-1}e_2$ & $E_4^t = t^{-3}e_4$\\
\hline

$\D 4{13}\left(\lb,\af\right)$ &$\to$& $\D 4{15}\left(\lb,\af\right)$ 
& $E_1^t = t^{-1}e_1$  & $E_3^t = t^{-2}e_3$\\
 &&& $E_2^t = t^{-1}e_2$ & $E_4^t = t^{-3}e_4$\\
\hline

$\D 4{01}\left(t^2, \af , t^{-1}\right)$ &$\to$& $\D 4{16}\left(\af\right)$ 
&

$E^t_1= -\af D e_1 + \af(\af t - 1) D^2 e_3$ &
 
$E^t_3= -\af^3 t D^2 e_3 -\af^3 D^3 e_4$ \\

\multicolumn{3}{|l|}{$D=((\af^2 + \af + 1)t - \af  - 1)^{-1}$}&
$E^t_2= -\af D e_1 + \af^2t D e_2 + \af(\af t - 1) D^2 e_3$ &

$E^t_4= \af^4 t D^3 e_4$\\
\hline

$\D 4{01}\left(t^2, -1 , t^{-1}\right)$ &$\to$& $\D 4{17}\left(\af\right)$ 
&

$E^t_1= D e_1 + ((\af - 2) t + \af)D^2 e_3$ &
 
$E^t_3= tD^2 e_3 + \af D^3 e_4$ \\

\multicolumn{3}{|l|}{$D=((\af + 1)t)^{-1}$}&
$E^t_2= D e_1 + tD e_2 + ((\af - 1)t + \af )D^2 e_3$ &

$E^t_4= tD^3 e_4$\\
\hline

$\D 4{01}\left(\lb\left(1+\frac t{\Theta^2}\right), \af + \Psi, 1\right)$ &$\to$& $\D 4{18}\left(\lb,\af\right)$ 
&
\multicolumn{2}{l|}{$E^t_1= t\Theta^3D(\af + \Psi)\left( e_1 - \Psi e_2 - \Theta(\Theta^2\af + \af t + \lb) D e_3\right)$}\\

\multicolumn{3}{|l|}{$D=(\af t^2+\Theta\af(\af + \Theta) t+\Theta^2(\af  + 1)(\Theta\af + \Psi))^{-1}$}&

\multicolumn{2}{l|}{$E^t_2= t\Theta(\af + \Psi)D\left(\Theta e_1+(t-\lb)e_2 - (t^2+(\af  - \Psi)\Theta t+\Theta(\Theta\af + \Psi)\Theta^2) D e_3\right)$}\\

\multicolumn{3}{|l|}{$\Psi=1-\Theta$}&

\multicolumn{2}{l|}{$E^t_3= t^3\Theta^4(\af + \Psi)^2D^2\left(e_3 -\Theta(\Theta\af  + t) D e_4\right)$} \\

&&&

\multicolumn{2}{l|}{$E^t_4= t^4\Theta^6(\af + \Psi)^3D^3e_4$}\\
\hline

$\D 4{01}\left(\lb\left(1+\frac t{\Psi^2}\right), \af + \Theta, 1\right)$ &$\to$& $\D 4{19}\left(\lb,\af\right)$ 
&
\multicolumn{2}{l|}{$E^t_1= t\Psi^3D(\af + \Theta)\left( e_1 - \Theta e_2 - \Psi(\Psi^2\af + \af t + \lb) D e_3\right)$}\\

\multicolumn{3}{|l|}{$D=(\af t^2+\Psi\af(\af + \Psi) t+\Psi^2(\af  + 1)(\Psi\af + \Theta))^{-1}$}&

\multicolumn{2}{l|}{$E^t_2= t\Psi(\af + \Theta)D\left(\Psi e_1+(t-\lb)e_2 - (t^2+(\af - \Theta)\Psi t+\Psi(\Psi\af + \Theta)\Psi^2) D e_3\right)$}\\

\multicolumn{3}{|l|}{$\Psi=1-\Theta$}&

\multicolumn{2}{l|}{$E^t_3= t^3\Psi^4(\af + \Theta)^2D^2\left(e_3 -\Psi(\Psi\af  + t) D e_4\right)$} \\

&&&

\multicolumn{2}{l|}{$E^t_4= t^4\Psi^6(\af + \Theta)^3D^3e_4$}\\
\hline

$\D 4{18}\left(\lb,\af\right)$ &$\to$& $\D 4{20}\left(\lb,\af\right)$ 
& $E_1^t = t^{-1}e_1$  & $E_3^t = t^{-2}e_3$\\
 &&& $E_2^t = t^{-1}e_2$ & $E_4^t = t^{-3}e_4$\\
\hline

$\D 4{19}\left(\lb,\af\right)$ &$\to$& $\D 4{21}\left(\lb,\af\right)$ 
& $E_1^t = t^{-1}e_1$  & $E_3^t = t^{-2}e_3$\\
 &&& $E_2^t = t^{-1}e_2$ & $E_4^t = t^{-3}e_4$\\
\hline

$\D 4{01}\left(\lb\left(1+\frac t{\Theta^2}\right), -\frac{\lb^2}{\Theta^3}, 1\right)$ &$\to$& $\D 4{22}\left(\lb\right)$ 
&
\multicolumn{2}{l|}{$E^t_1= \Theta^2D(\Theta e_1 - \lb e_2 - \Theta\lb D e_3)$}\\

\multicolumn{3}{|l|}{$D=(\Psi(t-2\lb+3\Theta-1))^{-1}$}&

\multicolumn{2}{l|}{$E^t_2= \Theta D\left(\Theta e_1 + (t-\lb)e_2 - \Theta^3\lb^{-1}((2\Theta - 1)t - \Theta (2\lb + 1)  + 1)  D e_3\right)$}\\

\multicolumn{3}{|l|}{$\Psi=1-\Theta$}&

\multicolumn{2}{l|}{$E^t_3= \Theta^4D^2\left(t e_3 - \Theta^3\lb^{-1}(2\Theta - 1)(t - \Psi) D e_4\right)$} \\

&&&

\multicolumn{2}{l|}{$E^t_4= t\Theta^6 D^3 e_4$}
   
    \\
\hline

$\D 4{01}\left(\lb\left(1+\frac t{\Psi^2}\right), -\frac{\lb^2}{\Psi^3}, 1\right)$ &$\to$& $\D 4{23}\left(\lb\right)$ 
&
\multicolumn{2}{l|}{$E^t_1= \Psi^2D(\Psi e_1 - \lb e_2 - \Psi\lb D e_3)$}\\

\multicolumn{3}{|l|}{$D=(\Theta(t-2\lb+3\Psi-1))^{-1}$}&

\multicolumn{2}{l|}{$E^t_2= \Psi D\left(\Psi e_1 + (t-\lb)e_2 - \Psi^3\lb^{-1}((2\Psi - 1)t - \Psi (2\lb + 1)  + 1)  D e_3\right)$}\\

\multicolumn{3}{|l|}{$\Psi=1-\Theta$}&

\multicolumn{2}{l|}{$E^t_3= \Psi^4D^2\left(t e_3 - \Psi^3\lb^{-1}(2\Psi - 1)(t - \Theta) D e_4\right)$} \\

&&&

\multicolumn{2}{l|}{$E^t_4= t\Psi^6 \lb^{-1} D^3 e_4$}
   
    \\
 
\hline

$\D 4{01}\left(\lb\left(1-\frac t{\Theta^2}\right), -\frac{\lb^2}{\Theta^3}, 1\right)$ &$\to$& $\D 4{24}\left(\lb\right)$ 
&

$E^t_1=  -\Theta D\left(\Theta e_1 - \lb e_2 + \Theta^2 D e_3\right)$ &
 
$E^t_3= -t\Theta^2 D^2 e_3$ \\

\multicolumn{3}{|l|}{$D=(t - 2\Theta + 1)^{-1}$}&
$E^t_2= -D\left(\Theta e_1 - (t+\lb) e_2 + \Theta^2 D e_3\right)$ &

$E^t_4= t\Theta^3 D^3 e_4$ \\
\hline

$\D 4{01}\left(\lb\left(1-\frac t{\Psi^2}\right), -\frac{\lb^2}{\Psi^3}, 1\right)$ &$\to$& $\D 4{25}\left(\lb\right)$ 
&

$E^t_1=  -\Psi D\left(\Psi e_1 - \lb e_2 + \Psi^2 D e_3\right)$ &
  
$E^t_3= -t\Psi^2 D^2 e_3$ \\

\multicolumn{3}{|l|}{$D=(t - 2\Psi + 1)^{-1},\ \Psi=1-\Theta$}&
$E^t_2= -D\left(\Psi e_1 - (t+\lb) e_2 + \Psi^2 D e_3\right)$ &

$E^t_4= t\Psi^3\lb^{-1} D^3 e_4$\\
\hline

$\D 4{01}\left(\lb\left(1+\frac t{\Theta^2}\right), -\Theta, 1\right)$ &$\to$& $\D 4{26}\left(\lb\right)$ 
& $E_1^t = \Theta t^{-1}\left(\Theta e_1 - \lb e_2 - \Theta t^{-1} e_3\right)$  & $E_3^t = \Theta^2 t^{-1} e_3$\\
 &&& $E_2^t = t^{-1}\left(\Theta e_1 + (t-\lb) e_2 - \Theta^2 t^{-1} e_3\right)$ & $E_4^t = \Theta^3 t^{-2} e_4$\\
\hline

$\D 4{01}\left(\lb\left(1+\frac t{\Psi^2}\right), -\Psi, 1\right)$ &$\to$& $\D 4{27}\left(\lb\right)$
& $E_1^t = \Psi t^{-1}\left(\Psi e_1 - \lb e_2 - \Psi t^{-1} e_3\right)$  & $E_3^t = \Psi^2 t^{-1} e_3$\\
 &&& $E_2^t = t^{-1}\left(\Psi e_1 + (t-\lb) e_2 - \Psi^2 t^{-1} e_3\right)$ & $E_4^t = \Psi^3 t^{-2} e_4$\\
\hline

$\D 4{01}\left(\lb\left(1+\frac t{\Theta^2}\right), -\Theta, 1\right)$ &$\to$& $\D 4{28}\left(\lb\right)$ 
&

\multicolumn{2}{l|}{$E^t_1=  \Theta^3D\left(\Theta e_1 - \lb e_2 - t^{-1}\Theta^2(\Psi^2t + 4\Theta\lb - \Theta) D e_3\right)$}\\
  
\multicolumn{3}{|l|}{$D=(\Psi t-2\lb+\Theta)^{-1},\ \Psi=1-\Theta$}&

\multicolumn{2}{l|}{$E^t_2= \Theta^2D\left(\Theta e_1 + (t-\lb) e_2 + t^{-1}\Theta((1 - 2\Theta)t^2 + \Theta(3\lb-2\Psi)t -\Theta^2(4\lb-1)) D e_3\right)$}\\

&&&

\multicolumn{2}{l|}{$E^t_3=  \Theta^6 D^2\left(t e_3 + (2\lb-\Theta)(t - \Theta) D e_4\right)$} \\

&&&

\multicolumn{2}{l|}{$E^t_4= t\Theta^9 D^3 e_4$}\\
\hline

$\D 4{01}\left(\lb\left(1+\frac t{\Psi^2}\right), -\Psi, 1\right)$ &$\to$& $\D 4{29}\left(\lb\right)$ 
&

\multicolumn{2}{l|}{$E^t_1=  \Psi^3D\left(\Psi e_1 - \lb e_2 - t^{-1}\Psi^2(\Theta^2t + 4\Psi\lb - \Psi) D e_3\right)$}\\
  
\multicolumn{3}{|l|}{$D=(\Theta t-2\lb+\Psi)^{-1},\ \Psi=1-\Theta$}&

\multicolumn{2}{l|}{$E^t_2= \Psi^2D\left(\Psi e_1 + (t-\lb) e_2 + t^{-1}\Psi((1 - 2\Psi)t^2 + \Psi(3\lb-2\Theta)t -\Psi^2(4\lb-1)) D e_3\right)$}\\

&&&

\multicolumn{2}{l|}{$E^t_3=  \Psi^6 D^2\left(t e_3 + (2\lb-\Psi)(t - \Psi) D e_4\right)$} \\

&&&

\multicolumn{2}{l|}{$E^t_4= t\Psi^9 D^3 e_4$}\\
\hline

$\D 4{01}\left(\lb\left(1-\left(2+\frac{2\lb-1}{\Theta^2}\right)t\right), \frac\lb{\Theta^2}\left(1-\frac\Theta t\right), \frac 1t\right)$ &$\to$& $\D 4{30}\left(\lb\right)$ 
&

\multicolumn{2}{l|}{$E^t_1= -D\left(\Theta e_1 - \lb e_2 + t\Psi^{-1}D e_3\right)$}\\
  
\multicolumn{3}{|l|}{$D=\frac{\lb(t-\Theta)}{t\left((1-2\Theta)t+\Theta^2\right)},\ \Psi=1-\Theta$}&

\multicolumn{2}{l|}{$E^t_2= -D\left(e_1 - \Theta^{-1}((2\Theta-1)t + \lb) e_2 + t\Psi^{-1}(2t-\Theta)(t-\Theta)^{-1} D e_3\right)$}\\

&&&

\multicolumn{2}{l|}{$E^t_3= tD^2\left((1-2\Theta)e_3 + \Theta^2\lb^{-1}(t-\Theta)^{-1}D e_4\right)$} \\

&&&

\multicolumn{2}{l|}{$E^t_4= t D^3(2\Theta - 1) e_4$}

    \\
\hline

$\D 4{01}\left(\lb\left(1-\left(2+\frac{2\lb-1}{\Psi^2}\right)t\right), \frac\lb{\Psi^2}\left(1-\frac\Psi t\right), \frac 1t\right)$ &$\to$& $\D 4{31}\left(\lb\right)$ 
&

\multicolumn{2}{l|}{$E^t_1= -D\left(\Psi e_1 - \lb e_2 + t\Theta^{-1}D e_3\right)$}\\
  
\multicolumn{3}{|l|}{$D=\frac{\lb(t-\Psi)}{t\left((1-2\Psi)t+\Psi^2\right)},\ \Psi=1-\Theta$}&

\multicolumn{2}{l|}{$E^t_2= -D\left(e_1 - \Psi^{-1}((2\Psi-1)t + \lb) e_2 + t\Theta^{-1}(2t-\Psi)(t-\Psi)^{-1} D e_3\right)$}\\

&&&

\multicolumn{2}{l|}{$E^t_3= tD^2\left((1-2\Psi)e_3 + \Psi^2\lb^{-1}(t-\Psi)^{-1}D e_4\right)$} \\

&&&

\multicolumn{2}{l|}{$E^t_4= t D^3(2\Psi - 1) e_4$}

    \\
\hline

$\D 4{30}\left(\lb\right)$ &$\to$& $\D 4{32}\left(\lb\right)$ 
& $E_1^t = t^{-1}e_1$  & $E_3^t = t^{-2}e_3$\\
 &&& $E_2^t = t^{-1}e_2$ & $E_4^t = t^{-3}e_4$\\
\hline

$\D 4{31}\left(\lb\right)$ &$\to$& $\D 4{33}\left(\lb\right)$ 
& $E_1^t = t^{-1}e_1$  & $E_3^t = t^{-2}e_3$\\
 &&& $E_2^t = t^{-1}e_2$ & $E_4^t = t^{-3}e_4$\\
\hline

$\D 4{01}\left(\frac t{(t + 1)2}, 1 , 1 + \frac 1t\right)$ &$\to$& $\D 4{34}$
&

\multicolumn{2}{l|}{$E^t_1= -(t + 1)^2 D e_1 + (t + 1)t D e_2 + (t + 1)^2(t - 1)D^2 e_3$}\\
  
\multicolumn{3}{|l|}{$D=(t^2 + t - 1)^{-1}$}&

\multicolumn{2}{l|}{$E^t_2= -(t + 1)^2 D e_1 + (t + 1)^2t D e_2 + (t + 1)^2 D e_3$}\\

&&&
\multicolumn{2}{l|}{$E^t_3=  -(t + 1)^2t^2 D^2 e_3 - (t + 1)^4t D^3 e_4$} \\

&&&

\multicolumn{2}{l|}{$E^t_4= (t + 1)^4t^2 D^3 e_4$}
    \\
\hline

$\D 4{01}\left(\lb, t^{-1}, 1\right)$ &$\to$& $\D 4{35}\left(\lambda\right)$ 
& $E_1^t = t^{-1}e_1$  & $E_3^t = t^{-2}e_3 + (\lb t^2)^{-1}e_4$\\
 &&& $E_2^t = t^{-1}e_2 + (\lb t)^{-1}e_3$ & $E_4^t = t^{-4}e_4$\\
\hline

$\D 4{35}\left(\lambda\right)$ &$\to$& $\D 4{36}\left(\lambda\right)$ 
& $E_1^t = t^{-1}e_1$  & $E_3^t = t^{-2}e_3$\\
 &&& $E_2^t = t^{-1}e_2$ & $E_4^t = t^{-3}e_4$\\
\hline

$\D 4{01}\left(\lb,\Theta t^{-1},t^{-1}\right)$ &$\to$& $\D 4{37}\left(\lb\right)$ 
& $E_1^t = te_1$  & $E_3^t = t^2e_3$\\
 &&& $E_2^t = te_2$ & $E_4^t = t^2e_4$\\
\hline

$\D 4{01}\left(\lb,(1-\Theta) t^{-1},t^{-1}\right)$ &$\to$& $\D 4{38}\left(\lb\right)$ 
& $E_1^t = te_1$  & $E_3^t = t^2e_3$\\
 &&& $E_2^t = te_2$ & $E_4^t = t^2e_4$\\
\hline

$\D 4{37}\left(\lambda\right)$ &$\to$& $\D 4{39}\left(\lambda\right)$ 
& $E_1^t = t^{-1}e_1$  & $E_3^t = t^{-2}e_3$\\
 &&& $E_2^t = t^{-1}e_2$ & $E_4^t = t^{-3}e_4$\\
\hline

$\D 4{38}\left(\lambda\right)$ &$\to$& $\D 4{40}\left(\lambda\right)$ 
& $E_1^t = t^{-1}e_1$  & $E_3^t = t^{-2}e_3$\\
 &&& $E_2^t = t^{-1}e_2$ & $E_4^t = t^{-3}e_4$\\
\hline

\end{longtable}
}

\end{document}